\documentclass[final,3p,times]{elsarticle}

\usepackage{amssymb}
\usepackage{amsmath}
\usepackage{mathtools}
\usepackage{amsthm}
\usepackage{color}
\usepackage{multirow}
\usepackage{lscape}

\usepackage{subfigure}
\usepackage{float}

\usepackage[svgnames]{xcolor}
\usepackage{soul}

\newtheorem{theorem}{Theorem}
\newtheorem{lemma}{Lemma}

\journal{Journal of Computational Physics}

\begin{document}

\begin{frontmatter}



\title{Scheduled Relaxation Jacobi method: improvements and applications}


\author{J.E.~Adsuara\corref{cor1}\fnref{label1}} 
\author{I.~Cordero-Carri\'on\corref{cor1}\fnref{label2}}
\author{P.~Cerd\'a-Dur\'an\corref{cor1}\fnref{label1}}
\author{M.A.~Aloy\corref{cor1}\fnref{label1}}
\cortext[cor1]{jose.adsuara@uv.es,isabel.cordero@uv.es, pablo.cerda@uv.es, miguel.a.aloy@uv.es}
\address[label1]{Departamento de Astronom\'{\i}a y Astrof\'{\i}sica, Universidad de Valencia, E-46100, Burjassot, Spain.}
\address[label2]{Departamento de Matem\'atica Aplicada, Universidad de Valencia, E-46100, Burjassot, Spain.}

\begin{abstract}
  Elliptic partial differential equations (ePDEs) appear in a wide variety of areas of mathematics, physics and
  engineering. Typically, ePDEs must be solved numerically, which sets an ever growing demand for efficient and highly
  parallel algorithms to tackle their computational solution. The Scheduled Relaxation Jacobi (SRJ) is a promising class
  of methods, atypical for combining simplicity and efficiency, that has been recently introduced for solving linear
  Poisson-like ePDEs. The SRJ methodology relies on computing the appropriate parameters of a multilevel approach with
  the goal of minimizing the number of iterations needed to cut down the residuals below specified tolerances. The
  efficiency in the reduction of the residual increases with the number of levels employed in the algorithm. Applying
  the original methodology to compute the algorithm parameters with more than 5 levels notably hinders obtaining
  optimal SRJ schemes, as the mixed (non-linear) algebraic-differential equations from which they result become
  notably stiff. Here we present a new methodology for obtaining the parameters of SRJ schemes that overcomes the
  limitations of the original algorithm and provide parameters for SRJ schemes with up to 15 levels and resolutions of
  up to $2^{15}$ points per dimension, allowing for acceleration factors larger than several hundreds with respect to
  the Jacobi method for typical resolutions and, in some high resolution cases, close to 1000.  Furthermore, we extend
  the original algorithm to apply it to certain systems of non-linear ePDEs.
\end{abstract}


\end{frontmatter}


\section{Introduction}

Partial differential equations (PDEs) are the appropriate mathematical language for modeling many phenomena
\cite{Evans2010}. In particular, we are interested in elliptic PDEs (ePDEs), that arise when we face the solution of
equilibrium problems, in which the time evolution of the system is either neglected or irrelevant.
Poisson and Laplace equations are prototype second order ePDEs,
with and without source terms respectively.

Though the aforementioned Poisson and Laplace equations posses analytic solutions in a limited number of simple cases,
we usually need a numerical solution when more general problems are considered. One of the standard approaches for
solving these equations numerically is using finite differences methods. In this approach, both functions and operators
are discretized on a numerical mesh, leading to a system of linear algebraic equations, which can be solved with direct
or iterative methods. One of the simplest and most studied iterative schemes is the so called Jacobi method
\cite{Jacobi1845,Richardson11}, whose main drawback is its poor convergence rate.

In oder to improve the efficiency of the Jacobi method, many alternatives have been considered. A popular possibility is
the use of preconditioners \cite{JM60,GJS91,NT06} applied to linear systems, that make the associated Jacobi and
Gauss-Seidel methods converge asymptotically faster than the unpreconditioned ones. Indeed, the method we improve on
here, can be adapted as a preconditioner for other methods (e.g., the conjugate gradient method). Very widespread is the use of
multigrid methods \cite[e.g.,\,][]{TOS01} that, in many cases, provide the solution with ${\cal O}(N)$ operations, or
that can even be employed as preconditioners. Relaxation algorithms \citep[originally introduced in][]{Richardson11},
improve the performance of the Jacobi method by considering modifications of the Gauss-Seidel algorithm that include a
weight, for instance, successive overrelaxation (SOR) methods \cite{Young1954}.

Along this line, \cite[][YM14 henceforth]{Yang2014} have recently presented a significant acceleration (of the order of
100) over the Jacobi algorithm, employing the Scheduled Relaxation Jacobi (SRJ) method.  The SRJ method is a
generalization of the weighted Jacobi method which adds an overrelaxation factor to the classical Jacobi in a similar
fashion to the SOR. This generalization includes a number $P$ of different levels, in each of which, the overrelaxation
(or underrelaxation) parameter or weight is tuned to achieve a significant reduction of the number of iterations, thus
leading to a faster convergence rate.  The optimal set of weights depends on the actual discretization of the problem at
hand. Although the method greatly improves the convergence rate with respect to the original Jacobi, the schemes
presented by YM14, optimal up to $P=5$ and resolutions of up to 512 points per spatial dimension, are still not
competitive with other methods used currently in the field (e.g., spectral methods \cite{BW86}, or multigrid methods as
commented above). The main advantage of the SRJ method over other alternatives to solve numerically ePDEs is its
simplicity and the straightforward parallelization, since SRJ methods preserve the insensitivity of the Jacobi method to
domain decomposition (in contrast, e.g., to multigrid methods).

Following basically the same procedure as in YM14, \cite[][ACCA15 henceforth]{Adsuetal15} has obtained optimal SRJ
algorithms with up to $P=10$ levels and multiple numerical resolutions. However, the limitations of the methodology of
YM14 to compute optimal parameters for multilevel SRJ schemes prevents to develop algorithms with more than 10
levels. In this paper, we will show a new methodology to evaluate the parameters of optimal SRJ schemes with up to
$P=15$ levels and resolutions of up to $2^{15}$ points per spatial dimension of the problem, which in some cases may
yield accelerations of order $10^3$ with respect to the Jacobi method. Considering the difficulty in obtaining the SRJ
parameters for $P>5$, we provide the readers with a comprehensive set of tables for different SRJ schemes and different
resolutions.

We begin the paper giving an overview of the SRJ method (Sect.~\ref{sec:SRJ}) and describing the original methodology
for obtaining optimal schemes together with the improvements on them already made in ACCA15
(Sect.~\ref{sec:optimalPars}).  Then, we will present in Sect.~\ref{subsec:analyticalManipulations}, some analytical
work which reduces the number of unknowns to solve for to ${\cal O}(P)$ (instead of ${\cal O}(P^2)$ as in YM14 and
ACCA15). In Sect.~\ref{sec:results} we show a comparison of the new method to compute optimal parameters for SRJ schemes
with that of YM14. Furthermore, we test the SRJ methods in a case study, namely a Poisson equation with Dirichlet
boundary conditions (Sect.~\ref{subsec:casestudy}) that has analytic solution, and show that optimal SRJ parameters
computed for resolutions close to that of the problem at hand can bring two orders of magnitude smaller number of
iterations than the Jacobi method to solve such the problem. We have also assessed the performance of the new SRJ
schemes with a large number of sublevels with respect to other standard methods to solve ePDEs
(Sect.~\ref{subsec:spherical}). In particular, we compare SRJ schemes with $P=6$ and $P=15$ to direct inversion methods
and to spectral methods implemented in the {\tt LAPACK} and {\tt LORENE} packages, respectively. We outline the most
prominent conclusions of our study and discuss the limitations of the current methodology in
Sect.~\ref{sec:conclusions}.

\section{SRJ schemes}
\label{sec:SRJ}

In this section we recap the most salient results obtained by YM14 and set the notation for the rest of the
paper. 

First of all, if we define $\omega_iJ$ as a single step in a weighted Jacobi iteration using the weight $\omega_i$
($i=1,\ldots,P$), then the SRJ method can be cast as a successive application of elementary relaxation steps of the form
\begin{eqnarray*}
\overbrace{ \overbrace{ \omega_1 J \ldots \omega_1 J}^{q_1} \underbrace{ \omega_2 J \ldots \omega_2 J}_{q_2} \ldots
  \overbrace{ \omega_P J \ldots \omega_P J}^{q_P} }^M \overbrace{ \omega_1 J \ldots \omega_1 J}^{q_1} \ldots \, ,
\end{eqnarray*}
where the largest weight, $\omega_1$ is applied $q_1$ times, and each of the remaining and progressively smaller weights
$\omega_i$ ($i=1,\ldots, P$) is applied $q_i$ times, respectively. A single cycle of the scheme ends after $M$
elementary steps, where $M:=\sum_{i=1}^P q_i$. In order to reach a prescribed tolerance goal, we need to repeat a number
times the basic $M$-cycle of the SRJ method. Both, a vector of weights and a vector with the number of times we use each
weight, define each optimal scheme. We emphasize that, from the point of view of the implementation, the only difference
with the traditional weighted Jacobi is that, instead of having a fixed weight, SRJ schemes of $P$-levels require the
computation of $P$ weights.

In order to simplify the notation, we define $\boldsymbol{\omega} := (\omega_1,\ldots,\omega_P)$, with $\omega_{i} >
\omega_{i+1}$ and $\boldsymbol{q} := (q_1,\ldots,q_P)$ which is in one-to-one correspondence with the previous
$\boldsymbol{\omega}$. Also, we define $\boldsymbol{\beta} := (\beta_1,\ldots,\beta_P)$, where $\beta_i = q_i/ M$ is
the fraction of the iteration counts that a given weight $\omega_i$ is repeated in an $M$-cycle, with $\beta_P := 1 -
\sum_{i=1}^{P-1}\beta_i$.

The basic idea of the SRJ schemes is finding the optimal values for $\boldsymbol{\omega}$ and $\boldsymbol{\beta} $ that
minimize the total number of iterations to reach a prescribed tolerance for a given number of points (i.e., numerical
resolution) $N$.

\section{Finding the optimal parameters}
\label{sec:optimalPars}

Below we explain how to compute the optimal values of $\boldsymbol{\omega}$ and $\boldsymbol{\beta}$, following the
prescription of YM14 and rewrite some parts of the YM14 algorithm to make them amenable for extension to a larger number
of levels and resolutions.

\subsection{Converge analysis and optimization problem}
\label{sec:convergeAnalysis}

We perform a convergence analysis of the method in order to obtain a number of restrictions that the parameters of the
SRJ scheme must fulfill. As a model problem, we use the Laplace equation with homogeneous Neumann boundary conditions in
two spatial dimensions, in Cartesian coordinates and over a domain with unitary size:
\begin{equation}
\begin{cases}
\displaystyle\frac{\partial^2}{\partial x^2} u(x,y) + \frac{\partial^2}{\partial y^2} u(x,y) = 0, & (x,y) \in (0,1) \times (0,1) \\*[0.3cm]
\displaystyle\left.\frac{\partial}{\partial x}\right|_{x=0} u(x,y) = 0,\,  \left.\frac{\partial}{\partial x}\right|_{x=1} u(x,y)  = 0, & y \in (0,1) \\*[0.4cm]
\displaystyle\left.\frac{\partial}{\partial y}\right|_{y=0} u(x,y) = 0,\,  \left.\frac{\partial}{\partial y}\right|_{y=1} u(x,y)  = 0, & x \in (0,1). \\
\end{cases}
\label{eq:Laplace}
\end{equation}

Let us consider a 2nd-order central-difference discretization of Eq.\,(\ref{eq:Laplace}) on a uniform grid consisting of $N_x\times N_y$ zones, and define $N={\rm max}(N_x,N_y)$. Then, we apply the Jacobi method with a relaxation parameter $\omega$, so that the following iterative scheme results:
\begin{eqnarray}
u^{n+1}_{i,j} & = &  (1-\omega)u^n_{i,j}+\frac{\omega}{4}(u^n_{i,j-1}+u^n_{i,j+1}+u^n_{i-1,j}+u^n_{i+1,j}) \\
& = &  u^n_{i,j}+\frac{\omega}{4}(u^n_{i,j-1}+u^n_{i,j+1}+u^n_{i-1,j}+u^n_{i+1,j} - 4u^n_{ij}),
\label{eq:2nd-order_CD}
\end{eqnarray}
where $n$ is the index of iteration. 

At this point, we perform a von Neumann stability analysis for obtaining the amplification factor,
\begin{equation}
G_{\omega}(\kappa) = (1- \omega \kappa), 
\end{equation}
where $\kappa$ is a function of the wave-numbers in each dimension. For the problem at hand, 
\begin{equation}
\kappa(k_x,k_y) = \sin^2 \left(\frac{k_x}{2N_x}\right) + \sin^2 \left(\frac{k_y}{2N_y}\right)\, .
\end{equation}

$G_\omega$ expresses how much the error can grow up from one iteration to the next one using the relaxation Jacobi
method. Thus, if a single relaxation step is performed, we require $|G_\omega|<1$ to ensure convergence. However, in an
SRJ scheme, we perform a series of $M$-cycles (Sect.~\ref{sec:SRJ}). Hence, even if on an elementary step of the
algorithm one may violate the condition $|G_\omega|<1$ (which may happen, e.g., if such step is an overrelaxation of the
Jacobi method), the condition for convergence shall be obtained for the composition of $M$ elementary amplification
factors. As we apply Eq.~(\ref{eq:2nd-order_CD}) $M$-times but with $P$ different weights $\omega_i$, the following
composition of amplifications factors is obtained:
\begin{eqnarray}
\overbrace{ \overbrace{ G_{\omega_1} \ldots G_{\omega_1} }^{q_1} \underbrace{ G_{\omega_2} \ldots G_{\omega_2}}_{q_2}
  \ldots \overbrace{ G_{\omega_P} \ldots G_{\omega_P}}^{q_P} }^M = \prod_{i=1}^{P} G_{\omega_i}^{\phantom{\omega_i}q_i}.
\label{eq:cycleAmplFact}
\end{eqnarray}

Finally, it is not important how many times we use each of the weights $q_i$ but their relative frequency of use during
an $M$-cycle, which is defined by $\beta_i$. Therefore, following YM14, one can define the per-iteration amplification
factor function as a geometric mean of the modulus of the cycle amplification factor (Eq.\,\ref{eq:cycleAmplFact}):
\begin{equation}
\Gamma(\kappa) = \prod_{i=1}^{P} |1-\omega_i \kappa|^{\beta_i} .
\label{eq:Gamma}
\end{equation}

The previous transformation is very convenient to find deterministic optimal parameters for the SRJ schemes, since it
avoids working with a Diophantine equation (Eq.\,\ref{eq:cycleAmplFact}), because $q_i \in \mathbb{N}$, while $\beta_i \in \mathbb{R}$.

From the definition of $\Gamma(\kappa)$, it is evident that larger values of the per-iteration amplification factor
yield a larger number of iterations for the algorithm to converge. Thus, the optimal values for the SRJ parameters are
obtained by looking for the extrema of $\Gamma(\kappa)$ in $[\kappa_m, \kappa_M]$, which is the interval bounding the
allowed values of $\kappa$, namely
\begin{eqnarray}
\kappa_m = \sin^2 \left(\frac{\pi}{2 N}\right),  & \kappa_M = 2\, ,
\label{eq:boundaries}
\end{eqnarray}
and then, to minimize {\em globally} these extrema, so that the error per iteration decreases as much as possible. This
sets our optimization problem.

We explicitly point out that the value of $\kappa_m$ depends on the type of boundary conditions of the problem, on
  the discretization of the elliptic operator and on the dimensionality of the problem. This is not the case for
  $\kappa_M$, which equals 2 independent on the boundary conditions, dimensionality and discretization of the elliptic
  operator. For later reference, we write the explicit form of $\kappa_m$ as a function of the number of dimensions,
  $d$, for a Cartesian discretization of the elliptic operator an Neumann boundary conditions:
\begin{eqnarray}
\kappa_m^{\rm (d)} = \displaystyle\frac{2}{d}\sin^2 \left(\frac{\pi}{2 N^{\rm (d)}}\right) .
\label{eq:boundaries1D3D}
\end{eqnarray}
Obviously, we recover Eq.\,(\ref{eq:boundaries}) setting $d=2$. 
For practical purposes, it is possible to obtain the optimal SRJ parameters for any value of $d$ once we know the
optimal parameters in 2D. This is done by computing the effective number of points in 2D, $N_{\rm eff}^{\rm (2)}$, corresponding to a given
problem size $N^{\rm (d)}$ in $d$-dimensions through the relation:
\begin{eqnarray}
N_{\rm eff}^{\rm (2)} = \frac{\pi}{2\arcsin{ \left(\displaystyle\sqrt{\frac{2}{d}}
      \sin{\left(\frac{\pi}{2N^{\rm (d)}}\right)}\right)} } \simeq N^{\rm (d)}\sqrt{ \frac{d}{2} } ,
\label{eq:effectiveN}
\end{eqnarray}
where the approximated result holds for large values of $N^{\rm (d)}$.

Finally, as stated above, the values of $\kappa_m$ change depending on whether Neumann or Dirichlet boundary conditions
are considered, and so the optimal parameters change. Fortunately, there is a simple way to obtain the optimal
parameters in case of Dirichlet boundary conditions from the 2D optimal parameters computed for Neumann boundaries,
namely
\begin{eqnarray}
N_{\rm eff}^{\rm (2)} = \frac{\pi}{2\arcsin{ \left(\displaystyle\sqrt{\frac{2}{d}
 \sum_{i=1}^d \sin^2{\left(\frac{\pi}{2N^{\rm (d)}_{i,{\rm Dirichlet}}}\right)}}\right)} } \simeq \frac{\displaystyle\sqrt{\frac{d}{2} }}{\sqrt{
    \displaystyle\sum_{i=1}^d {\frac{1}{(N^{\rm (d)}_{i,{\rm Dirichlet}})^{2}}}}} ,
\label{eq:effectiveNDirichlet}
\end{eqnarray}
where the approximated result holds when the number of points in each dimension is sufficiently large.

\subsection{The non-linear system}
\label{subsec:NonLinSys}

In the optimization problem stated in the previous section we need to compute the location of the extrema
$\Gamma(\kappa)$, $\boldsymbol{\kappa}$ hereafter. From the optimization process one must also obtain the rest of the
parameters of the SRJ scheme, namely, $\boldsymbol{\omega}$ and $\boldsymbol{\beta}$. Thus, we
need to solve a system ${\cal S}(\boldsymbol{\omega}, \boldsymbol{\beta}, \boldsymbol{\kappa})$ to determine all these
unknowns.

For the evaluation of $\boldsymbol{\kappa}$, we must take into account that the location of the maxima can be either at
the edges of the domain, $\kappa_0:=\kappa_m$ and $\kappa_P:=\kappa_M$ (set by Eq.\,\ref{eq:boundaries}), or in other
$P-1$ internal values $\kappa_i$ ($i=1,\ldots,P-1$) determined (each of them) by the following condition:
\begin{equation}
 \frac{\partial}{\partial \kappa} \log \Gamma(\kappa)  = \sum_{i=1}^{P} \frac{\beta_i \omega_i}{1 - \kappa \omega_i} = 0.
\label{eq:partialG}
\end{equation}
From the solutions of Eq.\,(\ref{eq:partialG}), we obtain the $P-1$ different $\kappa_i = \kappa_i(\boldsymbol{\omega},
\boldsymbol{\beta})$, which allows us to reduce the number of unknowns of the system ${\cal S}(\boldsymbol{\omega},
\boldsymbol{\beta},\boldsymbol{\kappa})$ from $3P-2$ $(\omega_1,\ldots,\omega_P,\beta_1,\ldots,\beta_{P-1},\kappa_1,\ldots,\kappa_{P-1})$
to $2P-1$. Hence, we need to also obtain $2P-1$ equations to find a unique solution of the system.

For obtaining the set of the first $P$ equations for the system, and following YM14, we equalize all the maxima:
\begin{equation}
\Gamma(\kappa_0) = \Gamma(\kappa_{i}), \,\, i=1,\ldots,P 
\label{eq:equMax}.
\end{equation}
Furthermore, if we assume that $\boldsymbol{\omega}=\boldsymbol{\omega}(\boldsymbol{\beta})$, and therefore 
$\boldsymbol{\kappa}=\boldsymbol{\kappa}(\boldsymbol{\beta})$, a second set of $P-1$ equations can be obtained from the minimization of $\Gamma(\kappa_0)$:
\begin{equation}
\frac{\partial}{\partial \beta_j}\Gamma(\kappa_0)=0, \: j=1,\ldots
P-1.
\label{eq:partialbeta} 
\end{equation}
Thereby, our system is now 
$S(\boldsymbol{\omega}, \boldsymbol{\beta},\frac{\partial \boldsymbol{\omega}}{\partial \boldsymbol{\beta}})$, 
since the differentiation in Eq.\,(\ref{eq:partialbeta}) introduces $P(P-1)$ new ancillary variables, 
namely $\frac{\partial \omega_i}{\partial \beta_j}; \,i=1,\ldots,P,\,j=1,\ldots, P-1$. The final set of
$P(P-1)$ additional equations to account for this extra ancillary variables results from applying the same 
condition as in Eq.\,(\ref{eq:partialbeta}) to the remaining values of $\kappa_i$, deduced from 
Eqs.~(\ref{eq:equMax}) and (\ref{eq:partialbeta}):
\begin{equation}
0=\frac{\partial}{\partial \beta_j}\Gamma(\kappa_i),\:  i=1,\ldots
P,\: j=1\ldots P-1.
\label{eq:equDifBeta}
\end{equation}
We note that ACCA15 have shown that setting the
optimization problem in terms of $\boldsymbol{\omega}$ as reference
variables and considering $\boldsymbol{\beta}$ as function of the
later ($\boldsymbol{\beta}(\boldsymbol{\omega})$) brings the same
numerical solution for the system ${\cal S}(\boldsymbol{\omega},
\boldsymbol{\beta},\boldsymbol{\kappa})$.

\subsection{Basic improvements on the original SRJ algorithm}
\label{subsec:preImpVar}

In order to make the SRJ method competitive with other existing algorithms to solve ePDEs, we must find the optimal
parameters of SRJ schemes with a {\em sufficiently} large number of levels. Furthermore, since the optimal parameters depend on the
resolution of the discretization used to solve a given problem, we also need to compute optimal parameters for a range of
numerical resolutions larger than in YM14.  In this section we summarize some improvements done in ACCA15, which allow us
to solve the system from $P=6$ up to $P=10$ and for resolutions up to $2^{15}$. Additional improvements based on
analytical results will be commented in the following subsection.

Firstly, the stiffness of $S(\boldsymbol{\omega}, \boldsymbol{\beta}, \frac{\partial \boldsymbol{\omega}}{\partial
  \boldsymbol{\beta}})$ increases with the number of levels $P$, which prevented YM14 to compute optimal SRJ schemes for
$P>5$. We have been able to reduce the complexity of the numerical solution by manipulating parts of them
algebraically. On the one hand, we have hidden the part of the non-linear system involving the $\boldsymbol{\kappa}$
unknowns by solving for them symbolically and using those symbolic placeholders later when solving numerically for
$\boldsymbol{\omega}$ and $\boldsymbol{\beta}$.  On the other hand, and after the previous manipulation, we have seen in
Sect.~\ref{subsec:NonLinSys} that we need to solve numerically a non-linear system of $P^2+P-1$ equations with the same
number of unknowns. We aim to rewrite this system $S(\boldsymbol{\omega}, \boldsymbol{\beta}, \frac{\partial
  \boldsymbol{\omega}}{\partial \boldsymbol{\beta}})$ as $S(\boldsymbol{\omega}, \boldsymbol{\beta})$, which requires
obtaining $\frac{\partial \boldsymbol{\omega}}{\partial \boldsymbol{\beta}}$ as a function of $\boldsymbol{\omega}$ and
$\boldsymbol{\beta}$. We also compute the solutions of this linear subsystem symbolically. These manipulations of the
original system of equations permitted ACCA15 to compute the optimal SRJ parameters for a larger number of levels (up to
$P=10$) than in YM14.
  
Secondly, in order to increase the number of points, for the numerical solution of the system, since the accuracy of the
results for large values of $P$ critically depends on having sufficiently large precision, we needed Mathematica's
ability to perform arbitrary precision arithmetic. We ended up using up to twenty four digits in the representation of
the numbers in some cases.
\begin{figure*}[h]
\begin{tabular}{ll}
\hspace{-0.5cm}  \includegraphics[width=0.55\linewidth]{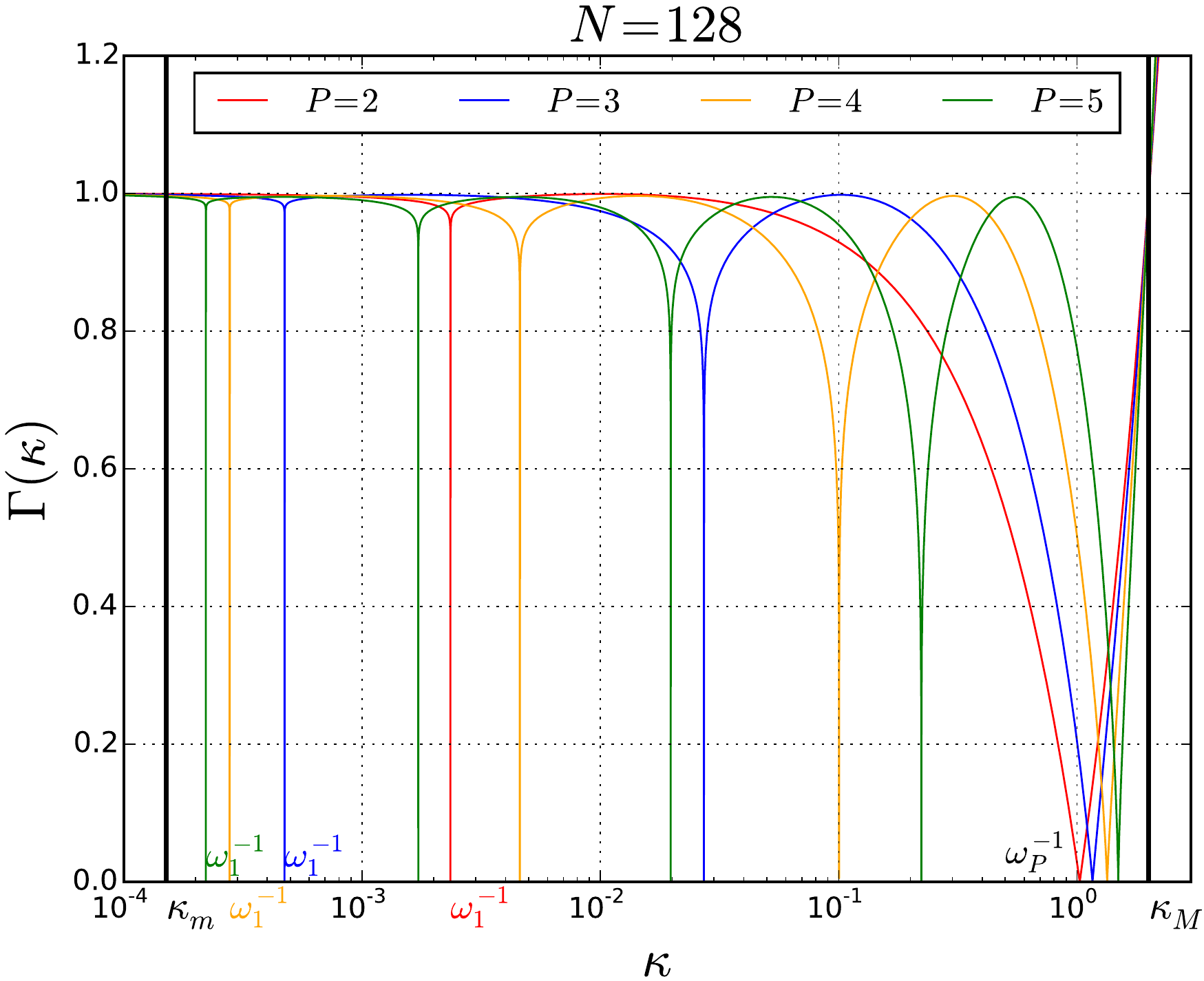} &
\hspace{-0.4cm} \includegraphics[width=0.45\linewidth]{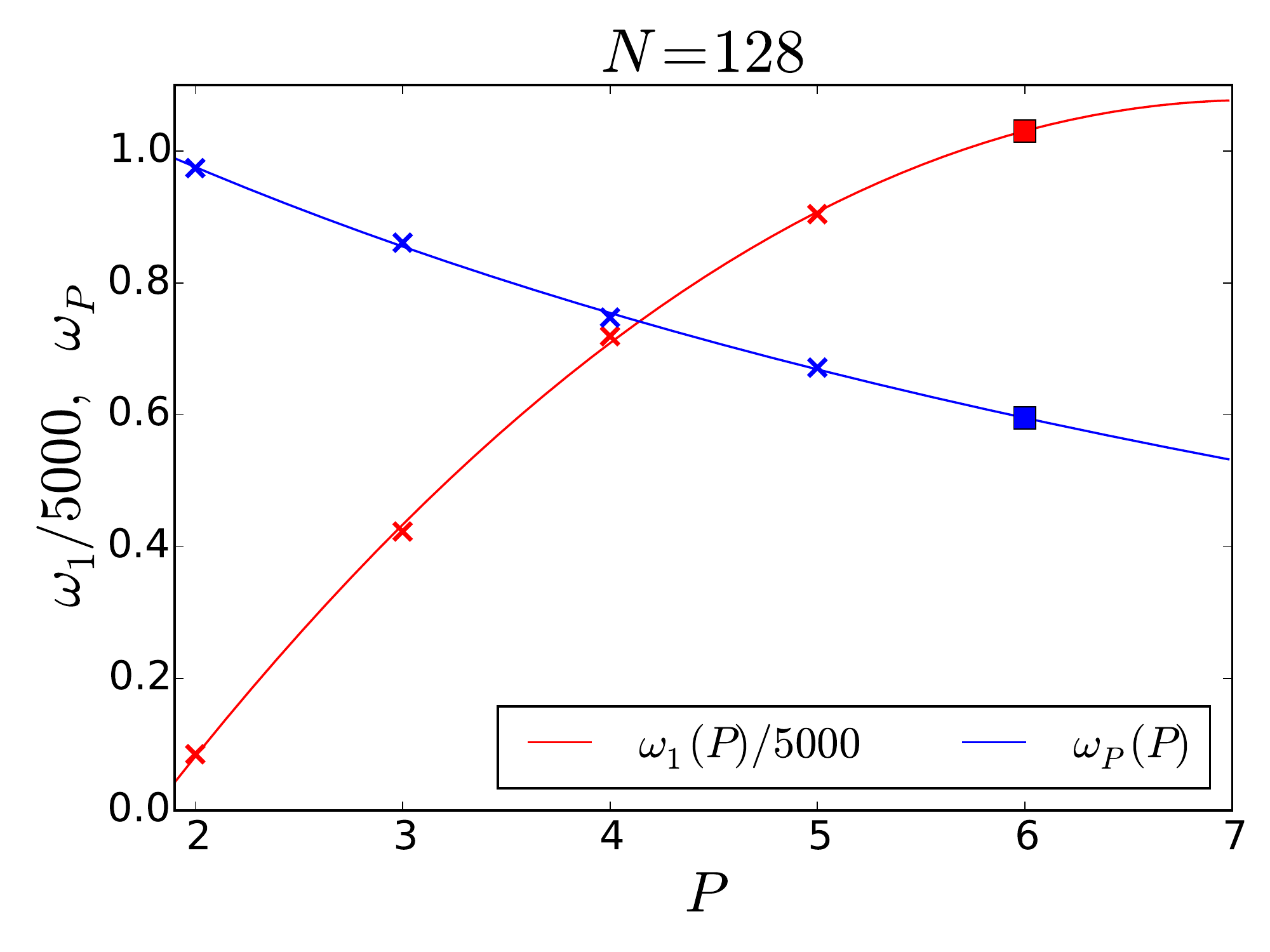}
\end{tabular}
    \caption{{\em Left}: $\Gamma(\kappa)$ functions corresponding to different values of the numbers of levels $P$ for
      $N=128$. The inverse of the minimal weight, $\omega_1^{-1}$, and the inverse of the maximal one, 
			$\omega_P^{-1}$ moves towards $\kappa_m$ and $\kappa_M$, respectively, when $P$ increases. The rest of the weights are
			distributed roughly logarithmically equally spaced inside $(\omega_1, \omega_P)$. In the right panel, we show the 
			discrete values of $\omega_1$ and of $\omega_P$ for the $P=6$, $N=128$ scheme, together with both fits of their 
			respective values as a function of $P$ to conics as described in Sect.~\ref{subsec:preImpVar}.}
    \label{fig:GammaOmegas}
\end{figure*}

Finally, the non-linear system we have to solve is very sensitive to the initial values that we guess for the
unknowns. In ACCA15, we developed a systematic way of setting the initial guesses from the values obtained from lower
levels and now we have improved it. As we can see in Fig.\,\ref{fig:GammaOmegas}, when we increase the number of levels
$P$, the values of $\omega_1^{-1}$ and $\omega_P^{-1}$ move towards $\kappa_m$ and $\kappa_M$, respectively. We can
also observe that the inverse of the rest of the weights of a scheme are roughly logarithmically equally spaced between
the values $\omega_1^{-1}$ and $\omega_P^{-1}$. Hence, we use as approximate location of the initial guesses for the inverse
of any weight for an SRJ scheme of $P$ levels the following expressions:
\begin{gather}
\{\omega_1^{-1}, \omega_1^{-1} \left(\frac{\omega_1}{\omega_P}\right)^{\frac{1}{2(P-1)}}, 
\omega_1^{-1} \left(\frac{\omega_1}{\omega_P}\right)^{\frac{3}{2(P-1)}}, \ldots,
\omega_1^{-1} \left(\frac{\omega_1}{\omega_P}\right)^{\frac{P-4}{2(P-1)}}, 
\omega_1^{-1} \left(\frac{\omega_1}{\omega_P}\right)^{\frac{P-1}{2(P-1)}},  \nonumber \\
 \omega_1^{-1} \left(\frac{\omega_1}{\omega_P}\right)^{\frac{P+2}{2(P-1)}},
\omega_1^{-1} \left(\frac{\omega_1}{\omega_P}\right)^{\frac{P+4}{2(P-1)}}, \ldots,
\omega_1^{-1} \left(\frac{\omega_1}{\omega_P}\right)^{\frac{2P-3}{2(P-1)}},
\omega_1^{-1} \left(\frac{\omega_1}{\omega_P}\right)=\omega_P^{-1}\}
\label{eq:omegaaprox1}
\end{gather}
when $P$ is odd, and 
\begin{gather}
\{\omega_1^{-1}, \omega_1^{-1} \left(\frac{\omega_1}{\omega_P}\right)^{\frac{1}{2(P-1)}}, 
\omega_1^{-1} \left(\frac{\omega_1}{\omega_P}\right)^{\frac{3}{2(P-1)}}, \ldots,
\omega_1^{-1} \left(\frac{\omega_1}{\omega_P}\right)^{\frac{P-5}{2(P-1)}},
\omega_1^{-1} \left(\frac{\omega_1}{\omega_P}\right)^{\frac{P-2}{2(P-1)}},  \nonumber \\
\omega_1^{-1} \left(\frac{\omega_1}{\omega_P}\right)^{\frac{P}{2(P-1)}},
\omega_1^{-1} \left(\frac{\omega_1}{\omega_P}\right)^{\frac{P+3}{2(P-1)}}, 
\omega_1^{-1} \left(\frac{\omega_1}{\omega_P}\right)^{\frac{P+5}{2(P-1)}}, \ldots,
\omega_1^{-1} \left(\frac{\omega_1}{\omega_P}\right)^{\frac{2P-3}{2(P-1)}},
\omega_P^{-1}\}
\label{eq:omegaaprox2}
\end{gather}
when $P$ is even.

Looking at Eqs.\,(\ref{eq:omegaaprox1}) and (\ref{eq:omegaaprox2}), the initial (guess) values of the weights for a new
SRJ scheme with an additional level ($P'=P+1$) can be built providing suitable estimates of the smallest and largest
weights, which are obtained by fitting to two conics the values of the smallest and largest weights computed for SRJ
schemes with $P-3$ to $P$ levels, and then extrapolating the result. For instance, in Fig.\,\ref{fig:GammaOmegas} we
show on the right panel the values of $\omega_1$ and $\omega_P$ as a function of $P$ with red and blue symbols,
respectively. If we want to obtain the initial values of $\omega_1$ and $\omega_P$ for $P=6$, we fit the values of
$\omega_1$ and, separately, those of $\omega_P$ for $P=2,\ldots,5$ to a hyperbola $\omega = \frac{a}{b P -c} + d$ or a
parabola $\omega=a P^2 + b P + c$ depending on the flatness of the points, an infer the value for our $P$ (in
Fig.\,\ref{fig:GammaOmegas}, the fit functions are plot with continuous lines, and the extrapolated values of
$\omega_1$ and $\omega_P$ for $P=6$ with squares).

As we shall see in the next section, and improving on the procedure outlined in ACCA15, we do not need to provide
initial values of $\boldsymbol{\beta}$, since they can be obtained analytically from the values of $\boldsymbol{\omega}$.

\subsection{Advanced analytical improvements}
\label{subsec:analyticalManipulations}

In this section, we prove two important theorems, which tell us how to calculate analytically the ancillary variables
$\frac{\partial \boldsymbol{\omega}}{\partial \boldsymbol{\beta}}$ and the parameters $\boldsymbol{\beta}$ of the SRJ
schemes in terms of the $\boldsymbol{\omega}$ and $\boldsymbol{\kappa}$ variables. Let us start with some technical
results we need for the proof of these theorems. Notice that in all products appearing from now on, each index of the product refers only to expressions containing that particular index.
\begin{lemma}
\label{parFraDec}
Let $A$ and $B$ be two matrices defined as $A:=(a_{ij})=\left(\frac{\kappa_i \, \beta_j}{1-\kappa_i \, \omega_j}\right)$ and 
$B:=(b_{ij})=\left(\frac{1-\omega_j/\omega_P}{1-\kappa_i\,\omega_j}\right)$, $i,j=1,\ldots,P-1$, respectively. The inverse matrices, $A^{-1}=\tilde{A}=(\tilde{a}_{ij})$ and $B^{-1}=\tilde{B}=(\tilde{b}_{ij})$, are given by:
\begin{equation}
\tilde{a}_{ij} = \prod_{\substack{k=1\\ k \neq j}}^{P} \prod_{\substack{l=1\\ l \neq i}}^{P}
\frac{(1-\kappa_j \,\omega_i)(1-\kappa_k \, \omega_i)(1-\kappa_j \, \omega_l)}
{\beta_i \, \kappa_j (\kappa_k-\kappa_j)(\omega_l-\omega_i)},
\label{eq:lemmaM}
\end{equation}
\begin{equation}
\tilde{b}_{ij} = \frac{\omega_P \, (1-\kappa_j \, \omega_i)}{(1-\kappa_j\,\omega_P)}
\prod_{\substack{ k=1 \\ k \neq j} }^{P-1} \prod_{\substack{ l=1 \\ l \neq i} }^{P}
\frac{(1-\kappa_k \,\omega_i)(1-\kappa_j \, \omega_l)}
{(\kappa_k-\kappa_j)(\omega_l-\omega_i)}.
\label{eq:lemmaN}
\end{equation}
\end{lemma}
\begin{proof}
We just need to check that $\sum_{j=1}^P \tilde{a}_{ij} a_{jm} = \sum_{j=1}^P \tilde{b}_{ij} b_{jm} = \delta_{im}$. For convenience, we define $\kappa_P:=\kappa_M$.

We will start checking that
\begin{equation}
\sum_{j=1}^P \frac{1}{(-1+\kappa_j \, \omega_m)} 
\left(\prod_{\substack{k=1\\ k \neq j}}^{P} \prod_{\substack{l=1\\ l \neq i}}^{P}
\frac{(-1+\kappa_j \, \omega_l)}{(\kappa_j-\kappa_k)}\right) = \delta_{im} 
\left(\prod_{\substack{k=1}}^{P} \prod_{\substack{l=1\\ l \neq i}}^{P}
\frac{(\omega_i-\omega_l)}{(-1+\kappa_k \, \omega_i)}\right).
\end{equation}
We consider first the case where $i \neq m$. In general, taking into account that all the $\kappa_i$ are strictly different, for a polynomial $F(x)$, with $\deg F(x)<P-1$, we can do a partial fraction decomposition of the following form:
\begin{equation}
\frac{F(x)}{\prod_{j=1}^{P-1} (x-\kappa_j)} = \sum_{j=1}^{P-1} \frac{F(\kappa_j)}{(x-\kappa_j)} 
\left(\prod_{\substack{k=1\\ k \neq j}}^{P-1} \frac{1}{(\kappa_j-\kappa_k)}\right).
\end{equation}
Considering $F(x)=\prod_{\substack{l=1\\ l \neq i,m}}^{P} (-1+x \, \omega_l)$
in the above expression, and evaluating at $x=\kappa_P$, we get the desired expression.

We consider now the remaining case $i=m$. For convenience, we define $\kappa_0=1/\omega_i$, that satisfies $\kappa_0 \neq \kappa_j, j=1,\ldots,P-1$. For a polynomial $F(x)$, with $\deg F(x)<P$,
we can do a partial fraction decomposition of the following form:
\begin{equation}
\frac{F(x)}{\prod_{j=0}^{P-1} (x-\kappa_j)} = \sum_{j=0}^{P-1} \frac{F(\kappa_j)}{(x-\kappa_j)} 
\left(\prod_{\substack{k=0\\ k \neq j}}^{P-1} \frac{1}{(\kappa_j-\kappa_k)}\right).
\end{equation}
Considering $F(x)=\prod_{\substack{l=1\\ l \neq i}}^{P} (-1+x \, \omega_l)$ in the above expression,
and evaluating at $x=\kappa_P$, we get the desired expression.

We use this equality to check the expression for the inverse matrix $A^{-1}$ (as well as $B^{-1}$ below):
\begin{gather}
\sum_{j=1}^P \tilde{a}_{ij} a_{jm} = \frac{\beta_m}{\beta_i}
\sum_{j=1}^P \frac{(1-\kappa_j \,\omega_i)}{(1-\kappa_j \, \omega_m)}
\prod_{\substack{k=1\\ k \neq j}}^{P} \prod_{\substack{l=1\\ l \neq i}}^{P}
\frac{(-1+\kappa_k \, \omega_i)(-1+\kappa_j \, \omega_l)}{(\kappa_j-\kappa_k)(\omega_i-\omega_l)}
\nonumber \\ 
= \frac{\beta_m}{\beta_i}
\left(\prod_{\substack{k=1}}^{P} \prod_{\substack{l=1\\ l \neq i}}^{P}
\frac{(-1+\kappa_k \, \omega_i)}{(\omega_i-\omega_l)}\right)
\sum_{j=1}^P \frac{1}{(-1+\kappa_j \, \omega_m)}
\prod_{\substack{k=1\\ k \neq j}}^{P} \prod_{\substack{l=1\\ l \neq i}}^{P}
\frac{(-1+\kappa_j \, \omega_l)}{(\kappa_j-\kappa_k)}
\nonumber \\
= \frac{\beta_m}{\beta_i}
\left(\prod_{\substack{k=1}}^{P} \prod_{\substack{l=1\\ l \neq i}}^{P}
\frac{(-1+\kappa_k \, \omega_i)}{(\omega_i-\omega_l)}\right)
\left(\prod_{\substack{k=1}}^{P} \prod_{\substack{l=1\\ l \neq i}}^{P}
\frac{(\omega_i-\omega_l)}{(-1+\kappa_k\omega_i)}\right) \delta_{im} = \delta_{im}.
\end{gather}

Finally, we check the expression for the inverse matrix $B^{-1}$:
\begin{gather}
\sum_{j=1}^P \tilde{b}_{ij} b_{jm} \nonumber \\
= \frac{(\omega_m-\omega_P)}{(\omega_i-\omega_P)}
\left(\prod_{\substack{k=1}}^{P-1} \prod_{\substack{ l=1 \\ l \neq i} }^{P-1}
\frac{(-1+\kappa_k \, \omega_i)}{(\omega_i-\omega_l)}\right)
\sum_{j=1}^P\frac{-1}{(-1+\kappa_j \, \omega_m)}
\prod_{\substack{ k=1 \\ k \neq j} }^{P-1} \prod_{\substack{ l=1 \\ l \neq i} }^{P-1}
\frac{(-1+\kappa_j \, \omega_l)}{(\kappa_j-\kappa_k)}
\nonumber \\
= \frac{\omega_m-\omega_P}{\omega_i-\omega_P} \delta_{im} = \delta_{im}.
\end{gather}
\end{proof}

\begin{theorem}
\begin{equation}
\frac{\partial}{\partial \beta_q} \omega_i = \sum_{j=1}^{P} \prod_{\substack{ k=1 \\ k \neq j} }^{P} \prod_{\substack{ l=1 \\ l \neq i}}^{P} \frac{(1 - \kappa_j \omega_i)}{\beta_i \kappa_j}
\frac{(1 - \kappa_k \omega_i)}{(\kappa_k - \kappa_j)}
\frac{(1 - \kappa_j \omega_l)}{(\omega_l - \omega_i)}
\log \bigg | \frac{1 - \kappa_j \omega_q}{1 - \kappa_j \omega_P} \bigg |. 
\label{eq:ancVar}
\end{equation}
\end{theorem}

\begin{proof}
We have that $\kappa_j$ with $j=1, \ldots, P-1$ are the roots where the local extrema are located. We will also use the already defined $\kappa_P$. Taking into account the Eqs.\,\eqref{eq:partialbeta} and \eqref{eq:equDifBeta}, for a fixed value of $l, 1\leq l \leq P-1$, we construct the linear system 
$A \, \frac{\partial \boldsymbol{\omega}}{\partial \boldsymbol{\beta}} = \boldsymbol{f}$, which in components reads:
\begin{equation}
[a_{ij}] \cdot
\left [ 
   \begin{array}{c}
      \frac{\partial}{\partial \beta_l} \omega_1 \\
      \frac{\partial}{\partial \beta_l} \omega_2 \\
      \vdots \\
      \frac{\partial}{\partial \beta_l} \omega_P \\
   \end{array} 
\right ]
= 
\left [ 
   \begin{array}{c}
      \log | 1 - \kappa_1 \omega_l / 1 - \kappa_1 \omega_P |  \\
      \log | 1 - \kappa_2 \omega_l / 1 - \kappa_2 \omega_P |  \\
      \vdots \\
      \log | 1 - \kappa_P \omega_l / 1 - \kappa_P \omega_P |  \\
   \end{array} 
\right ],
\label{eq:LinSys}
\end{equation}
where $a_{ij}$ is defined in Lemma~\ref{parFraDec}. We can solve for the ancillary variables 
$\frac{\partial \omega_i}{\partial \beta_j}, \, i=1,\ldots,P,\, j=1,\ldots, P-1$, analytically, just inverting the matrix $A$:
\begin{equation}
\left [
   \begin{array}{c}
      \frac{\partial}{\partial \beta_l} \omega_1 \\
      \frac{\partial}{\partial \beta_l} \omega_2 \\
      \vdots \\
      \frac{\partial}{\partial \beta_l} \omega_P \\
   \end{array} 
\right ]
= 
[a_{ij}]^{-1} \cdot
\left [ 
   \begin{array}{c}
      \log | 1 - \kappa_1 \omega_l / 1 - \kappa_1 \omega_P |  \\
      \log | 1 - \kappa_2 \omega_l / 1 - \kappa_2 \omega_P |  \\
      \vdots \\
      \log | 1 - \kappa_P \omega_l / 1 - \kappa_P \omega_P |  \\
   \end{array} 
\right ]
\end{equation}

Using Lemma~\ref{parFraDec} to get the expression of the inverse matrix and doing the corresponding matrix product, we obtain Eq.\,\eqref{eq:ancVar}. Notice that we obtain the same results when we consider $\kappa_m$ instead of $\kappa_P$.
\end{proof}

\begin{theorem}
\begin{equation}
\beta_i = \prod_{k=1}^{P-1} \prod_{\substack{ l=1 \\ l \neq i}}^{P} 
(1 - \kappa_k \omega_i) \frac{\omega_l}{(\omega_l - \omega_i)}.
\label{eq:betai}
\end{equation}
\end{theorem}

\begin{proof}
We consider now the $P-1$ equations resulting from Eq.\,\eqref{eq:partialG} when $\kappa$ is replaced
by $\kappa_i, i=1,\ldots, P-1$. We write $\beta_P = \beta_P(\beta_j)$, and obtain:
\begin{equation}
\sum_{j=1}^{P-1} \frac{(1-\omega_j/\omega_P)}{(1-\kappa_i \, \omega_j)} \beta_j = 1, \, i=1,\ldots,P-1,
\end{equation}
These equations can be rewritten in matrix form, $B \boldsymbol{\beta} = \boldsymbol{g}$, which in components reads: 
\begin{equation}
[b_{ij}] \cdot
\left [ 
   \begin{array}{c}
      \beta_1 \\
      \beta_2 \\
      \vdots \\
      \beta_{P-1} \\
   \end{array} 
\right ] 
= 
\left [ 
   \begin{array}{c}
      1  \\
      1  \\
      \vdots \\
      1  \\
   \end{array} 
\right ],
\label{eq:LinSys}
\end{equation}
where $b_{ij}$ is defined in Lemma~\ref{parFraDec}. Using Lemma~\ref{parFraDec} to get the expression of the inverse matrix and doing the corresponding matrix product, we obtain:
\begin{gather}
\beta_i = \omega_P \sum_{j=1}^{P-1} \frac{(1-\kappa_j \, \omega_i)}{(1-\kappa_j\,\omega_P)}
\prod_{\substack{ k=1 \\ k \neq j} }^{P-1} \prod_{\substack{ l=1 \\ l \neq i} }^{P}
\frac{(-1+\kappa_k\,\omega_i)(-1+\kappa_j\,\omega_l)}{(\kappa_j-\kappa_k)(\omega_i-\omega_l)}
\nonumber \\
= \omega_P \left(\prod_{\substack{k=1}}^{P-1} \prod_{\substack{ l=1 \\ l \neq i} }^{P}
\frac{(-1+\kappa_k \, \omega_i)}{(\omega_i-\omega_l)} \right)
\sum_{j=1}^{P-1} \prod_{\substack{ k=1 \\ k \neq j} }^{P-1} \prod_{\substack{ l=1 \\ l \neq i} }^{P-1}
\frac{(-1+\kappa_j\,\omega_l)}{(\kappa_j-\kappa_k)}.
\end{gather}
In particular, using the simplification proven at the beginning of Lemma~\ref{parFraDec} for $i=m$, changing $P$ by $P-1$, and setting $\omega_i=0$ just formally, we also get that:
\begin{equation}
\sum_{j=1}^{P-1} \left( \prod_{\substack{k=1\\ k \neq j}}^{P-1} \prod_{\substack{l=1\\ l \neq i}}^{P-1}
\frac{(-1+\kappa_j \, \omega_l)}{(\kappa_j-\kappa_k)} \right) =  
\prod_{\substack{l=1\\ l \neq i}}^{P-1} \omega_l.
\end{equation}
Using this expression,
\begin{equation}
\beta_i = \prod_{\substack{k=1}}^{P-1} \prod_{\substack{ l=1 \\ l \neq i} }^{P}
\frac{(-1+\kappa_k \, \omega_i) \, \omega_l}{(\omega_i-\omega_l)}.
\end{equation}
\end{proof}

\subsection{Advantages of the rewritten system}
\label{subsec:rewSys}

In the most general case, as we have commented above, we need to solve a system
$S(\boldsymbol{\omega},\boldsymbol{\beta},\boldsymbol{\kappa},\frac{\partial \boldsymbol{\omega}}{\partial
  \boldsymbol{\beta}})$. With the theorems recently introduced in Sect.~\ref{subsec:analyticalManipulations}, which
basically express $\frac{ \partial \boldsymbol{\omega}}{\partial \boldsymbol{\beta} }$ as $\frac{ \partial
  \boldsymbol{\omega}}{\partial \boldsymbol{\beta} } (\boldsymbol{\omega}, \boldsymbol{\kappa})$ and
$\boldsymbol{\beta}$ also as $\boldsymbol{\beta}(\boldsymbol{\omega}, \boldsymbol{\kappa})$, and substituting them into
Eq.\,\eqref{eq:equMax} and Eq.\,\eqref{eq:partialbeta}, the non-linear, algebraic-differential system reduces to a
purely algebraic system of the form $S(\boldsymbol{\omega},\boldsymbol{\kappa})$.

We have developed a code that implements everything commented in Sect.~\ref{subsec:preImpVar} and that automatically constructs the system
and solves it for any value of $P$. This program is written in Mathematica and combines both symbolic with numerical calculations. We obtain two major benefits from the new system of equations to be solved:

\begin{enumerate}

\item For any given number of levels $P$, we reduce by orders of magnitude the computing time in the generation of the
  symbolic system that we need to solve. In ACCA15, we built $S(\boldsymbol{\omega}, \boldsymbol{\beta})$ and needed a
  lot of time for the symbolic calculations involved in obtaining, on the one hand, $\frac{ \partial
    \boldsymbol{\omega}}{\partial \boldsymbol{\beta} } (\boldsymbol{\omega}, \boldsymbol{\beta})$, solving the
  corresponding linear subsystem (Eq.\,\ref{eq:LinSys}) and, on the other hand,
  $\boldsymbol{\kappa}(\boldsymbol{\omega}, \boldsymbol{\beta})$. With the new methodology, employing
  Eqs.\,\eqref{eq:ancVar} and \eqref{eq:betai}, we exchange the role of $\boldsymbol{\beta}$ and $\boldsymbol{\kappa}$,
  since we have analytic formulae to express the ancillary variables and $\boldsymbol{\beta}$ as functions of
  $\boldsymbol{\omega}$ and $\boldsymbol{\kappa}$, respectively. With the consequent reduction in the calculation time,
  the computational time to solve the remaining equations becomes negligible.

\item The solution of the non-linear system to obtain the optimal SRJ parameters needs a suitable methodology to compute
  their initial values (Sect.~\ref{subsec:preImpVar}). In ACCA15, we had to provide initial values for
  $\boldsymbol{\omega}$ and $\boldsymbol{\beta}$. With the newly derived theorems of the previous
  section, we only need to provide initial guesses for $\boldsymbol{\omega}$, since employing Eq.\,\eqref{eq:betai}
  $\boldsymbol{\beta} = \boldsymbol{\beta}(\boldsymbol{\omega},\boldsymbol{\kappa})$, and $\boldsymbol{\kappa}$
  satisfies
\begin{equation}
\kappa_i \in \bigg ( \frac{1}{\omega_i}, \frac{1}{\omega_{i+1}} \bigg ).
\end{equation}
From the plots of $\Gamma$ (see, e.g., Fig.\,\ref{fig:GammaOmegas}), we can see that each maximum $\kappa_i$ is roughly placed at:
\begin{equation}
\kappa_i \approx \frac{1}{\omega_i} + \frac{\frac{1}{\omega_{i+1}} - \frac{1}{\omega_i} }{3}, 
\end{equation}
which are the values that we will use as initial guesses. 
\end{enumerate}

With these two improvements, we have reduced by four orders of magnitude the total time for finding the parameters of an
optimal scheme. For example, for $P=10$, it was necessary to spend a calculation time in Mathematica of the order of one
week with the methodology employed by ACCA15. In contrast, with the improvements reported here, we can accomplish the
same task in tens of seconds. While, in practice, in ACCA15 we were limited (due to the computing time) to SRJ schemes
with $P \le10$, now we can tackle larger number of levels.

\section{Results}
\label{sec:results}
\begin{table*}
\caption{Parameters $\boldsymbol{w}$, $\boldsymbol{\beta}$ and the
  estimation of the convergence performance index $\rho =
  \sum_{i=1}^{P} \omega_i \beta_i$ of the $P=2$, $P=3$, $P=4$ and
  $P=5$ schemes for a number of values of $N$ and the model problem
  specified in Eq.\,(\ref{eq:Laplace}).}
\scriptsize
\begin{centering}
\bigskip
\begin{tabular}{c c | c | c | c }
 N & & $P=2$ & $P=3$ & $P=4$  \\
\hline
$\multirow{3}{*}{100}$ & $\boldsymbol{\omega}$ & $\{321.074,0.968096\}$ & $\{1420.73,30.0648,0.845599\}$ & $\{2308.12,162.259,8.50839,0.732499\}$  \\
 & $\boldsymbol{\beta} $ & $ \{0.00993673,0.990063\}$ & $\{0.00502828,0.0729552,0.922017\}$ & $\{0.00430412,0.0245487,0.158309,0.812838\}$  \\
      & $\rho$ & 4.15 & 10.12 & 15.86 \\
\hline
$\multirow{3}{*}{150}$ & $\boldsymbol{\omega} $ & $ \{509.976,0.977667\}$ & $\{2724.66,41.8246,0.870558\}$ & $\{4707.62,259.325,10.9382,0.756243\}$  \\
 & $\boldsymbol{\beta} $ & $ \{0.0064850,0.99352\}$ & $\{0.00304593,0.0574955,0.939459\}$ & $\{0.00268244,0.0182254,0.138417,0.840676\}$  \\
      & $\rho$ & 4.28 & 11.52 & 19.50 \\
\hline
$\multirow{3}{*}{200}$ & $\boldsymbol{\omega} $ & $ \{704.099,0.982735\}$ & $\{4295.,52.6521,0.886485\}$ & $\{7968.04,368.694,13.2257,0.774085\}$  \\
 & $\boldsymbol{\beta} $ & $ \{0.00478657,0.995213\}$ & $\{0.00211898,0.048302,0.949579\}$ & $\{0.00185082,0.0144313,0.124148,0.85957\}$  \\
      & $\rho$ & 4.35 & 12.49 & 22.38 \\
\hline
$\multirow{3}{*}{250}$ & $\boldsymbol{\omega} $ & $ \{901.84,0.985888\}$ & $\{6090.23,62.8089,0.897814\}$ & $\{11871.8,481.378,15.1867,0.786453\}$  \\
 & $\boldsymbol{\beta} $ & $ \{0.00378101,0.996219\}$ & $\{0.00159331,0.0420853,0.956321\}$ & $\{0.00139269,0.0120297,0.114584,0.871993\}$  \\
      & $\rho$ & 4.39 & 13.21 & 24.75 \\
\hline
$\multirow{3}{*}{300}$ & $\boldsymbol{\omega} $ & $ \{1102.34,0.988045\}$ & $\{8082.34,72.4478,0.906414\}$ & $\{16301.,591.753,17.0536,0.797245\}$  \\
 & $\boldsymbol{\beta} $ & $ \{0.00311809,0.996882\}$ & $\{0.00125948,0.0375472,0.961193\}$ & $\{0.00110797,0.0104108,0.106471,0.88201\}$  \\
      & $\rho$ & 4.42 & 13.77 & 26.74 \\
\hline
$\multirow{3}{*}{350}$ & $\boldsymbol{\omega} $ & $ \{1305.06,0.989617\}$ & $\{10250.9,81.6684,0.913233\}$ & $\{21362.1,707.502,18.7691,0.805757\}$  \\
 & $\boldsymbol{\beta} $ & $ \{0.00264906,0.997351\}$ & $\{0.001031,0.0340609,0.964908\}$ & $\{0.000908394,0.0091685,0.100232,0.889691\}$  \\
      & $\rho$ & 4.44 & 14.23 & 28.49 \\
\hline
$\multirow{3}{*}{400}$ & $\boldsymbol{\omega} $ & $ \{1509.63,0.990814\}$ & $\{12580.2,90.5404,0.918815\}$ & $\{27421.7,850.177,20.3972,0.812165\}$  \\
 & $\boldsymbol{\beta} $ & $ \{0.00230021,0.9977\}$ & $\{0.000866032,0.0312827,0.967851\}$ & $\{0.000751992,0.00802258,0.0955775,0.895648\}$  \\
      & $\rho$ & 4.46 & 14.62 & 30.12 \\
\hline
$\multirow{3}{*}{450}$ & $\boldsymbol{\omega} $ & $ \{1715.8,0.991758\}$ & $\{15057.7,99.1151,0.923493\}$ & $\{35453.1,1161.24,24.1452,0.825463\}$  \\
 & $\boldsymbol{\beta} $ & $ \{0.00203086,0.997969\}$ & $\{0.000742062,0.0290068,0.970251\}$ & $\{0.000612978,0.00627071,0.0853787,0.907738\}$  \\
      & $\rho$ & 4.47 & 14.94 & 31.82 \\
\hline
$\multirow{3}{*}{500}$ & $\boldsymbol{\omega} $ & $ \{1923.36,0.992522\}$ & $\{17673.1,107.432,0.927487\}$ & $\{41329.,1177.71,25.2645,0.831242\}$  \\
 & $\boldsymbol{\beta} $ & $ \{0.00181679,0.998183\}$ & $\{0.000645947,0.0271018,0.972252\}$ & $\{0.000546643,0.00623392,0.0815398,0.91168\}$ \\
      & $\rho$ & 4.49 & 15.23 & 32.75 
\end{tabular}

\bigskip
\scriptsize
\begin{tabular}{c c | c }
 N & & $P=5$  \\
\hline
$\multirow{3}{*}{100}$ & $\boldsymbol{\omega}$ & $ \{2846.74,411.781,40.0941,3.97003,0.659793\}$  \\
      & $\boldsymbol{\beta}$ & $\{0.00395334,0.0134445,0.0549429,0.22302,0.70464\}$   \\
            & $\rho$ & 20.34 \\
\hline
$\multirow{3}{*}{150}$ & $\boldsymbol{\omega}$ & $ \{6083.43,723.916,58.9841,4.88096,0.679269\}$  \\
      & $\boldsymbol{\beta}$ & $\{0.00248163,0.00967631,0.04472,0.205462,0.73766\}$   \\
            & $\rho$ & 26.24 \\
\hline
$\multirow{3}{*}{200}$ & $\boldsymbol{\omega}$ & $ \{10402.8,1077.5,77.4789,5.6526,0.693256\}$  \\
      & $\boldsymbol{\beta}$ & $\{0.00176797,0.00760734,0.038404,0.192846,0.759375\}$   \\
            & $\rho$ & 31.18 \\
\hline
$\multirow{3}{*}{250}$ & $\boldsymbol{\omega}$ & $ \{15750.6,1464.91,95.6673,6.33405,0.704153\}$  \\
      & $\boldsymbol{\beta}$ & $\{0.00135255,0.00628676,0.0340103,0.183094,0.775256\}$   \\
            & $\rho$ & 35.47 \\
\hline
$\multirow{3}{*}{300}$ & $\boldsymbol{\omega}$ & $ \{22085.5,1881.21,113.603,6.95081,0.713063\}$  \\
      & $\boldsymbol{\beta}$ & $\{0.00108339,0.00536588,0.0307308,0.175201,0.787619\}$   \\
            & $\rho$ & 39.29 \\
\hline
$\multirow{3}{*}{350}$ & $\boldsymbol{\omega}$ & $ \{29373.9,2322.91,131.323,7.51817,0.720588\}$  \\
      & $\boldsymbol{\beta}$ & $\{0.000896175,0.00468493,0.028165,0.168605,0.797649\}$   \\
            & $\rho$ & 42.75 \\
\hline
$\multirow{3}{*}{400}$ & $\boldsymbol{\omega}$ & $ \{37587.8,2787.39,148.854,8.04621,0.727091\}$  \\
      & $\boldsymbol{\beta}$ & $\{0.000759202,0.00415986,0.0260888,0.162962,0.806030\}$   \\
            & $\rho$ & 45.91 \\
\hline
$\multirow{3}{*}{450}$ & $\boldsymbol{\omega}$ & $ \{46703.1,3272.60,166.218,8.54196,0.732811\}$  \\
      & $\boldsymbol{\beta}$ & $\{0.000655107,0.00374204,0.0243656,0.158048,0.813189\}$   \\
            & $\rho$ & 48.84 \\
\hline
$\multirow{3}{*}{500}$ & $\boldsymbol{\omega}$ & $ \{56698.8,3776.87,183.430,9.01057,0.737910\}$  \\
      & $\boldsymbol{\beta}$ & $\{0.000573622,0.00340130,0.0229066,0.153708,0.819411\}$   \\
            & $\rho$ & 51.56 \\
\end{tabular}
\label{table:weiBet}

\end{centering}
\end{table*}

In this section we first calibrate our new method comparing the
parameters of our optimal SRJ schemes with those of YM14 for $P\le 5$
(Sect.\,\ref{sec:calibration}). Later (Sect.\,\ref{sec:newSRJ}), we
present new optimal schemes computed employing the new methodology
sketched in the previous section, up to $P=15$.

\subsection{Calibration of the method}
\label{sec:calibration}

To calibrate the new methodology, we have recomputed the optimal
parameters for SRJ schemes with $P\le 5$ and found that our results
are the same as those obtained by YM14, when we use the same number of points
per dimension $N$ on the same model problem (Eq.\,\ref{eq:Laplace}).  

Following the ideas of YM14, the performance of any SRJ scheme with respect to the Jacobi method can be quantified
estimating the \emph{convergence performance index}, $\rho$,
\begin{equation}
\rho := \sum_{i=1}^{P} \omega_i \beta_i
\label{eq:rho}
\end{equation} 
which we have calculated for each SRJ method we have computed, and checked that it approaches its theoretical value when
we solve numerically (Eq.\,\ref{eq:Laplace}). We point out that the value of $\rho$ depends on the dimensionality of
the problem since the value of $\kappa_m$ does (see, Eq.\,\ref{eq:boundaries1D3D}).

YM14 showed that the optimal parameters computed for coarser grids can be used for finer ones. Nevertheless, minimizing
the gaps between different values of $N$ is important because the acceleration of the convergence with respect to the
Jacobi method may not be the largest possible unless we compute the optimal SRJ parameters corresponding to a given
problem size. Thus, we have completed the tables presented by YM14 minimizing the possible gaps between
resolutions. Furthermore, we have computed the optimal SRJ parameters for a number of intermediate values of $N$ in
Tab.~\ref{table:weiBet}, where we also show the value of $\rho$.
\begin{figure*}
\centering
\includegraphics[width=0.8\textwidth]{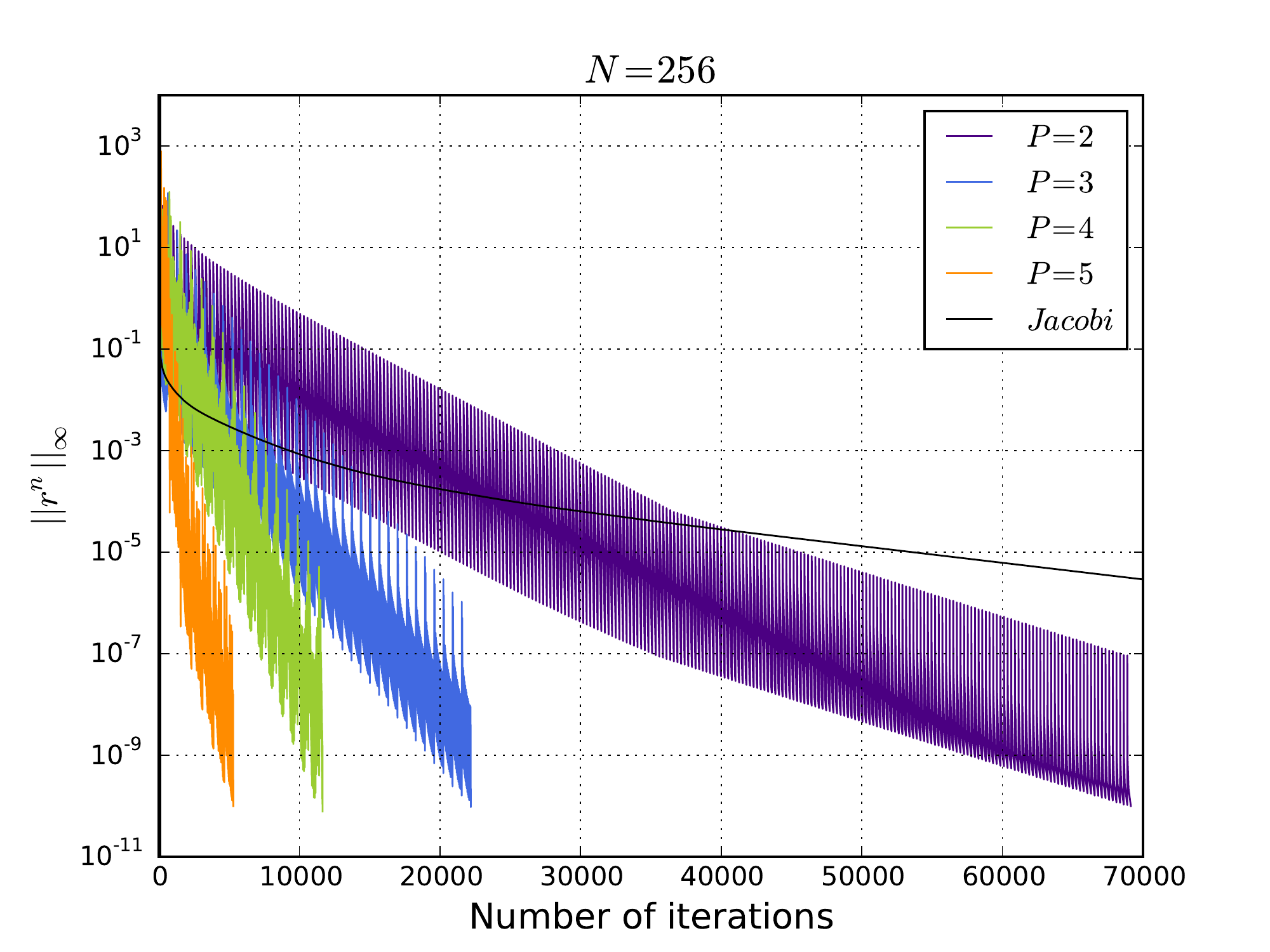} \\
\includegraphics[width=0.8\textwidth]{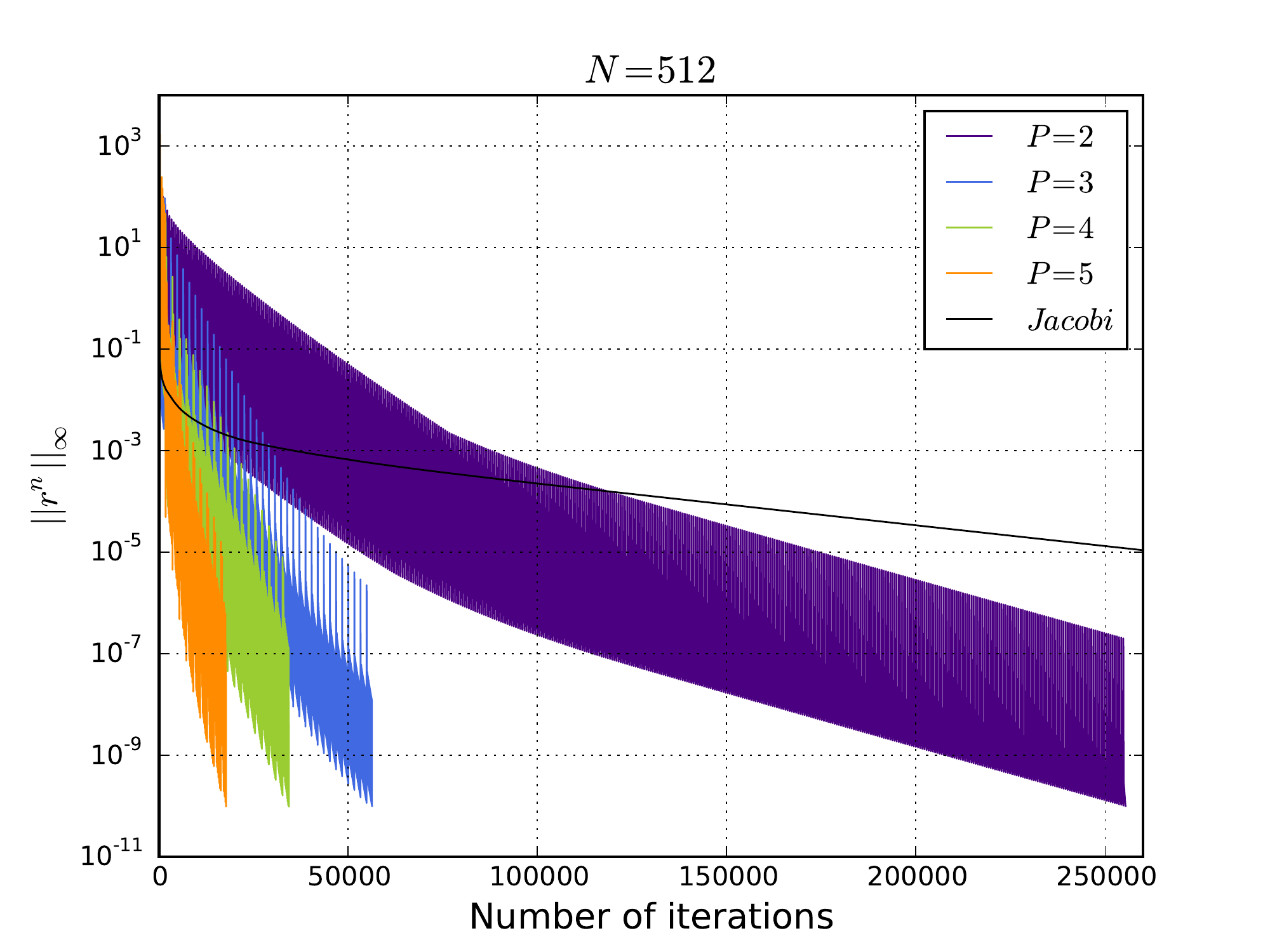} 
\caption{Comparison of the evolution of the difference between consecutive approximate solutions ($||r^n||_{\infty}$) of
  Eq.~(\ref{eq:2nd-order_CD}) of SRJ schemes from $P=2$ to $P=5$ for a
  grid with $N=256$ and with $N=512$ zones per dimension. We also include the evolution
  of the residual for the Jacobi method as a
  reference.}\label{fig:P2toP5}
\end{figure*}
In order to verify the correct behaviour of the schemes computed, we
monitor in Fig.\,\ref{fig:P2toP5} the evolution of the difference between two consecutive approximations
of the solution for the model problem specified in Eq.~(\ref{eq:Laplace}),
\begin{equation}
r^{n}_{ij} = u^{n}_{ij} - u^{n-1}_{ij},
\end{equation}
using element-wise norms and operations, that in the bidimensional case would be, for example,
\begin{equation}
||r^{n}||_{\infty} = \max_{ij} r^{n}_{ij},
\end{equation}
as a function of the number of iterations $n$ for SRJ schemes having all
the values of $P$ given by YM14. In the same figure, we also include
the residual evolution for the Jacobi method (black line). As expected,
the number of iterations to reach the prescribed tolerance decreases%
\footnote{If not explicitly mentioned, in all the cases considered in this paper the absolute tolerance 
is fixed so that $||r^n||_{\infty} < 10^{-10}$.}
as $P$ increases. For all the schemes shown in
Fig.\,\ref{fig:P2toP5}, where the number of points is set to $N=256$,
we have obtained the expected theoretical value of $\rho$.

%
\begin{table}
\caption{Parameters for optimized $P=14$ SRJ schemes for various values of $N$.}
\tiny
\begin{center}
\begin{tabular}{c c c c}
\hline
N & Optimal Scheme Parameters & $\rho$  \\
\hline
\multirow{2}{*}{100} & $\boldsymbol{\omega}=\{3878.55,2785.67,1570.02,769.498,351.868,155.898,68.1753,29.7098,12.9941,5.75755,2.75287,1.36341,0.743467,0.523835\}$ & \multirow{2}{*}{35.12} \\
& $\boldsymbol{\beta}=\{0.00223066,0.00268294,0.00367875,0.00541949,0.00825729,0.012765,0.019848,0.0309124,0.0481036,0.0748715,0.0959847,0.170428,0.2368,0.288017\}$ & \\
\hline
\multirow{2}{*}{150} & $\boldsymbol{\omega}=\{8671.64,5959.56,3144.72,1438.03,615.285,255.886,105.238,43.1569,17.7519,7.38476,3.24126,1.48992,0.772445,0.526323\}$ & \multirow{2}{*}{50.34} \\
& $\boldsymbol{\beta}=\{0.00151803,0.00187383,0.00266836,0.00408745,0.00646311,0.0103505,0.0166552,0.0268341,0.0431958,0.0694114,0.100573,0.169923,0.243817,0.302629\}$ & \\
\hline
\multirow{2}{*}{200} & $\boldsymbol{\omega}=\{15342.9,10203.7,5132.76,2234.12,912.112,362.768,142.862,56.1217,22.1054,8.79573,3.65467,1.59758,0.796262,0.528322\}$ & \multirow{2}{*}{64.84} \\
& $\boldsymbol{\beta}=\{0.00115001,0.00144748,0.00211891,0.00333763,0.00541838,0.00889808,0.0146726,0.0242198,0.0399456,0.0657302,0.101006,0.169306,0.248859,0.313891\}$ & \\
\hline
\multirow{2}{*}{250} & $\boldsymbol{\omega}=\{23881.3,15470.9,7496.09,3140.75,1236.83,475.305,181.025,68.7946,26.2065,10.0762,4.02155,1.69204,0.816583,0.529997\}$ & \multirow{2}{*}{78.79} \\
& $\boldsymbol{\beta}=\{0.000924757,0.00118231,0.00176872,0.00284695,0.00471686,0.00789814,0.0132733,0.0223266,0.0375256,0.0629164,0.100286,0.168614,0.252701,0.323018\}$ & \\
\hline
\multirow{2}{*}{300} & $\boldsymbol{\omega}=\{34278.1,21725.7,10207.4,4146.18,1585.72,592.623,219.656,81.2544,30.1232,11.2643,4.35503,1.7766,0.834356,0.531439\}$ & \multirow{2}{*}{92.3} \\
& $\boldsymbol{\beta}=\{0.000772581,0.00100079,0.00152417,0.002497,0.00420622,0.00715567,0.0122136,0.020863,0.0356119,0.0606376,0.0992048,0.167894,0.25575,0.330669\}$ & \\
\hline
\multirow{2}{*}{350} & $\boldsymbol{\omega}=\{46526.,28939.3,13245.9,5241.76,1956.08,714.096,258.693,93.545,33.8951,12.3814,4.66283,1.85345,0.850189,0.532707\}$ & \multirow{2}{*}{105.45} \\
& $\boldsymbol{\beta}=\{0.000662864,0.000868418,0.00134284,0.00223296,0.00381437,0.00657652,0.0113733,0.0196825,0.0340395,0.0587264,0.0980248,0.16717,0.258242,0.337243\}$ & \\
\hline
\multirow{2}{*}{400} & $\boldsymbol{\omega}=\{60618.7,37088.,16594.7,6420.75,2345.86,839.25,298.085,105.695,37.5482,13.4414,4.95006,1.92411,0.864496,0.533839\}$ & \multirow{2}{*}{118.29} \\
& $\boldsymbol{\beta}=\{0.000580015,0.000767473,0.00120256,0.00202561,0.00350221,0.00610869,0.010685,0.0187016,0.0327121,0.0570842,0.0968455,0.166458,0.260328,0.343\}$ & \\
\hline
\multirow{2}{*}{450} & $\boldsymbol{\omega}=\{76550.6,46151.6,20240.1,7677.7,2753.45,967.72,337.793,117.724,41.1009,14.4541,5.22033,1.98968,0.877567,0.534862\}$ & \multirow{2}{*}{130.87} \\
& $\boldsymbol{\beta}=\{0.000515249,0.000687865,0.00109052,0.00185784,0.00324648,0.00572076,0.0101073,0.017868,0.0315686,0.0556475,0.0957059,0.165764,0.262105,0.348115\}$ & \\
\hline
\multirow{2}{*}{500} & $\boldsymbol{\omega}=\{94317.,56112.2,24170.1,9008.13,3177.53,1099.21,377.785,129.647,44.5668,15.4266,5.47629,2.05098,0.889618,0.535797\}$ & \multirow{2}{*}{143.20} \\
& $\boldsymbol{\beta}=\{0.000463236,0.000623426,0.000998804,0.00171891,0.00303236,0.00539248,0.00961327,0.0171471,0.0305679,0.054373,0.0946203,0.165092,0.263642,0.352715\}$ & \\
\hline
\multirow{2}{*}{550} & $\boldsymbol{\omega}=\{113913.,66954.2,28374.8,10408.2,3617.01,1233.48,418.035,141.476,47.9566,16.3643,5.71997,2.10865,0.900811,0.536657\}$ & \multirow{2}{*}{155.31} \\
& $\boldsymbol{\beta}=\{0.000420555,0.000570165,0.00092222,0.0016017,0.00284993,0.00511011,0.00918421,0.0165148,0.0296809,0.0532295,0.0935924,0.164444,0.264987,0.356893\}$ & \\
\hline
\multirow{2}{*}{600} & $\boldsymbol{\omega}=\{135336.,78663.8,32845.3,11874.7,4070.95,1370.32,458.522,153.221,51.2786,17.2714,5.95295,2.16317,0.91127,0.537454\}$ & \multirow{2}{*}{167.24} \\
& $\boldsymbol{\beta}=\{0.000384908,0.000525385,0.000857231,0.0015013,0.00269227,0.00486396,0.00880696,0.015954,0.0288864,0.0521941,0.0926211,0.163818,0.266176,0.360718\}$ & \\
\hline
\multirow{2}{*}{650} & $\boldsymbol{\omega}=\{158580.,91228.1,37573.6,13404.7,4538.56,1509.57,499.227,164.888,54.5397,18.1514,6.17652,2.21494,0.921097,0.538198\}$ & \multirow{2}{*}{178.98} \\
& $\boldsymbol{\beta}=\{0.000354693,0.000487197,0.000801333,0.0014142,0.00255438,0.00464697,0.00847177,0.0154516,0.0281686,0.0512493,0.0917036,0.163215,0.267237,0.364244\}$ & \\
\hline
\multirow{2}{*}{700} & $\boldsymbol{\omega}=\{183644.,104636.,42552.7,14995.7,5019.15,1651.07,540.135,176.484,57.7456,19.0071,6.39171,2.26428,0.930369,0.538894\}$ & \multirow{2}{*}{190.56} \\
& $\boldsymbol{\beta}=\{0.000328761,0.000454234,0.000752702,0.00133783,0.00243256,0.00445388,0.00817133,0.0149979,0.0275151,0.0503814,0.090836,0.162634,0.268191,0.367514\}$ & \\
\hline
\multirow{2}{*}{750} & $\boldsymbol{\omega}=\{210523.,118876.,47776.4,16645.5,5512.09,1794.7,581.231,188.016,60.9012,19.8408,6.59939,2.31147,0.939153,0.53955\}$ & \multirow{2}{*}{201.99} \\
& $\boldsymbol{\beta}=\{0.000306266,0.000425485,0.000709977,0.00127025,0.00232401,0.00428065,0.00789997,0.0145854,0.0269165,0.0495796,0.0900146,0.162073,0.269053,0.370562\}$ & \\
\hline
\multirow{2}{*}{800} & $\boldsymbol{\omega}=\{239214.,133939.,53238.8,18352.1,6016.86,1940.35,622.504,199.487,64.0103,20.6544,6.80028,2.35671,0.947503,0.540169\}$ & \multirow{2}{*}{213.28} \\
& $\boldsymbol{\beta}=\{0.000286569,0.000400187,0.00067212,0.00120995,0.00222654,0.00412413,0.00765326,0.0142079,0.026365,0.0488352,0.0892357,0.161532,0.269837,0.373415\}$ & \\
\hline
\multirow{2}{*}{850} & $\boldsymbol{\omega}=\{269715.,149816.,58934.6,20113.6,6532.95,2087.9,663.944,210.901,67.0766,21.4498,6.99499,2.40021,0.955464,0.540756\}$ & \multirow{2}{*}{224.44} \\
& $\boldsymbol{\beta}=\{0.000269181,0.000377748,0.000638326,0.00115578,0.00213845,0.00398185,0.00742767,0.0138606,0.0258544,0.048141,0.0884956,0.161009,0.270553,0.376097\}$ & \\
\hline
\multirow{2}{*}{900} & $\boldsymbol{\omega}=\{302024.,166499.,64858.9,21928.4,7059.93,2237.28,705.541,222.262,70.1032,22.2283,7.18405,2.44212,0.963074,0.541314\}$ & \multirow{2}{*}{235.47} \\
& $\boldsymbol{\beta}=\{0.000253722,0.000357707,0.00060796,0.00110682,0.00205836,0.00385178,0.00722033,0.0135396,0.0253796,0.0474912,0.0877914,0.160504,0.271211,0.378626\}$ & \\
\hline
\multirow{2}{*}{950} & $\boldsymbol{\omega}=\{336137.,183978.,71007.3,23794.9,7597.41,2388.4,747.287,233.574,73.0926,22.9911,7.36793,2.48257,0.970367,0.541846\}$ & \multirow{2}{*}{246.38} \\
& $\boldsymbol{\beta}=\{0.000239888,0.000339698,0.000580514,0.00106231,0.00198518,0.00373232,0.00702889,0.0132417,0.0249364,0.0468808,0.0871199,0.160016,0.271817,0.38102\}$ & \\
\hline
\multirow{2}{*}{1000} & $\boldsymbol{\omega}=\{372052.,202248.,77375.6,25711.8,8145.01,2541.19,789.175,244.838,76.0472,23.7394,7.54701,2.52169,0.97737,0.542354\}$ & \multirow{2}{*}{257.19} \\
& $\boldsymbol{\beta}=\{0.000227438,0.000323423,0.000555578,0.00102166,0.001918,0.0036221,0.0068514,0.0129641,0.0245213,0.0463056,0.0864787,0.159543,0.272378,0.38329\}$ & \\
\hline
\hline
\multirow{2}{*}{32} & $\boldsymbol{\omega}=\{402.8,318.533,208.846,120.929,65.1531,33.8031,17.2446,8.7701,4.50164,2.36779,1.58821,1.06413,0.669773,0.517215\}$ & \multirow{2}{*}{12.47} \\
& $\boldsymbol{\beta}=\{0.00647086,0.0073607,0.00926188,0.0124338,0.0173072,0.0245363,0.035068,0.0502174,0.0717549,0.104058,0.0350767,0.168979,0.212666,0.24481\}$ & \\
\hline
\multirow{2}{*}{64} & $\boldsymbol{\omega}=\{1598.72,1200.1,724.027,381.957,187.598,88.992,41.5807,19.3493,9.04533,4.29382,2.29786,1.25078,0.716642,0.521476\}$ & \multirow{2}{*}{23.51} \\
& $\boldsymbol{\beta}=\{0.00338329,0.0039658,0.00523063,0.00739484,0.0108296,0.0161218,0.0241681,0.0363079,0.0545086,0.0822675,0.082198,0.170812,0.229322,0.27349\}$ & \\
\hline
\multirow{2}{*}{128} & $\boldsymbol{\omega}=\{6330.52,4427.82,2398.11,1126.86,494.844,210.953,88.8667,37.3202,15.7238,6.70433,3.03854,1.43707,0.760478,0.525303\}$ & \multirow{2}{*}{43.75} \\
& $\boldsymbol{\beta}=\{0.00176623,0.00215768,0.00302688,0.00456605,0.00711566,0.0112388,0.0178429,0.0283674,0.0450598,0.0714887,0.0994468,0.170158,0.241043,0.296722\}$ & \\
\hline
\multirow{2}{*}{256} & $\boldsymbol{\omega}=\{25031.,16170.,7803.54,3256.38,1277.48,489.145,185.638,70.2999,26.6852,10.2232,4.06315,1.70266,0.818836,0.530181\}$ & \multirow{2}{*}{80.43} \\
& $\boldsymbol{\beta}=\{0.000903457,0.00115703,0.00173491,0.00279897,0.00464739,0.00779789,0.0131313,0.022132,0.0372734,0.0626189,0.100168,0.168529,0.253103,0.324005\}$ & \\
\hline
\multirow{2}{*}{512} & $\boldsymbol{\omega}=\{98853.4,58634.4,25154.6,9337.93,3281.64,1131.19,387.422,132.494,45.3869,15.6547,5.53583,2.06513,0.892377,0.536009\}$ & \multirow{2}{*}{146.12} \\
& $\boldsymbol{\beta}=\{0.000452243,0.000609742,0.000979198,0.00168901,0.00298598,0.00532095,0.00950493,0.016988,0.0303456,0.0540876,0.0943683,0.164935,0.26398,0.353754\}$ & \\
\hline
\multirow{2}{*}{$1024$} & $\boldsymbol{\omega}=\{389930.,211296.,80509.3,26649.4,8411.36,2615.1,809.329,250.229,77.4537,24.0938,7.63137,2.54003,0.980636,0.542591\}$ & \multirow{2}{*}{262.34} \\
& $\boldsymbol{\beta}=\{0.000221894,0.000316155,0.000544398,0.00100336,0.00188766,0.00357215,0.00677068,0.0128374,0.0243311,0.0460409,0.086181,0.159321,0.272632,0.38434\}$ & \\
\hline
\multirow{2}{*}{2048} & $\boldsymbol{\omega}=\{1536240,757099.,256177.,75783.7,21526.3,6043.96,1691.56,473.123,132.43,37.1993,10.5906,3.1509,1.08412,0.549805\}$ & \multirow{2}{*}{465.33} \\
& $\boldsymbol{\beta}=\{0.000106867,0.000161473,0.000298341,0.000587241,0.00117546,0.00236283,0.00475453,0.00956894,0.0192541,0.038699,0.077294,0.152063,0.278729,0.414945\}$ & \\
\hline
\multirow{2}{*}{4096} & $\boldsymbol{\omega}=\{6045072,2698661,811234.,214918.,55022.7,13963.2,3535.64,894.898,226.617,57.5285,14.7527,3.93026,1.20411,0.557585\}$ & \multirow{2}{*}{815.34} \\
& $\boldsymbol{\beta}=\{0.0000505714,0.0000812954,0.000161234,0.000338826,0.000721677,0.00154174,0.00329580,0.00704624,0.0150625,0.0321750,0.0684563,0.143616,0.282283,0.445170\}$ & \\
\hline
\end{tabular}
\label{table:weiP14_1}
\end{center}
\end{table}
%

\begin{table}
\caption{Parameters for optimized $P=15$ SRJ schemes for various values of $N$.}
\tiny 
\begin{center}
\begin{tabular}{c c c c}
\hline
N & Optimal Scheme Parameters & $\rho$  \\
\hline
\multirow{2}{*}{100} & $\boldsymbol{\omega}=\{3900.7,2917.56,1750.24,917.535,447.788,211.033,97.8831,45.1243,20.8015,9.63811,5.19063,2.69523,1.30463,0.729676,0.52263\}$ & \multirow{2}{*}{35.62} \\
& $\boldsymbol{\beta}=\{0.00211354,0.00248261,0.00328491,0.00466027,0.00684846,0.0102307,0.015395,0.0232377,0.0351381,0.0539546,0.0440768,0.115928,0.169266,0.232834,0.28055\}$ & \\
\hline
\multirow{2}{*}{150} & $\boldsymbol{\omega}=\{8731.9,6299.82,3574.97,1764.89,812.568,362.251,159.248,69.6404,30.449,13.37,6.43047,3.027,1.40229,0.752546,0.524622\}$ & \multirow{2}{*}{51.31} \\
& $\boldsymbol{\beta}=\{0.00143811,0.00172513,0.00235617,0.00345693,0.00524685,0.00808252,0.0125283,0.0194675,0.0302848,0.0474906,0.0542321,0.11255,0.169529,0.238971,0.29264\}$ & \\
\hline
\multirow{2}{*}{200} & $\boldsymbol{\omega}=\{15459.6,10836.2,5887.55,2775.66,1222.6,522.574,220.57,92.6754,38.9339,16.42,7.3844,3.28605,1.47625,0.769405,0.526065\}$ & \multirow{2}{*}{66.29} \\
& $\boldsymbol{\beta}=\{0.00109279,0.0013333,0.00186699,0.0028111,0.00437317,0.00689647,0.010935,0.017374,0.0276256,0.0441322,0.0577854,0.110203,0.169452,0.243027,0.301093\}$ & \\
\hline
\multirow{2}{*}{250} & $\boldsymbol{\omega}=\{24074.2,16482.5,8649.12,3932.14,1672.98,691.895,282.873,115.18,46.8951,19.1619,8.22785,3.52027,1.54157,0.783992,0.527297\}$ & \multirow{2}{*}{80.75} \\
& $\boldsymbol{\beta}=\{0.00088128,0.00109027,0.00155754,0.00239368,0.00379676,0.006099,0.00984518,0.0159203,0.0257581,0.0418003,0.0589403,0.108288,0.169245,0.246258,0.308126\}$ & \\
\hline
\multirow{2}{*}{300} & $\boldsymbol{\omega}=\{34568.,23204.9,11831.2,5221.34,2159.63,869.493,346.353,137.457,54.5501,21.7213,9.00657,3.73816,1.60113,0.797059,0.528388\}$ & \multirow{2}{*}{94.79} \\
& $\boldsymbol{\beta}=\{0.000738052,0.000923817,0.00134184,0.00209712,0.00337956,0.00551145,0.00902828,0.0148121,0.0243112,0.0399833,0.0591017,0.106635,0.168961,0.248953,0.314222\}$ & \\
\hline
\multirow{2}{*}{350} & $\boldsymbol{\omega}=\{46934.7,30975.9,15411.5,6633.06,2679.17,1054.58,410.977,159.617,61.9913,24.1524,9.7385,3.94279,1.65611,0.808933,0.52937\}$ & \multirow{2}{*}{108.48} \\
& $\boldsymbol{\beta}=\{0.00063454,0.000802316,0.00118197,0.00187368,0.00306013,0.00505457,0.00838329,0.0139235,0.0231325,0.0384872,0.058814,0.105172,0.168631,0.251248,0.319602\}$ & \\
\hline
\multirow{2}{*}{400} & $\boldsymbol{\omega}=\{61168.7,39772.5,19371.5,8159.09,3228.87,1246.47,476.666,181.708,69.2673,26.4843,10.4333,4.13611,1.70724,0.819823,0.530261\}$ & \multirow{2}{*}{121.87} \\
& $\boldsymbol{\beta}=\{0.000556209,0.000709559,0.00105828,0.00169833,0.00280596,0.0046861,0.00785612,0.0131873,0.0221419,0.0372151,0.0583153,0.103854,0.168272,0.253232,0.324411\}$ & \\
\hline
\multirow{2}{*}{450} & $\boldsymbol{\omega}=\{77265.3,49575.,23696.,9792.69,3806.5,1444.59,543.333,203.752,76.408,28.7354,11.0971,4.3196,1.7551,0.829887,0.531078\}$ & \multirow{2}{*}{135.00} \\
& $\boldsymbol{\beta}=\{0.000494861,0.000636339,0.000959489,0.00155652,0.00259789,0.0043809,0.00741432,0.0125628,0.0212909,0.0361099,0.0577191,0.102656,0.167899,0.25497,0.328752\}$ & \\
\hline
\multirow{2}{*}{500} & $\boldsymbol{\omega}=\{95220.,60366.1,28371.8,11528.2,4410.2,1648.44,610.9,225.763,83.4336,30.9182,11.7345,4.49446,1.80013,0.839249,0.531833\}$ & \multirow{2}{*}{147.89} \\
& $\boldsymbol{\beta}=\{0.000445513,0.00057702,0.000878606,0.00143911,0.00242379,0.00412284,0.00703687,0.0120236,0.0205478,0.0351345,0.0570832,0.101557,0.167517,0.256507,0.332706\}$ & \\
\hline
\multirow{2}{*}{550} & $\boldsymbol{\omega}=\{115029.,72130.4,33387.7,13360.8,5038.38,1857.63,679.299,247.749,90.3589,33.0419,12.3487,4.66169,1.8427,0.848005,0.532533\}$ & \multirow{2}{*}{160.58} \\
& $\boldsymbol{\beta}=\{0.00040496,0.000527957,0.000811066,0.0013401,0.00227553,0.00390105,0.00670944,0.0115513,0.0198904,0.034263,0.0564379,0.100541,0.167134,0.25788,0.336333\}$ & \\
\hline
\multirow{2}{*}{600} & $\boldsymbol{\omega}=\{136689.,84854.1,38733.6,15286.4,5689.71,2071.79,748.469,269.715,97.1953,35.1137,12.9426,4.82212,1.88311,0.856234,0.533187\}$ & \multirow{2}{*}{173.08} \\
& $\boldsymbol{\beta}=\{0.000371046,0.000486682,0.000753751,0.00125532,0.00214748,0.00370786,0.00642182,0.0111329,0.0193026,0.0334767,0.0557995,0.099597,0.166751,0.259115,0.339681\}$ & \\
\hline
\multirow{2}{*}{650} & $\boldsymbol{\omega}=\{160196.,98524.8,44400.5,17301.1,6363.,2290.64,818.358,291.664,103.952,37.1391,13.5182,4.97644,1.92159,0.863999,0.5338\}$ & \multirow{2}{*}{185.41} \\
& $\boldsymbol{\beta}=\{0.000342267,0.000451462,0.000704454,0.00118179,0.00203555,0.00353769,0.00616654,0.0107586,0.0187724,0.0327614,0.0551766,0.0987156,0.166373,0.260234,0.342789\}$ & \\
\hline
\multirow{2}{*}{700} & $\boldsymbol{\omega}=\{185548.,113131.,50380.6,19401.8,7057.24,2513.9,888.92,313.598,110.635,39.1228,14.0774,5.12526,1.95835,0.871356,0.534377\}$ & \multirow{2}{*}{197.58} \\
& $\boldsymbol{\beta}=\{0.000317542,0.000421048,0.000661567,0.00111734,0.00193672,0.00338637,0.00593793,0.010421,0.0182905,0.0321064,0.0545736,0.0978891,0.165999,0.261253,0.345689\}$ & \\
\hline
\multirow{2}{*}{750} & $\boldsymbol{\omega}=\{212741.,128662.,56666.5,21585.6,7771.5,2741.35,960.116,335.519,117.253,41.0685,14.6216,5.26907,1.99358,0.878346,0.534923\}$ & \multirow{2}{*}{209.60} \\
& $\boldsymbol{\beta}=\{0.000296072,0.000394511,0.000623891,0.00106032,0.0018487,0.00325072,0.00573167,0.0101143,0.0178497,0.0315029,0.0539927,0.0971112,0.165631,0.262187,0.348405\}$ & \\
\hline
\multirow{2}{*}{800} & $\boldsymbol{\omega}=\{241773.,145109.,63251.7,23849.8,8504.98,2972.78,1031.91,357.429,123.809,42.9795,15.1522,5.40832,2.0274,0.88501,0.53544\}$ & \multirow{2}{*}{221.48} \\
& $\boldsymbol{\beta}=\{0.000277257,0.00037115,0.000590511,0.00100947,0.00176971,0.00312826,0.00554433,0.00983409,0.0174443,0.030944,0.0534343,0.0963765,0.16527,0.263047,0.350959\}$ & \\
\hline
\multirow{2}{*}{850} & $\boldsymbol{\omega}=\{272642.,162462.,70130.,26192.,9256.94,3208.01,1104.27,379.329,130.308,44.8586,15.6702,5.54338,2.05996,0.891377,0.535932\}$ & \multirow{2}{*}{233.24} \\
& $\boldsymbol{\beta}=\{0.000260633,0.000350423,0.000560715,0.000963807,0.00169836,0.00301701,0.00537318,0.0095766,0.0170695,0.0304242,0.0528984,0.0956807,0.164917,0.263841,0.353369\}$ & \\
\hline
\multirow{2}{*}{900} & $\boldsymbol{\omega}=\{305344.,180712.,77296.,28610.2,10026.7,3446.87,1177.17,401.219,136.753,46.708,16.1767,5.67458,2.09136,0.897476,0.536401\}$ & \multirow{2}{*}{244.87} \\
& $\boldsymbol{\beta}=\{0.000245842,0.000331905,0.000533945,0.000922543,0.00163353,0.00291539,0.00521603,0.00933889,0.0167216,0.0299387,0.0523842,0.0950198,0.16457,0.264579,0.355649\}$ & \\
\hline
\multirow{2}{*}{950} & $\boldsymbol{\omega}=\{339879.,199853.,84744.3,31102.1,10813.7,3689.21,1250.58,423.1,143.15,48.5298,16.6725,5.80222,2.12169,0.903331,0.53685\}$ & \multirow{2}{*}{256.39} \\
& $\boldsymbol{\beta}=\{0.000232596,0.00031526,0.000509752,0.000885049,0.00157432,0.00282213,0.00507108,0.0091185,0.0163973,0.0294837,0.0518908,0.0943906,0.16423,0.265265,0.357814\}$ & \\
\hline
\multirow{2}{*}{1000} & $\boldsymbol{\omega}=\{376243.,219876.,92470.3,33666.1,11617.4,3934.9,1324.47,444.973,149.499,50.326,17.1582,5.92656,2.15105,0.908961,0.537279\}$ & \multirow{2}{*}{267.81} \\
& $\boldsymbol{\beta}=\{0.000220668,0.000300214,0.000487773,0.000850813,0.00152,0.00273615,0.00493681,0.0089134,0.016094,0.0290559,0.0514173,0.0937903,0.163898,0.265906,0.359873\}$ & \\
\hline
\hline
\multirow{2}{*}{64} & $\boldsymbol{\omega}=\{1604.55,1236.6,777.72,429.57,220.699,109.268,53.1653,25.7023,12.4395,6.06839,3.77684,2.26342,1.17188,0.697364,0.519746\}$ & \multirow{2}{*}{23.72} \\
& $\boldsymbol{\beta}=\{0.00324844,0.00375019,0.00483085,0.00665688,0.00950942,0.0138266,0.0202681,0.0298105,0.0439172,0.0661899,0.0257826,0.120006,0.167699,0.222552,0.261952\}$ & \\
\hline
\multirow{2}{*}{128} & $\boldsymbol{\omega}=\{6371.83,4666.19,2709.05,1370.56,646.134,294.622,132.367,59.1371,26.4159,11.8529,5.94244,2.89717,1.36449,0.743778,0.523863\}$ & \multirow{2}{*}{44.51} \\
& $\boldsymbol{\beta}=\{0.00167178,0.00198802,0.00268004,0.00387851,0.00580983,0.00883865,0.0135361,0.0207864,0.031965,0.0496738,0.0510383,0.113836,0.169489,0.236713,0.288096\}$ & \\
\hline
\multirow{2}{*}{256} & $\boldsymbol{\omega}=\{25234.4,17233.,9009.26,4080.06,1729.54,712.786,290.428,117.862,47.8274,19.4772,8.32423,3.5472,1.54899,0.785632,0.527435\}$ & \multirow{2}{*}{82.46} \\
& $\boldsymbol{\beta}=\{0.000861253,0.0010671,0.00152771,0.00235298,0.00373991,0.00601951,0.00973541,0.0157724,0.0255664,0.0415604,0.0589968,0.108078,0.169214,0.246607,0.308902\}$ & \\
\hline
\multirow{2}{*}{512} & $\boldsymbol{\omega}=\{99805.2,63101.3,29545.,11959.4,4558.78,1698.18,627.242,231.042,85.1043,31.433,11.8839,4.53525,1.81056,0.841402,0.532005\}$ & \multirow{2}{*}{150.96} \\
& $\boldsymbol{\beta}=\{0.000435073,0.000564418,0.000861316,0.00141386,0.00238609,0.00406664,0.00695417,0.0119047,0.0203829,0.0349166,0.0569284,0.101306,0.167425,0.256851,0.333604\}$ & \\
\hline
\multirow{2}{*}{$1024$} & $\boldsymbol{\omega}=\{394347.,229799.,96276.,34921.9,12008.9,4053.99,1360.11,455.47,152.531,51.1795,17.388,5.98513,2.16481,0.91159,0.537479\}$ & \multirow{2}{*}{273.25} \\
& $\boldsymbol{\beta}=\{0.000215354,0.000293494,0.00047792,0.00083541,0.00149547,0.0026972,0.00487579,0.00881986,0.0159551,0.0288594,0.0511968,0.0935117,0.163741,0.266199,0.360827\}$ & \\
\hline
\multirow{2}{*}{$2048$} & $\boldsymbol{\omega}=\{1556575.,832736.,312142.,101721.,31639.4,9698.49,2959.69,902.095,274.961,83.9203,25.799,8.03399,2.62374,0.99542,0.543653\}$ & \multirow{2}{*}{489.09} \\
& $\boldsymbol{\beta}=\{0.000104643,0.000150277,0.000261282,0.000485959,0.000922098,0.0017595,0.00336258,0.00642874,0.0122909,0.0234935,0.0445452,0.0851626,0.158288,0.273725,0.38902\}$ & \\
\hline
\end{tabular}
\label{table:weiP15_1}
\end{center}
\end{table}

\subsection{New SRJ optimal schemes}
\label{sec:newSRJ}

After verifying that we recover the optimal parameters computed in YM14, we have improved on their results computing the
optimal values of SRJ schemes with $P>5$. In \ref{sec:P2_P13} we provide the
Tables\,\ref{table:weiP6_1}~to~\ref{table:weiP13_2}, corresponding to the optimal parameters for SRJ schemes with
$P=6,\ldots,13$ and various resolutions for the Laplace problem (Eq.\,\ref{eq:Laplace}). In Tables \ref{table:weiP14_1},
and \ref{table:weiP15_1}, we show the optimal solution parameters for $P=14$ and $P=15$.%
\footnote{All the files needed to efficiently implement any of the SRJ algorithms shown in this paper can be found
    at \texttt{http://www.uv.es/camap/SRJ.html}.}
We encountered that
finding optimal parameters at low resolution is increasingly more difficult as the number of levels increases. Indeed,
as we can see, for $P=15$ the minimum value of $N$ we have been able to compute is 64. The reason for the inability of
the proposed method to find optimal parameters for low $N$ and large $P$ is that larger values of $P$ imply that the
results are extremely sensitive to tiny changes in the smaller wave numbers (i.e., to the values of $\kappa_i$ close to
$\kappa_m$), and small numerical errors prevent a full evaluation of the solution of the non-linear system ${\cal S}$,
unless the (guessed) initial values are very close to the optimal ones.
\begin{figure*}[h]
\centering
\includegraphics[width=0.9\textwidth]{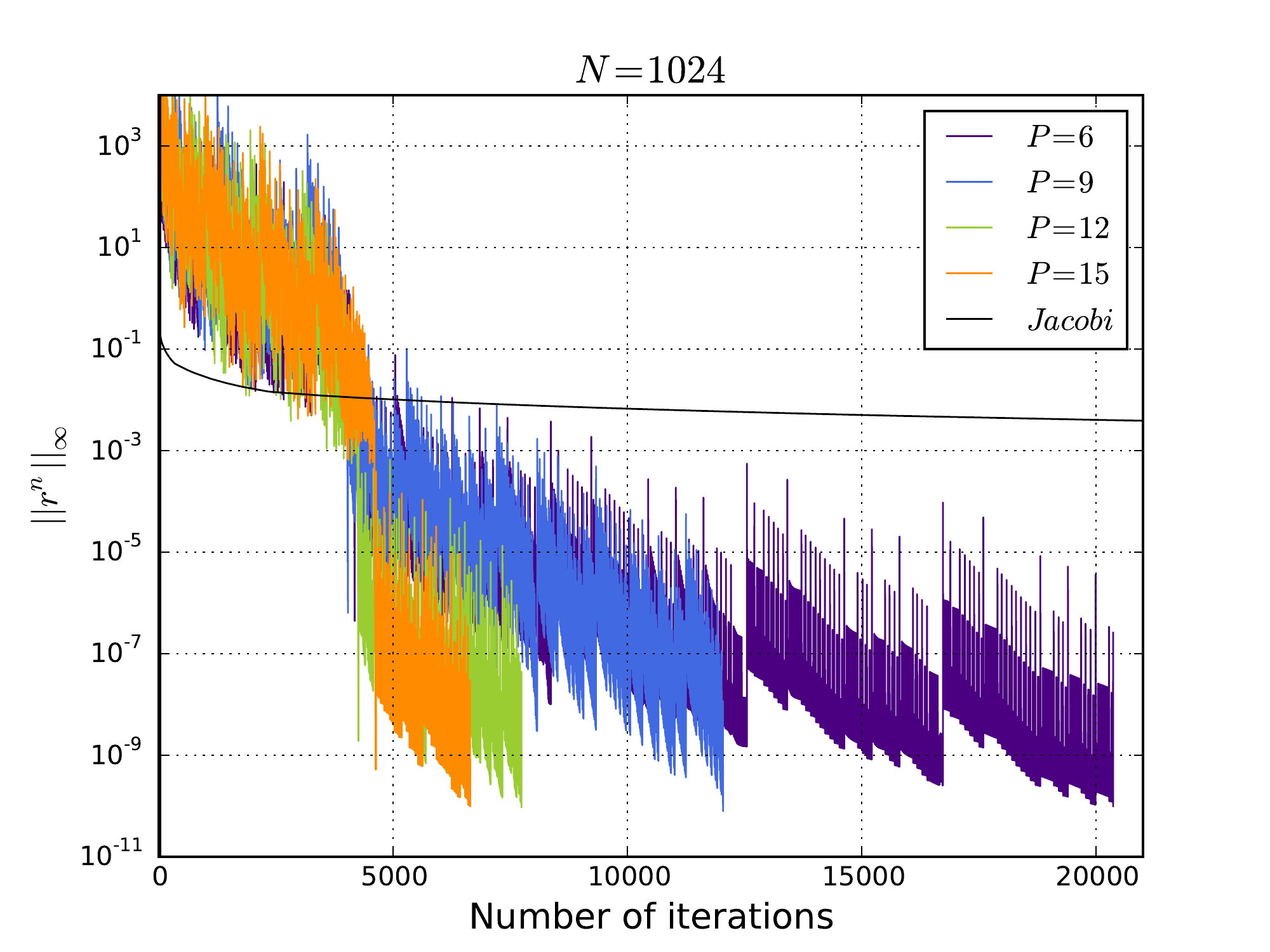} \\
\caption{Comparison of the residual evolution for the optimal SRJ
  schemes with $P=6, 9, 12$ and 15 and $N=1024$ points per
  dimension. For reference, we also include the evolution of the
  residual for the Jacobi method (black line).  }
\label{fig:P6toP15}
\end{figure*}

We remark that thanks to the improvements done (Sect.~\ref{subsec:preImpVar})
and specially with the analytic solution of a part of the unknowns of
the system (Sect.~\ref{subsec:analyticalManipulations}), not only
the optimal solution is achievable, but also it is reachable with a
moderate computational cost: employing Mathematica on a standard
workstation, the computational time of the optimal parameters ranges
from tenths of a second for the $P=6$ scheme to tens of seconds for
$P=15$.

In Fig.\,\ref{fig:P6toP15} we show the evolution of the residual for
some of the new optimal SRJ schemes solving the model problem used
throughout the paper. These new schemes show a progressively larger
efficiency as $P$ grows. A good proxy for the performance of the
method is the convergence performance index, which grows with the
number of levels. We achieve a reduction in the time of computation to
solve the problem because of the reduction in the number of iterations
to reach convergence.  This reduction is roughly proportional to
$P\log_{10}(P+1)$. However, the rate of reduction of the error is non
monotonic. For large values of $P$ (namely, $P\ge 12$), a direct inspection of
Fig.\,\ref{fig:P6toP15} shows a faster decline of the residual once
any given SRJ method reduces its residual below the one
corresponding to the Jacobi method (in this case, this happens after
about 4.500 iterations).

\subsection{Obtention of the integer parameters $\boldsymbol{q}$} 
\label{sec:real2integer}

There is a step in the practical implementation of SRJ methods that
may impact on the performance of the resulting algorithm, measured by
the number of iterations needed to reduce the residual below a
prescribed tolerance.  Once the solution has been found and we know
the \emph{real} values of $\boldsymbol{\omega}$ and
$\boldsymbol{\beta}$, one must obtain the \emph{integer} values of
$\boldsymbol{q}$. The conversion to integer begins by defining
$\boldsymbol{\bar{\beta}}:=\frac{\boldsymbol{\beta}}{\beta_1}$, so
that $\beta_1 = q_1=1$. For the conversion to integer of the rest of
the $\bar{\beta}_i$ ($i=2,\ldots,P$), we have tested several
possibilities, including the floor $\boldsymbol{q} = \lfloor
\boldsymbol{\bar{\beta}} \rfloor$, rounding $\boldsymbol{\bar{\beta}}$
to the nearest integer, taking the ceiling function $\boldsymbol{q} =
\lceil \boldsymbol{\bar{\beta}} \rceil$, or combinations of the former
alternatives, since it is possible to apply different recipes for
every $\beta_i / \beta_1$. Each of these alternatives may yield a
different number of iterations to reach convergence (see below). 
\begin{figure}[H]
\centering
\begin{tabular}{cc}
\includegraphics[width=0.48\textwidth]{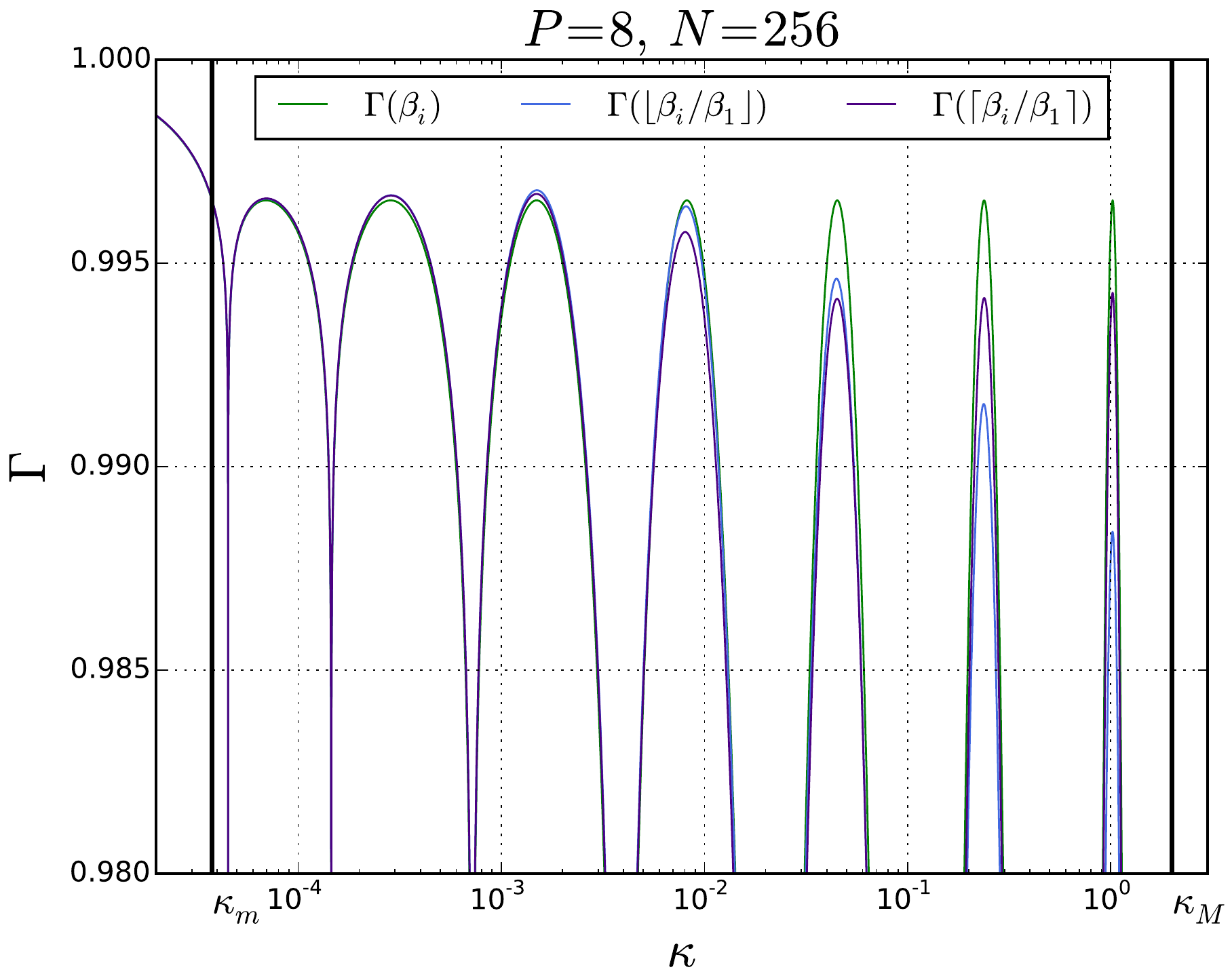} &
\includegraphics[width=0.48\textwidth]{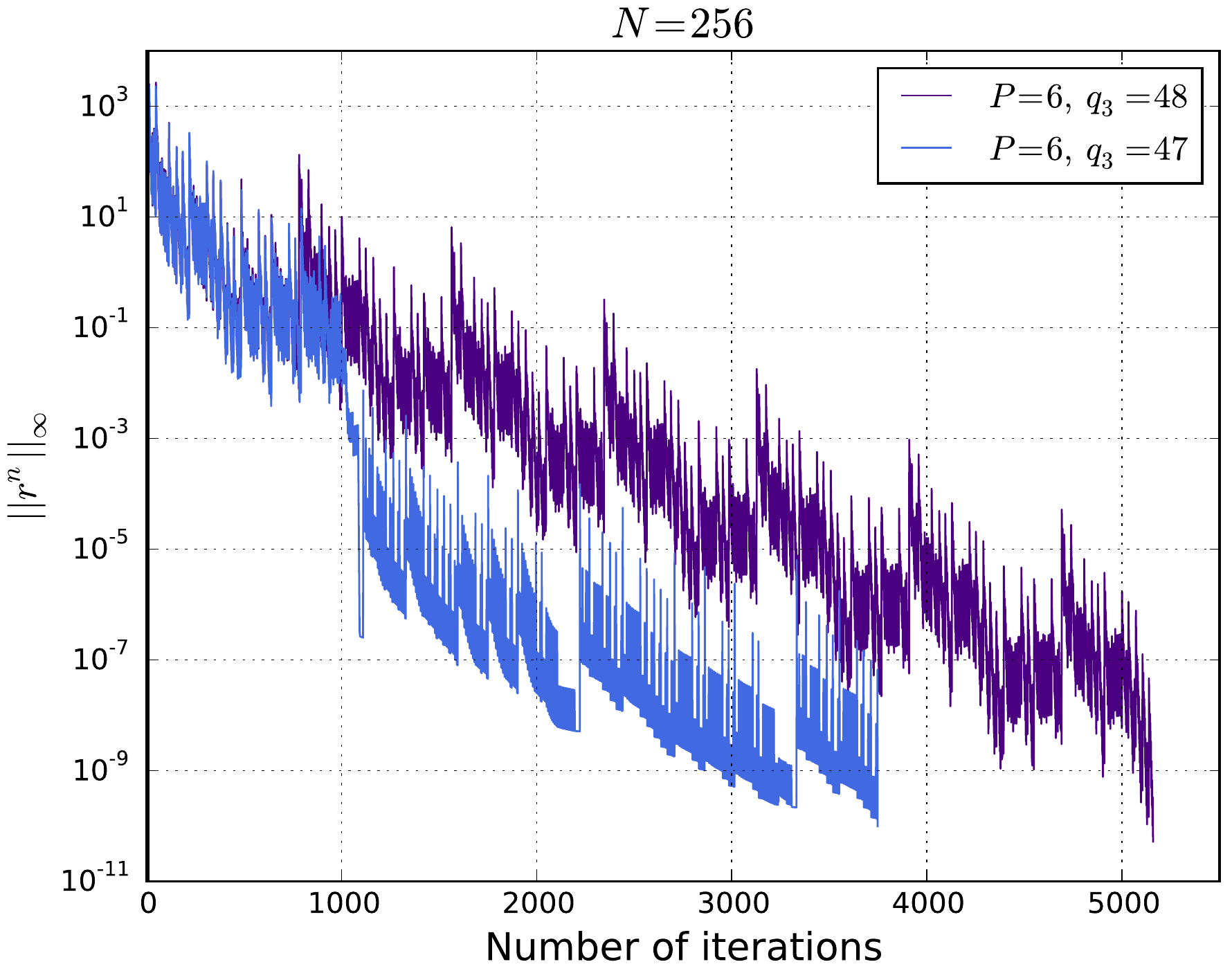} 
\end{tabular}
\caption{Comparison of the evolution of the residual of Eq.\,(\ref{eq:2nd-order_CD}) of two SRJ schemes having $P=6$
  and $N=256$ zones per dimension. Top: mean amplification factor per $M$-cycle. Bottom: evolution of the
  residual as a function of the number of iterations. The different variants of the $P=6$ SRJ method are displayed with
  different color lines showing the dependence of the performance of the method on the conversion of the real values of
  the optimal solution for $\beta_i$ to the integer values $q_i=\lfloor \beta_i / \beta_1 \rfloor$ (blue) and
  $q_i=\lceil \beta_i / \beta_1 \rceil$ (purple).  }
\label{fig:f_flo_to_int}
\end{figure}
After
computing the integer values of $\boldsymbol{q}$, a key point to
account for is that the $\Gamma(\kappa)$ function must remain below
1, since otherwise our method diverges. In Fig.~\ref
{fig:f_flo_to_int} (upper panel) we observe that the amplification
factor per $M$-cycle may change by more than 10\%, for values of
$\kappa$ close to $\kappa_M$, depending on the method adopted to
convert $\boldsymbol{\beta}$ to integer.

While the number of levels is small, the differences among the
distinct conversions from real to integer do not change much either
the number of iterations or the convergence rate of the resulting
scheme. However, when $P$ increases, there can be non-negligible
changes in the total number of iterations to reduce the residual of
our model equation below a prescribed tolerance. In the lower panel of the
Fig.~\ref{fig:f_flo_to_int} we show the evolution of the residual as a
function of the iteration number for two different choices of the
integer conversion of $\boldsymbol{\beta}$ into $\boldsymbol{q}$ in
the case $P=6$ (the optimal parameters of which can be found on
Tab.~\ref{table:weiP6_2}). We note that there is a difference of more
than 1200 iterations ($\sim 25\%$) between the distinct integer
conversions. Unfortunately, changing the number of levels, the same
recipes for converting reals to integers yield efficiencies of the
methods that do not display a clear trend. Fortunately, increasing the
number of levels by one unit results in a reduction of the number of
iterations to reach convergence which is larger than that resulting
from any manipulation of the integer values of $q_i$ in an SRJ scheme
with a given $P$. Hence, in the following, the results we will provide
are obtained by taking simply $q_i = \lfloor \beta_i / \beta_1
\rfloor$.

\section{Numerical examples}
\label{sec:examples}

\subsection{Poisson equation with Dirichlet boundaries}
\label{subsec:casestudy}

So far we have considered only the application of SRJ schemes to the
solution of the Laplace equation with homogeneous Neumann boundary
conditions (Eq.\,\ref{eq:Laplace}). In this section we consider a case
study consisting on solving a Poisson equation in two dimensions
endowed with Dirichlet boundary conditions. The exact problem setting
reads
\begin{eqnarray}
\frac{\partial^2}{\partial x^2} u(x,y) + \frac{\partial^2}{\partial y^2} u(x,y) =- e^{x y} (x^2 + y^2), & & (x,y) \in (0,1) \times (0,1),\nonumber \\
u(0,y) = -1,\,\, u(1,y) = -e^y, & & y \in [0, 1], \\ 
u(x,0) = -1,\,\, u(x,1) = -e^x, & & x \in [0, 1], \nonumber  
\label{eq:toyProblem}
\end{eqnarray}
which has analytic solution:
\begin{equation}
u(x,y) = -e^{x y}.
\end{equation}
This kind of problem will help us to assess whether the change in the boundary conditions affects the efficiency of an
SRJ scheme.  Imposing Dirichlet boundary conditions is typically more challenging than dealing with Neumann ones, since
Dirichlet boundary conditions change the value of $\kappa_m$ so that (YM14)
\begin{equation}
\kappa_{m,{\rm Dirichlet}} = \sin^2 \left(\frac{\pi}{2N_x}\right) + \sin^2 \left(\frac{\pi}{2N_y}\right)\, ,
\end{equation}
to be compared with Eq.\,(\ref{eq:boundaries}). Hence, the optimal SRJ values obtained for a given $N$ and Neumann
boundary conditions do not exactly coincide with those optimal in problems involving Dirichlet boundaries, hence we must
follow the recipe provided in Eq.\,(\ref{eq:effectiveNDirichlet}).
Furthermore, since in Eq.\,(\ref{eq:toyProblem}) we are considering a Poisson equation, we can test whether the presence
of source terms modifies the performance of SRJ methods.

For the case studied we choose a discretization consisting on $N_x
\times N_y = 585 \times 280$ uniform numerical zones. Although here we do not have Neumann boundary conditions, we
will use the optimal values of the SRJ scheme for this case in order to show that even though this choice is not
optimal, still it substantially speeds up the solution of the problem with respect to Jacobi. In this case, we
have that $N = 585$. If we apply the SRJ scheme with
$P=10$, we must look for the $\boldsymbol{\omega}$ and the
$\boldsymbol{\beta}$ parameters in Tab.\,\ref{table:weiP10_1}. In this
case, the table does not provide an entry for $N=585$, but as YM14
point out, we can chose as (non-optimal) parameters for the SRJ scheme
those corresponding to a smaller resolution.\footnote{It is, however,
  possible to compute the optimal values for $N=585$ employing our
  algorithm.} In our case, the closest resolution that matches this
criterion in Tab.\,\ref{table:weiP10_1} is that corresponding to the
row with $N=550$.

We use the simplest way to obtain the $q_i$ from the $\beta_i$
ensuring convergence, namely $q_i = \Big \lfloor
\frac{\beta_i}{\beta_1} \Big \rfloor$ with $1 \leq i \leq P$,
resulting in 
\begin{equation}
\boldsymbol{q} = \{1,1,3,9,21,49,116,268,587,1014\}, \, M = 2069.
\end{equation}

\begin{figure*}[h]
\centering
\includegraphics[width= 0.75\textwidth]{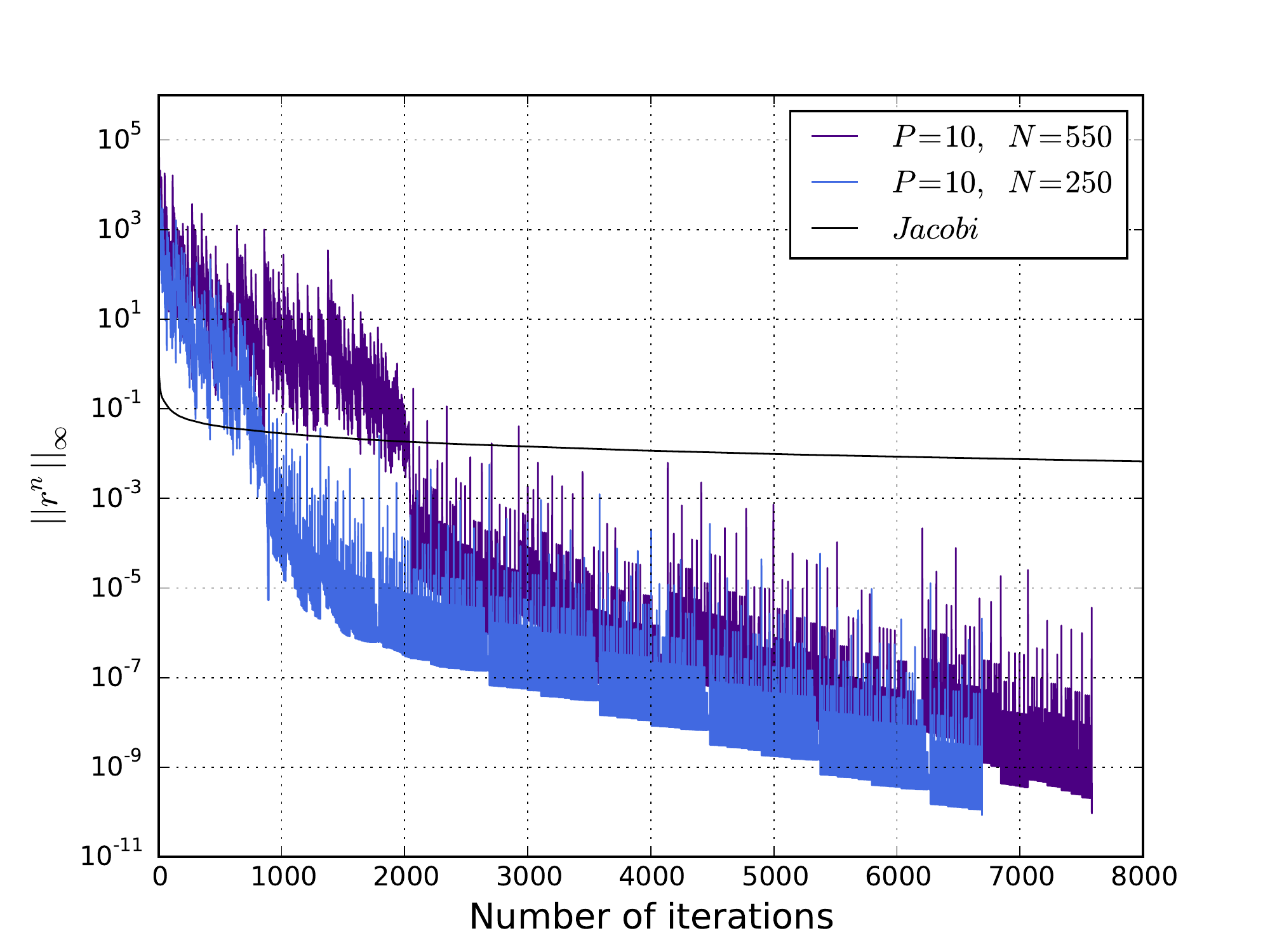} \\
\caption{Evolution of the residual as a function of the iteration number for different SRJ schemes. With violet line we
  show the case in which the parameters of an SRJ scheme with $P=10$, $N=550$ and Neumann boundary conditions are chosen
  to compute the solution of the problem stated in Eq.\,\eqref{eq:toyProblem}, on a grid of $N_x \times N_y = 585 \times
  280$ and {\em Dirichlet} boundary conditions. With blue line we show the case in which the the parameters of an SRJ
  scheme with $P=10$, $N=250$ and Neumann boundary conditions are used. The latter case corresponds to the closest value
  of $N$ to the effective resolution $N_{\rm eff}^{\rm(2)}=252.56$ that shall be used when Dirichlet boundaries (instead
  of Neumann ones) are used. For comparison, we also display the evolution of the residual for the Jacobi
  method.} 
\label{fig:example}
\end{figure*}

This means that we will use $\omega_1 = 106105$ and $\omega_2 =
40577.2$ once per $M$-cycle, $\omega_3 = 10230.6$ three times per
$M$-cycle, etc. In practice, it is necessary to distribute the
largest over-relaxation steps over the whole $M$-cycle to prevent the
occurrence of overflows. YM14 provide a Matlab script that generates a
schedule for the distribution of $\omega_i$ on the $M$-cycle that
guarantees the absence of overflows. We find that an even
distribution of the over-relaxations over the entire $M$-cycle is
sufficient in order to avoid overflows.

In Fig.~\ref{fig:example} we plot the evolution of the residual as a
function of the number of iterations for a SRJ scheme with $P=10,
N=550$ (instead of $N=585$), as well as the residual evolution
employing the Jacobi method for the solution of
Eq.~\eqref{eq:toyProblem}. This example shows that even picking an SRJ
scheme whose parameters are non-optimal for the problem size at hand
($N=585$ in this case), for the presence of source terms and for the kind of boundary conditions specified (Dirichlet for the problem at hand), we can largely speed up the convergence with
respect to the Jacobi method. Theoretically, for the optimal $P=10$,
$N=550$ SRJ method, an acceleration of the order of $\rho =125.85$
with respect to the Jacobi method is expected, something confirmed with our numerical results (Fig.~\ref{fig:example}). 

Finally, we briefly illustrate how to find the optimal scheme for Dirichlet boundaries (which are, indeed, the ones with
which the problem at hand is set). According to Eq.\,(\ref{eq:effectiveNDirichlet}), $N_{\rm eff}^{(2)} =
252.56$. Therefore, we choose the row with entry $N=250$ and we proceed as we have done previously. In
Fig.\,\ref{fig:example} we can see that the results are even better than for the (non-optimal) set of parameters with
$P=10$, $N=550$ and Neumann boundary conditions, since in the case with $N=250$, the accuracy goal is reached with $\sim
13\%$ less iterations.

\subsection{Poisson equation in spherical coordinates}
\label{subsec:spherical}

The Poisson equation appears, among others, in problems involving gravity, either Newtonian
or some approximations to General Relativity, and electrostatics. In numerical 
simulations, e.g., in Astrophysics and Cosmology, the computation of the gravitational potential 
is usually coupled to a hyperbolic set of equations describing the dynamics
of the fluid, e.g., the Euler equations. In those cases the Poisson equation is solved
on each time step (or every several time steps) of the evolution of the hyperbolic part.
It is thus crucial to test the efficiency of the SRJ compared with other methods 
currently used by the scientific community. In simulations of stellar interiors, spherical
coordinates is a popular choice of coordinates, so we adopt it for our test.
To mimic typical astrophysical scenarios we have chosen a test in which the 
source has compact support and boundary conditions are applied at radial infinity.

The Poisson equation in spherical coordinates $(r, \theta, \varphi)$ reads
\begin{equation}
  \frac{\partial^2 u}{\partial r^2}
  + \frac{2}{r} \frac{\partial u}{\partial r}
  + \frac{1}{r^2} \frac{\partial^2 u}{\partial \theta^2}
  + \frac{\cot{\theta}}{r^2} \frac{\partial u}{\partial \theta}
  + \frac{1}{r^2 \sin{\theta}} \frac{\partial^2 u}{\partial \varphi^2} = s,
  \label{eq:grav:poisson}
\end{equation}
being $u$ and $s$ functions of $(r,\theta,\varphi)$. For our test we choose
the source term to be the series
\begin{equation}
  s(r,\theta,\varphi) =
\begin{dcases} 
  - \sum_{n=0}^{n_{\rm max}} \sum_{m=-m_{\rm max}}^{m_{\rm max}} a_{2n}\, k_{2n}^2\, j_{2n} (k_{2n} r) \, Y_{2n}^{m,c}(\theta,\varphi), & {\rm for}\, r \le 1 \\
  \\ 0, & {\rm for}\, r>1, 
\end{dcases}
\label{eq:grav:source}
\end{equation}  
being $j_l$ the spherical Bessel functions of the first kind and $Y_l^{m,c}$ the real part of the spherical harmonics. Only even parity terms, $l=2n$
are considered. $k_l$ is the first
root of the spherical Bessel function of order $l$, such that $s (1,\theta,\varphi) = 0$. We chose $a_l = 1/2^{l}$,
such that the series is convergent. We impose homogeneous Neumann boundary conditions at $r=0$, $\theta=0$ and $\theta=\pi$, 
and periodic boundary conditions in the $\varphi$ direction. If we impose homogeneous Dirichlet condition at radial infinity  ($r \to \infty$), the
solution of this elliptic problem is
\begin{equation}
  u (r,\theta,\varphi) = 
  \begin{dcases}
    \sum_{n=0}^{n_{\rm max}} \sum_{m=-m_{\rm max}}^{m_{\rm max}} \left (a_{2n} \, j_{2n} (k_{2n} r) + b_{2n} r^{2n} \right ) \, Y_{2n}^{m, c}(\theta,\varphi), &{\rm for} \, r \le 1 \\
    \sum_{n=0}^{n_{\rm max}} \sum_{m=-m_{\rm max}}^{m_{\rm max}} \frac{c_{2n}}{r^{2n+1}} \, Y_{2n}^{m,c}(\theta,\varphi), & {\rm for} \, r > 1,
  \end{dcases}
  \label{eq:grav:sol}
\end{equation}  
where the coefficients $c_l$ and $b_l$ can be computed imposing continuity of u and its first derivatives at $r=1$, resulting in
\begin{equation}
b_l = c_l = - \frac{a_l}{2 l + 1} \partial_r j_l(k_l r)|_{r=1} = - \frac{a_l}{2 l + 1} \left [ l \,j_l(k_l) - k_l \,j_{l+1} (k_l) \right]
\end{equation}

Since our interest is to assess the performance of the SRJ method under the conditions which are found on real
applications, we solve this equation numerically in the domain $r\in[0,1]$, $\theta \in [0,\pi]$ and $\varphi \in [0,2
\pi]$, i.e. only in the region where the sources are non zero, and apply Dirichlet boundary conditions at $r=1$, using
the analytical solution given by Eq.~(\ref{eq:grav:sol}). 

We emphasize that this problem set up includes boundary
conditions of mixed type (Neumann and Dirichlet) and, hence, none of the schemes whose optimal parameters have been
tabulated in this paper is strictly optimal. However, as we shall see, even in such conditions, the new schemes
presented in this paper with $P=15$ will be competitive with other alternatives in the literature.

We set up three versions of the test with different dimensionality. In the 3D test, we choose $n_{\rm max} = \infty$ and $m_{\rm max}=2n$ and solve
the equation in the domain  $r\in[0,1]$, $\theta \in [0,\pi]$ and $\varphi \in [0,2 \pi]$, discretized in an equidistant grid of size $N\times N\times N$ points.
In the 2D case we consider axisymmetry, i.e., no $\varphi$ dependence in $u$ and $s$. We choose $n_{\rm max} = \infty$ and $m_{\rm max}=0$
and solve in the domain $r\in[0,1]$ and $\theta \in [0,\pi]$, discretized in a grid with $N\times N$ points. In the 1D case we consider spherical symmetry, i.e. no $\theta$
or $\varphi$ dependence in $u$ and $s$. We choose $n_{\rm max} =0$, $m_{\rm max} =0$  and solve in the domain $r\in[0,1]$ with $N$ points.
Since the series in Eqs.\,(\ref{eq:grav:source}) and (\ref{eq:grav:sol}) are convergent, we compute them numerically by adding terms until the last
significant digit does not change. We use a second order finite difference discretization for Eq.~(\ref{eq:grav:poisson}) and one ghost cell in
each direction to impose boundary conditions. For convenience, we multiply Eq.~(\ref{eq:grav:poisson}) by $r^2$ in the discretized version.

As an example, we present explicitly the discretization of the 1D problem. The 2D and 3D discretizations are analogous to what is described here.
We use a staggered grid with ghost cells, $r_i = (i - 1/2) \Delta r$ with $i=0,...,N+1$,  where
$\Delta r = 1 / N$. Points $i=0$ and $i=N+1$ are ghost cells used only for the purpose of imposing boundary conditions. Using second order centered derivatives 
and imposing spherical symmetry ($\partial_\theta=\partial_\varphi=0$) Eq.~(\ref{eq:grav:poisson}), multiplied by $r^2$, can be discretized as
\begin{equation}
r_i^2 \frac{u_{i+1}-2 u_i+u_{i-1}}{\Delta r^2} + 2 r_i\frac{u_{i+1} - u_{i-1}}{2 \Delta r} = r_i^2 s_i, \quad (i = 1, ..., N)
\end{equation}
where sub-index $i$ indicates a function evaluated at $r_i$. By imposing boundary conditions it is possible to set
the values of $u_0$ and $u_{N+1}$. The resulting linear system of $N$ equations with $N$ unknowns, $u_i, \, i=1,...,N$,
can be written in matrix form
\begin{equation}
\sum_{j=1}^{N}\mathcal{A}_{i j} u_{j} = r_i^2 s_i , \quad (i = 1, ..., N),
\end{equation}
being $\mathcal{A}_{i j}$ the elements of the coefficient matrix, which in the 1D case is a $N\times N$ tridiagonal matrix
\begin{eqnarray}
\mathcal{A}_{i i-1} &=& \left ( r_i - 2 \Delta r\right )\frac{r_i}{ \Delta r^2} \nonumber \\
\mathcal{A}_{i i} &=& \frac{-2 r_i^2}{ \Delta r^2} \nonumber \\
\mathcal{A}_{i i+1} &=& \left ( r_i + 2 \Delta r\right )\frac{r_i}{ \Delta r^2} \nonumber\\
\mathcal{A}_{i j} &=& 0, \quad \rm{otherwise}.
\end{eqnarray}
Note that this matrix is diagonally dominant by rows and columns except for the first two rows and the first column.
If the $r^2$ factor were not present, the matrix would not be diagonally dominant by columns and the convergence of the
iterative methods could not be guaranteed. Once the boundary conditions are applied, the coefficient matrix is effectively
modified. Wether the resulting effective matrix is diagonally dominant or not depends crucially on how the boundary
conditions are applied.

We impose Dirichlet boundary conditions at the outer boundary by setting $u_{N+1} = u_{\rm analytic}(r_{N+1})$,
being the analytic solution that given by Eq.~(\ref{eq:grav:sol}). In this case the equation at $i=N$ results
\begin{equation}
\mathcal{A}_{NN-1} u_{N-1} + \mathcal{A}_{NN} u_N = r_N^2 s_N - \mathcal{A}_{NN+1} u_{\rm analytic} (r_{N+1}),
\end{equation}
being  $\mathcal{A}_{NN+1}$ an extension of the coefficient matrix used for practical purposes.
At the inner boundary we impose homogeneous Neumann conditions, i.e. $\partial_r u |_{r_0} = 0$. A
second order discretization of this boundary condition is not unique. Two possible discretizations are
of the boundary condition are
\begin{eqnarray}
\frac{u_1 - u_0}{\Delta_r} &=& 0, \label{eq:bc1} \\
\frac{u_2 - u_0}{2 \Delta_r} &=& 0. 
\end{eqnarray}
Naively, one would expect that imposing either $u_0 = u_1$ or $u_0 = u_2$, would be an appropriate boundary 
condition. However, either of these cases results in an effective coefficient matrix which is not diagonally
dominant. To cure this pathology we impose $u_0=u_1=u_2$, which is compatible with both discretizations of the boundary
condition. 
Since we are fixing not only the value of the ghost cell, $u_0$, but also the value $u_1$, this condition
reduces the dimensionality of the linear system by $1$. The resulting effective coefficient matrix $\hat{\mathcal{A}}_{ij}$
is a tridiagonal matrix of size $(N-1)\times (N-1)$, with indices $i,j=2,...N$. The elements read
\begin{align}
\hat{\mathcal{A}}_{2 2} =& - \left ( r_2 + 2 \Delta r\right )\frac{r_2}{ \Delta r^2} = -\frac{21}{4}\nonumber \\
\hat{\mathcal{A}}_{2 3} =& \left ( r_2 + 2 \Delta r\right )\frac{r_2}{ \Delta r^2} = +\frac{21}{4}\nonumber \\
\hat{\mathcal{A}}_{i j} =& \mathcal{A}_{i j}, \quad {\rm otherwise}.
\end{align}
The new matrix is diagonally dominant by rows and columns, which guarantees the convergence of Jacobi-based iterative methods.

In practice, the effective coefficient matrix, $\hat{\mathcal{A}}_{i j}$, is not used in the iterative methods directly.
Instead we use a coefficient matrix $\mathcal{A}_{i j}$ extended to the ghost cells and  we set at each iteration the values at the ghost cells
($u_0$ and $u_{N+1}$) and at $u_1$ according to prescription given above. This procedure is equivalent to using the 
effective coefficient matrix, but it eases the implementation of the algorithm.

We perform series of calculations for different values of the number of points $N$. For each calculation we use the SRJ
coefficients computed in the previous sections matching the corresponding value of $N$.  For each series we perform
calculations using coefficients computed with different values of $P$. The tolerance goal is set to $10^{-5}/N^2$, which
depends on the number of points. Since we use a second order discretization, this scaling in the tolerance ensures that
the difference between the numerical and the analytical solution is dominated by finite differencing errors and at the
same time avoids unnecessary iterations in the low resolution calculations.  This prescription for the tolerance mimics
the tolerance choice that is used under realistic conditions and renders a fairer comparison in the computational cost
between different resolutions. For comparison we also perform calculations using other iterative methods: Jacobi,
Gauss-Seidel and SOR (weight equal to $1.9$). For each case involving iterative methods we perform two calculations: the
{\it ab initio calculation} in which the solution is initialized to zero, and the {\it realistic calculation} in which
the solution is initialized to $u_{\rm analytic} (1 + ran(-0.5,0.5)/N)$, being $ran(-0.5,0.5)$ a random number in the
interval $[-0.5,0.5]$. The {\it realistic calculation} tries to mimic the conditions encountered in many numerical
simulations in which an elliptic equation (or a system of PDEs) is solved coupled with evolutionary (hyperbolic) PDEs,\footnote{For instance, this is the case of the fluid equations.} which are typically solved using explicit methods, whose time
step is limited by the Courant-Friedrichs-Lewy (CFL) condition. This means that the change in the source of the Poisson
equation between subsequent calculations is ${\mathcal O}(\Delta x)$; therefore, if the solution of the previous time
step is used for the iteration in the elliptic solver, this should differ only ${\mathcal O}(\Delta x) \sim 1/N$, from
the solution.

In addition to iterative methods, we perform the calculations using a direct inversion method and spectral methods.  In
the {\it direct inversion method}, we compute the LU factorization of the matrix associated with the coefficients of the
discretized version of the equation by performing Gaussian elimination using partial pivoting with row interchanges. We
use the implementation in the {\tt dgbtrf} routine of the {\tt LAPACK} library \citep{lapack}, which allows for
efficient storage of the matrix coefficient in bands. Once the LU decomposition is known, we solve the system of linear
equations using the {\tt dgbtrs} routine. Since this method is non-iterative, its computational cost does not depend on
the initial value. However, this approach has advantages when used repeatedly, coupled to evolution equations (e.g., for
a fluid). Most of the computational cost of this method is due to the LU decomposition, but once it has been performed,
solving the linear system for different values of the sources is computationally less intensive. Therefore, we consider
the computational cost of the whole process, LU decomposition and solution of the system, in the {\it ab initio
  calculations}, while only the solution of the linear system, assuming the LU decomposition is given, in the {\it
  realistic calculations}.
\begin{figure*}[p!]
\centering
\includegraphics[width= 0.49\textwidth]{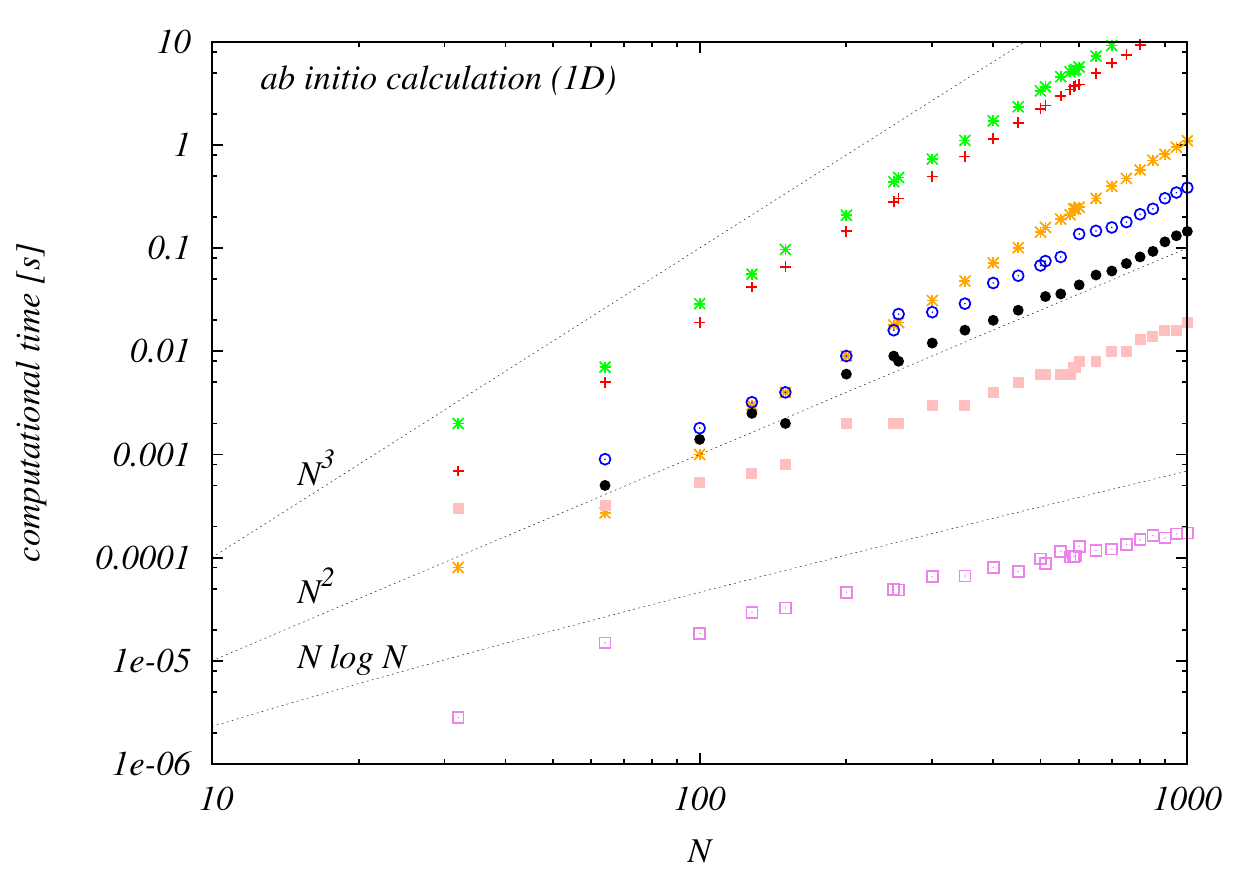} 
\includegraphics[width= 0.49\textwidth]{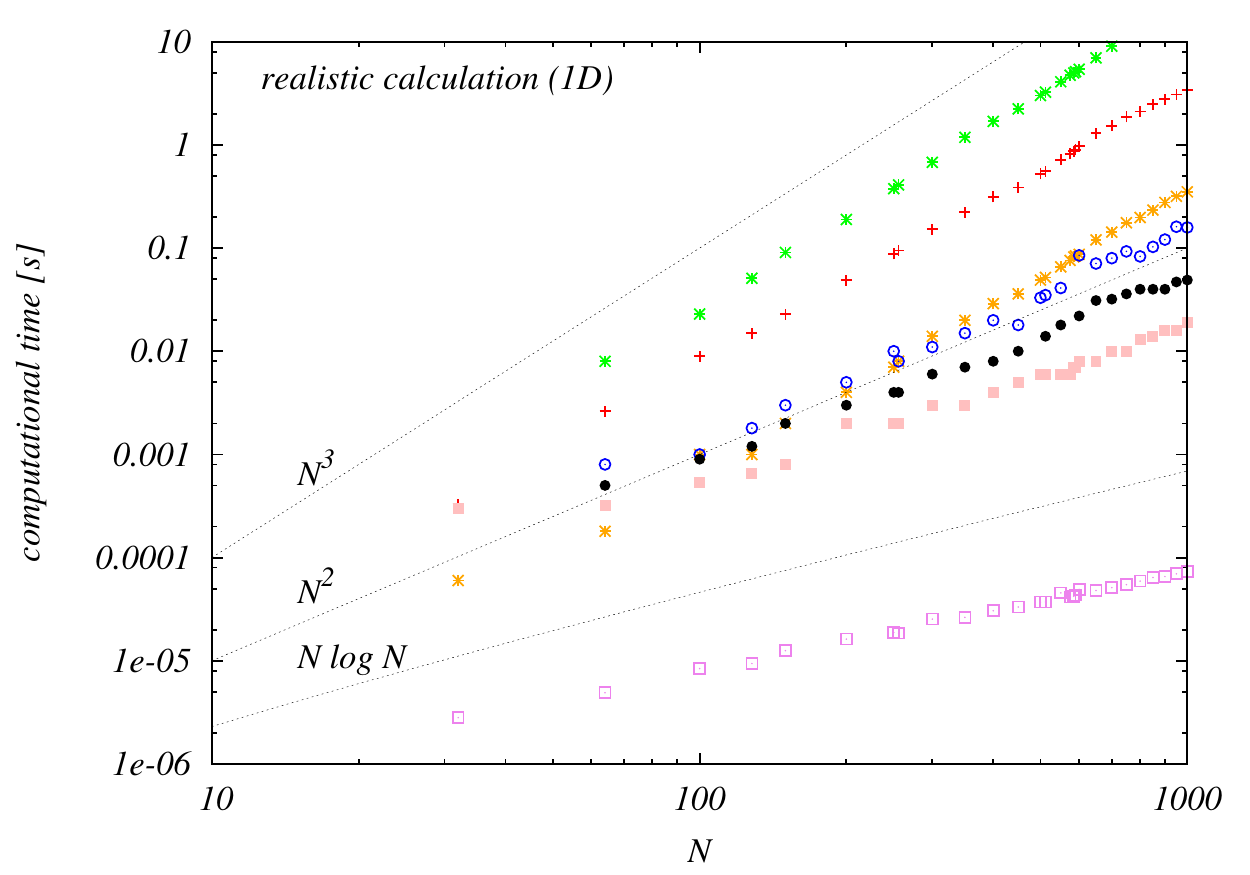} \\
\includegraphics[width= 0.49\textwidth]{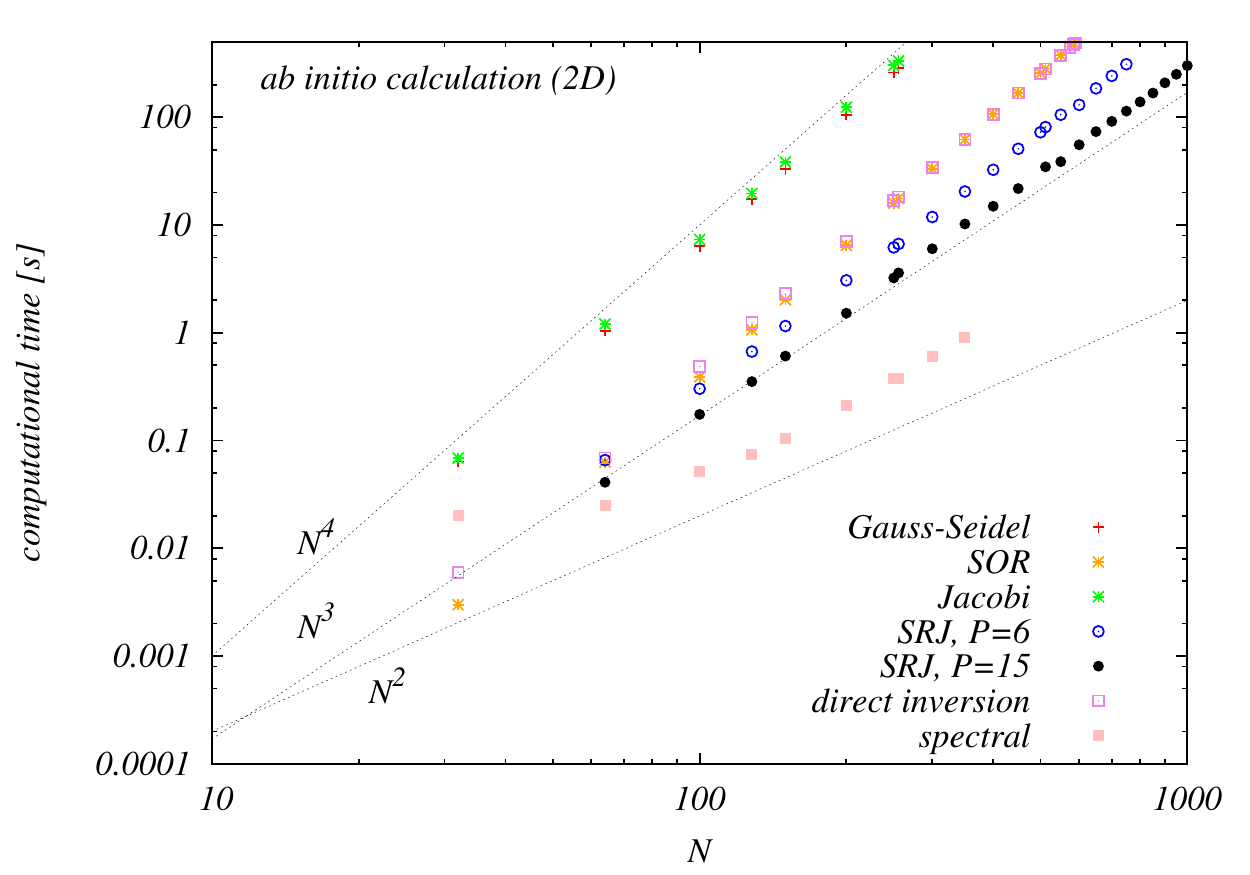} 
\includegraphics[width= 0.49\textwidth]{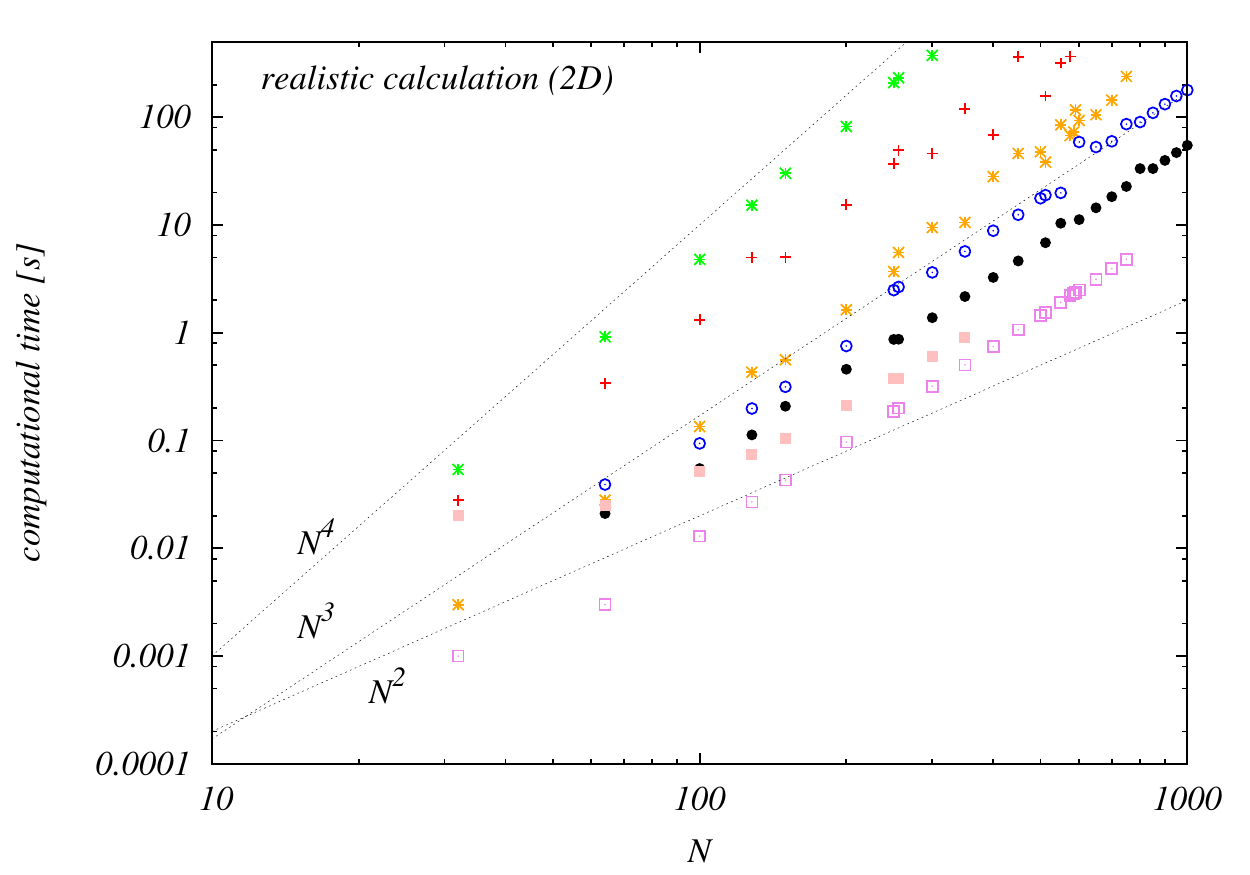} \\
\includegraphics[width= 0.49\textwidth]{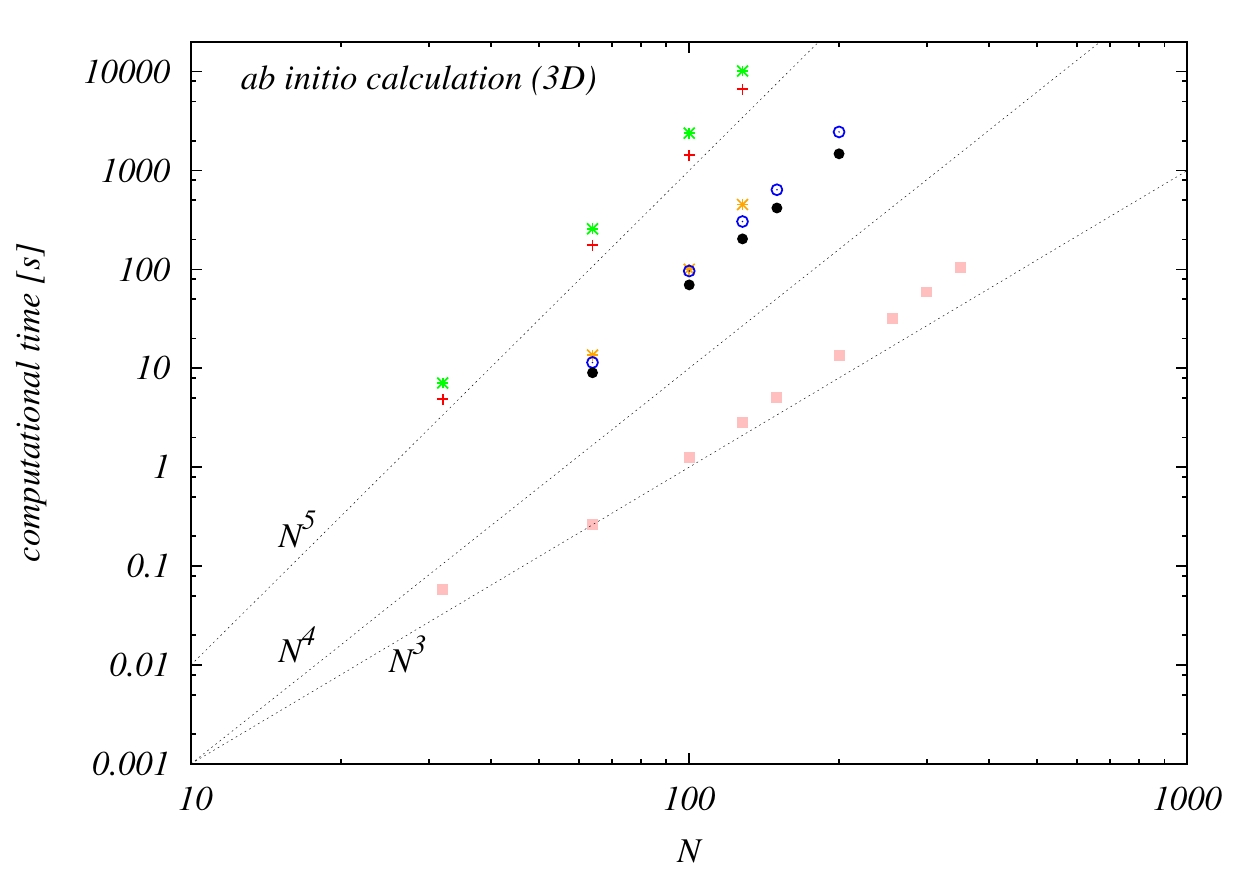} 
\includegraphics[width= 0.49\textwidth]{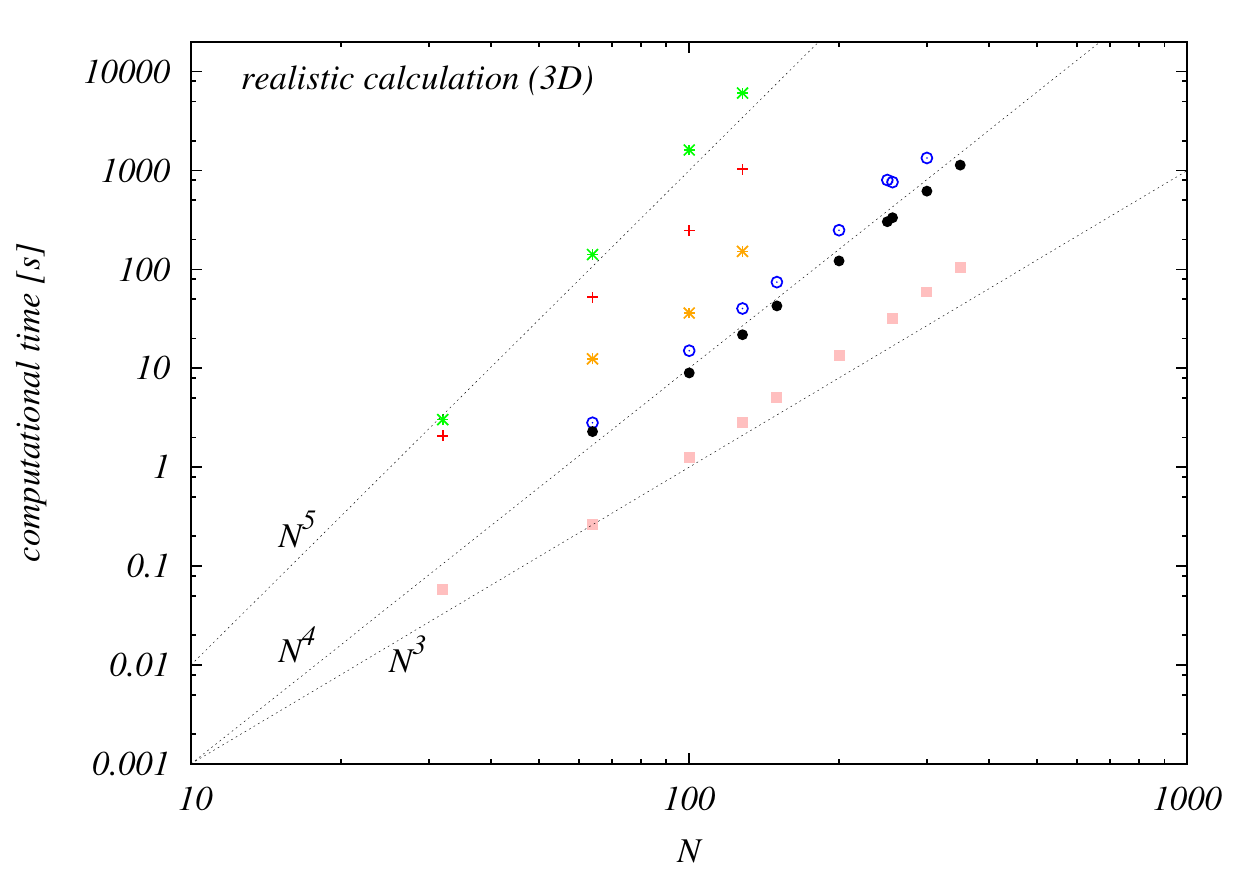} \\
\caption{Computational cost of the solution of the Poisson equation in spherical coordinates, depending on the size of the problem $N$,
using different numerical methods, including SRJ for the minimum ($P=6$) and maximum ($P=15$) set of coefficients computed in this work. 
Upper, middle and lower panels show the 1D, 2D and 3D test respectively. 
Right and left column panels show {\it ab initio} and {\it realistic
  calculations} respectively. 
}
\label{fig:grav:cpu}
\end{figure*}

For the spectral solver we use the {\tt LORENE} library \cite{lorene}. To provide results that are comparable to all other numerical methods used
in the present work we use the following procedure: first we evaluate the source, $s$, at the {\it finite differences grid} used in all other 
numerical methods; then, the source is interpolated to the collocation points in the {\it spectral grid}, which do not coincide with the 
{\it finite differences grid}; next the solution is computed by means of the {\tt LORENE} library; finally the function is evaluated at 
the cells of the {\it finite differences grid}. The details of the procedure are described in
\cite{Dimmelmeier2005}. The accuracy of the numerical 
method is dominated by the second order finite differences discretization error associated with the {\it finite differences grid}, provided 
sufficient number of collocation points are used in the {\it spectral grid}. We have tested that it is sufficient to use $N/2$ collocation points 
per dimension to fulfill this requirement. When using the spectral solver, there is no difference in the computational cost in {\it ab initio} or 
{\it realistic calculations}.

We have performed the calculations using a 3.4 GHz Intel core i7 and 16 GB of memory. All codes and libraries have been compiled using the {\tt GCC}
compiler \citep{gcc}. We have measured the computational time for each method timing exclusively the part of the code involved in the computation and not
the allocation and initialization of variables. Figure~\ref{fig:grav:cpu} shows the dependence of the computational time for 1D, 2D and 3D tests
in the {\it ab initio calculations} and the {\it realistic calculations} setups. 
As expected, for any dimensionality,  the SRJ methods render a significant speed up with respect to other iterative methods, due to the smaller 
number of iterations needed. Only SOR method has comparable computational time for low resolutions ($N < 100$). The computational time for SRJ
methods scales approximately as $N^{d+1}$, being $d$ the dimensionality of the test, i.e. the number of iterations is proportional to $N$. In 
comparison, the computational cost of other iterative methods (Jacobi, Gauss-Seidel, SOR), scale approximately as $N^{d+2}$, i.e. the number
of iterations needed scales as $N^2$. This factor $N$ improvement of SRJ with respect other iterative 
methods ensures that the method will always be less costly for sufficient high resolution. Compared to non-iterative methods the results 
depend on the dimensionality of the test. 

For the 1D test, both spectral and the direct inversion method are significantly faster than SRJ.
The computational cost of both methods are close to $N \log N$. Therefore, we conclude that SRJ methods are not competitive for 1D problems, even 
when {\it realistic} conditions are considered. 

In 2D, the computational cost of the direct inversion method for {\it ab initio calculations} increases significantly,
scaling as $N^4$, because the associated matrix is not tridiagonal anymore, as in the 1D case, but is a banded matrix of
band size $2N+1$. Hence, the direct inversion method is more costly than SRJ for resolutions $N>100$. However, in the
{\it realistic calculation}, in which the LU decomposition is not performed, the direct inversion method is still the
fastest, with a computational time scaling as $N^3$ (the same as SRJ) but with lower computational cost. Due to
limitations of the {\tt LORENE} library, we were not able to perform multidimensional computations using spectral
methods for $N>350$.  Compared to $P=15$, spectral methods are about a factor $2$ faster than SRJ in {\it realistic
  calculations} and become comparable for $N<100$. It seems fair to say that spectral methods perform better for {\it ab
  initio calculations}, since in this case, SRJ methods (see $P=6$ and $P=15$ in Figure~\ref{fig:grav:cpu}) scale as
$N^3$. Therefore, we conclude that for 2D calculations SRJ is a competitive method, when compared with spectral
methods. Although the direct inversion method is the fastest in the range of values of $N$ selected for our tests, we
expect that this advantage will disappear when going to larger number of points; the memory needed for the direct
inversion method scales as $N^3$ (due to the explicit use of the banded structure of the matrix) in comparison with
$N^2$ of all other methods (iterative and spectral). This strongly limits the size of the problem to be solved without
using parallelization.

In 3D all computations are significantly more costly, so we limit our
tests to what is achievable within $\sim 1$~hour of computation
time. For the SRJ methods tested this is $N \le 200$.  For $N>100$ the
computational cost of spectral methods is a factor $\approx 10$ lower
than a SRJ method with $P=15$, in the {\it realistic calculation}. Using the SRJ parameters for $P=15$ and an
  effective number of points per dimension as given by Eq.\,(\ref{eq:effectiveN}), the SRJ method becomes $\sim 20\%$
  faster, so that it is ``only'' $\approx 8$ slower than the spectral method. The
conclusion is that spectral methods still seem to have advantages over
SRJ methods, for the 3D test presented. However, both spectral and SRJ
methods scale approximately as $N^4$ in 3D. Due to the large amount of
memory needed for the direct inversion method, which scales as $N^5$,
we did not present any such calculation for the 3D case. In practice,
this limitation makes the direct inversion method unfeasible for
computations in a single CPU. The performance of all these methods and
a comparison between them in a parallel architecture is beyond the
scope of this work.

\begin{figure*}[p!]
\raggedleft
\includegraphics[width= 0.49\textwidth]{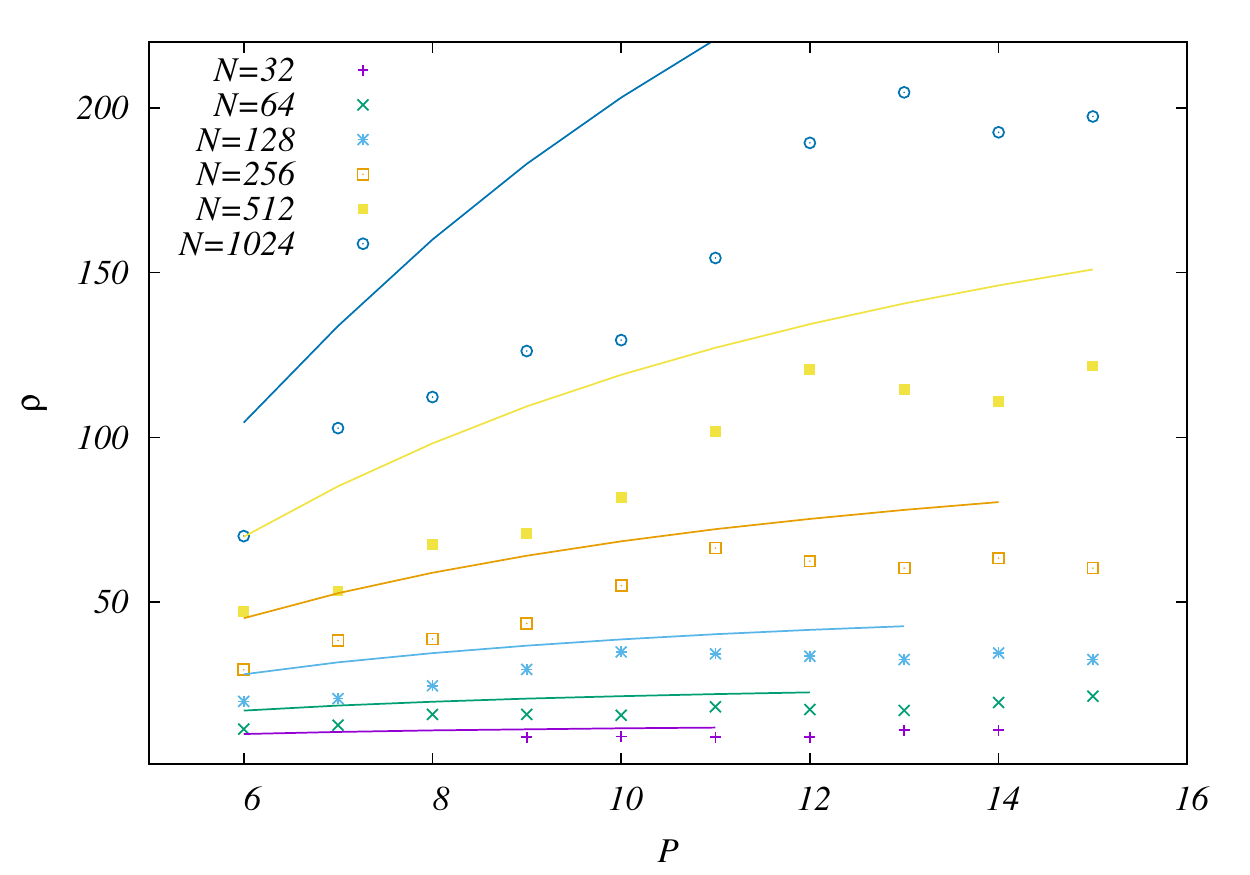}
\includegraphics[width= 0.49\textwidth]{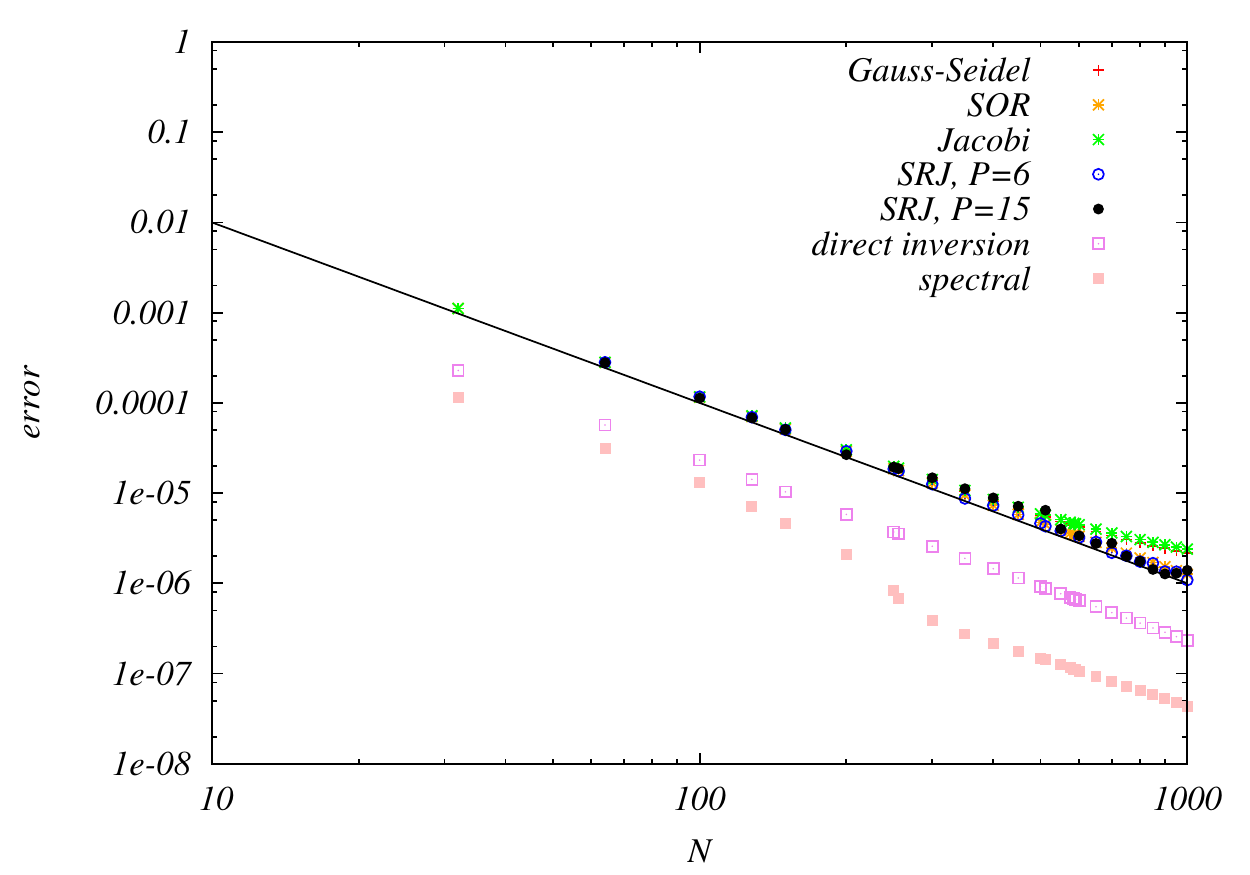}\\
\includegraphics[width= 0.47\textwidth]{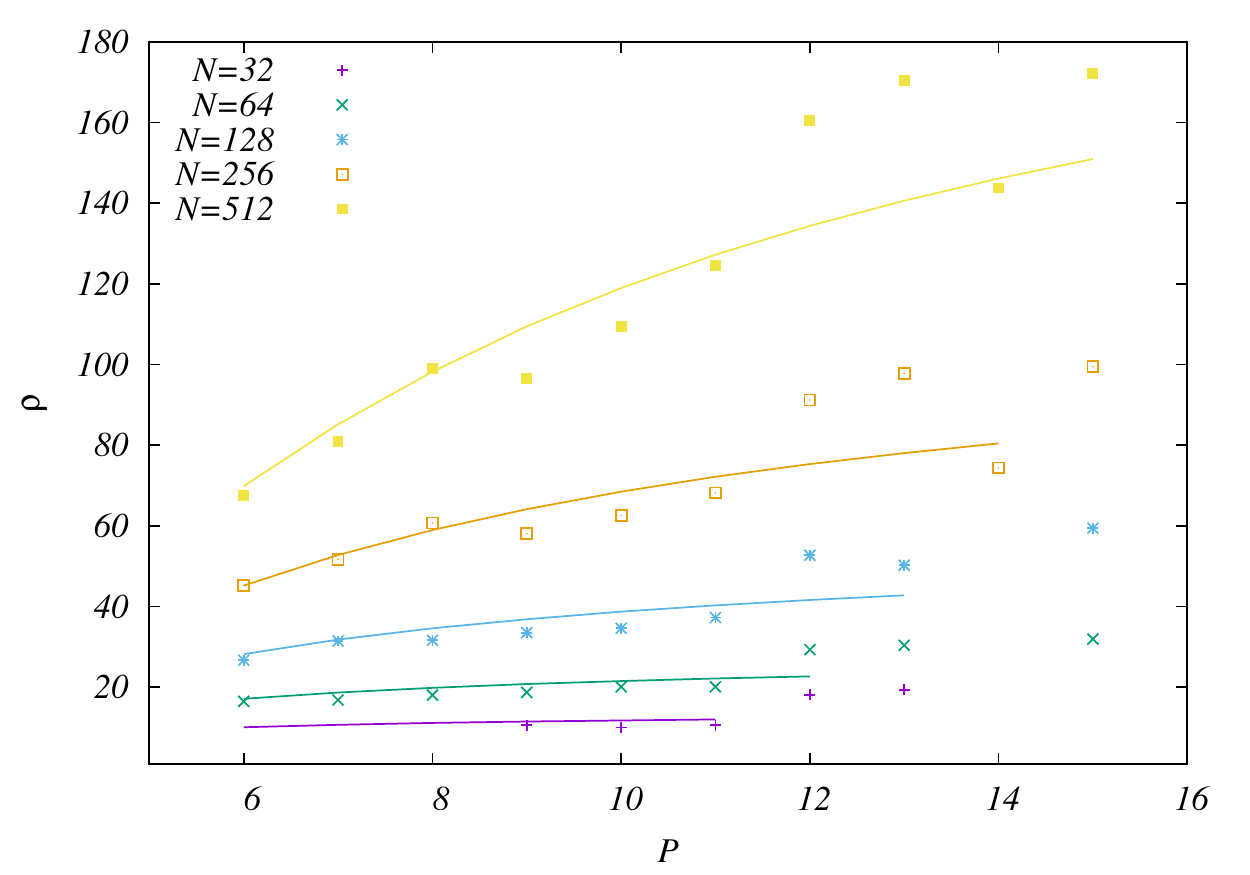}
\includegraphics[width= 0.49\textwidth]{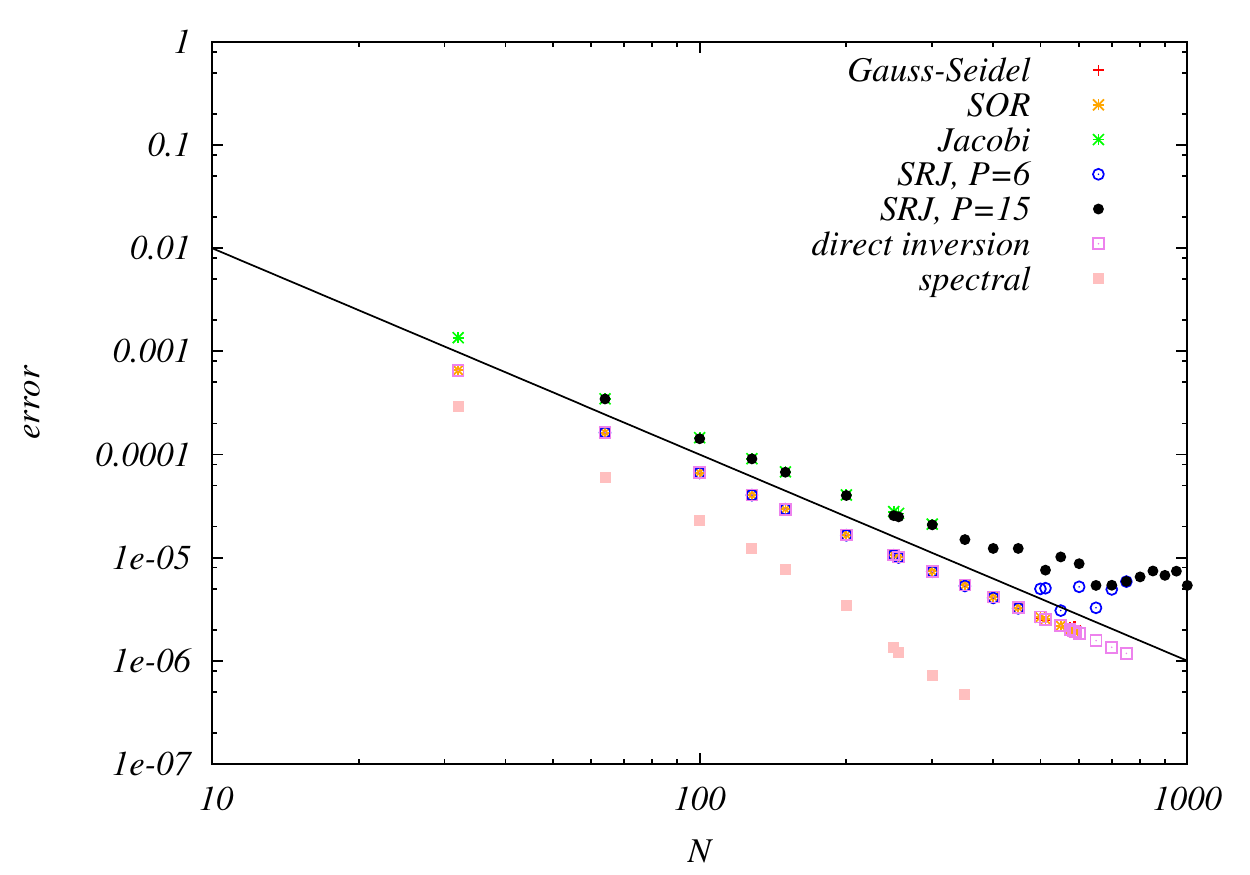}\\
\includegraphics[width= 0.47\textwidth]{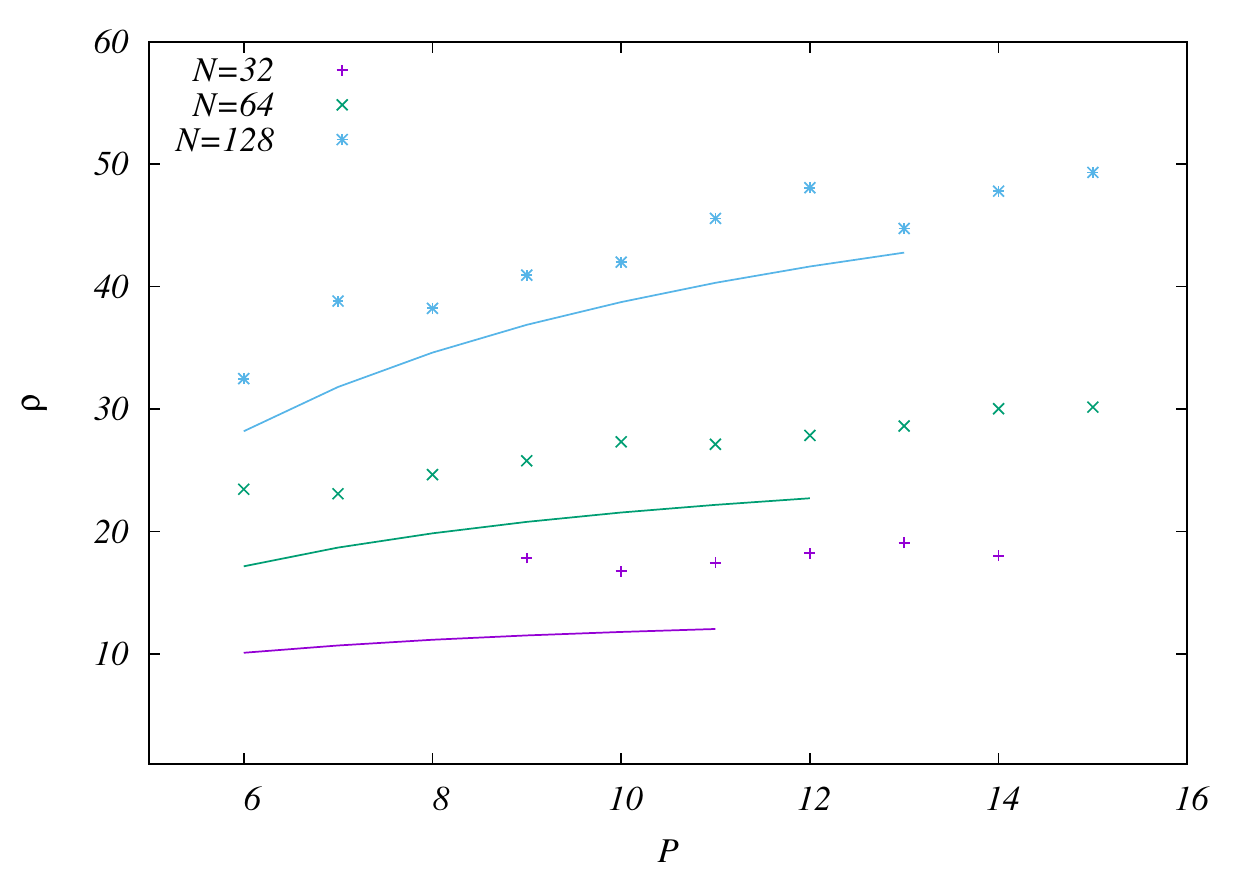}
\includegraphics[width= 0.49\textwidth]{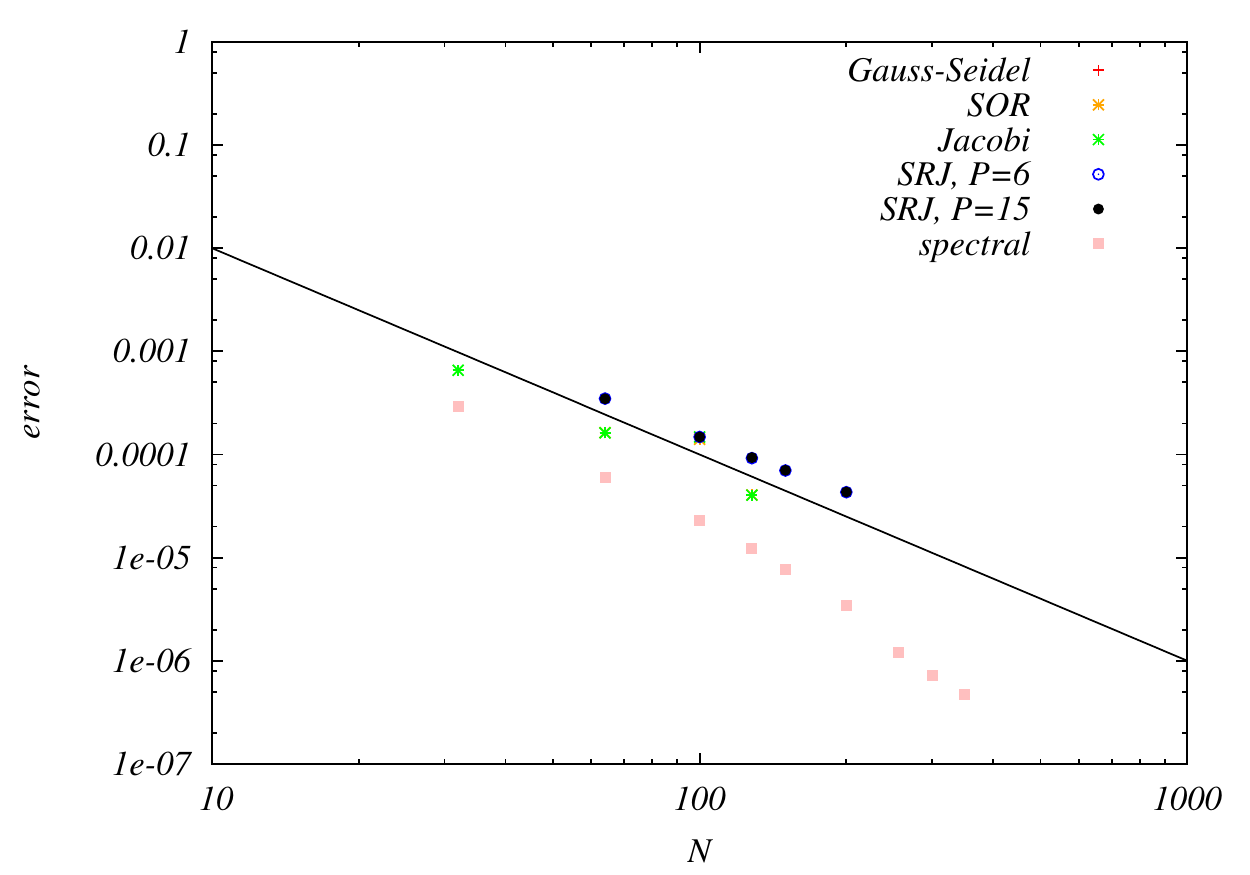}
\caption{
Detailed analysis of the solution of the Poisson equation for 1D, 2D and 3D calculations (upper, middle and lower panels respectively).
Left panels show the dependence of the numerically estimated value of $\rho$ on $P$, for several values of $N$ ranging from $32$ to $1024$.
Solid lines of the same color represent the theoretical estimate of $\rho$, for each case.
Right panels show the error in the solution as a function of $N$, computed as the $L_\infty$-norm of the difference between the numerical and 
the analytical solution. The solid black line represents $1/N^2$, which is an estimation of the expected finite difference error.
}
\label{fig:grav:rho}
\end{figure*}

Finally we have estimated numerically the value of $\rho$ for different SRJ weights, to be compared with the theoretical
predictions.  For this purpose we compute the ratio of number of iterations needed with Jacobi and a given SRJ method,
using the same tolerance and resolution, $N$. Left panels of Fig.\,\ref{fig:grav:rho} show the dependence of $\rho$ on
$P$ for several values of $N$, computed using the set of {\it ab initio calculations}. Regardless of the dimensionality,
in all calculations the numerical values of $\rho$ are close to the theoretical predictions (solid lines). In the 1D
test problems there is a tendency of the theoretical values to overestimate the numerically computed value. This trend
is exacerbated for large values of $N$ (namely, $N>512$). In 3D the situation is reversed, and the theoretical value of
$\rho$ falls below the numerical one. To explain these differences, we shall consider that the optimal weights depend on
the dimensionality of the problem, since $\kappa_m$ does also depend on dimensionality (see
Eq.\,\ref{eq:boundaries1D3D}).  As the optimization of the weights has been performed for the 2D case, it is not
surprising to find such discrepancies when using the same weights and the same value of $\rho$ in a problem with
different dimensionality. Indeed, we have repeated some of the 3D and 1D test problems employing the optimal SRJ
parameters corresponding to the effective number of points set according to Eq.\,(\ref{eq:effectiveN}), and found that
(i) the SRJ scheme runs is $\sim 20\%$ less iterations and, (ii) for this effective number of points, the theoretical
convergence performance index, computed with the dimensionality corrections mentioned below Eq.\,(\ref{eq:rho}), becomes
an upper bound for the numerical values of $\rho$. Adding to this arguments, we also note that the discretization of the
Laplacian operator in spherical coordinates may also change slightly the optimal weights. Finally, another factor that
explains the discrepancies is that the boundary conditions of this problem are mixed (as commented above), and the
optimal weights are computed for purely Neumann boundaries.

The right panels of Fig.\,\ref{fig:grav:rho} show the error in the solution as a function of $N$, computed as 
the $L_\infty$-norm of the difference between the numerical and the analytical solution. In all cases the error is 
dominated by the finite difference error associated to the discretization of the elliptic operator, which, 
for a second order method, is expected to be $\sim\mathcal{O}((\Delta x)^2) \sim \mathcal{O}(1/N^2)$. This is a 
symptom that our prescription for the tolerance is yielding converged numerical solutions, in iterative methods.
It also shows that the number of spectral grid points used is sufficient for such calculations.

We have also tried different discretizations of the equation and the boundary conditions, although not as systematically
as the presented case. In general, using discretizations which lead to non-diagonally dominant coefficient matrices,
increases the number of iterations to converge or, in some cases, they do not converge at all. The Jacobi method is the
most sensitive to this, while all other iterative methods (Gauss-Seidel, SOR, SRJ) seem less affected by this issue. As
an example, if just $u_0=u_1$ is used for the inner boundary condition (consistent with Eq.\,(\ref{eq:bc1})), the
Jacobi method needs about $5$ times more iterations in 1D, while all other iterative methods remain almost unaltered
(only SOR shows differences for $N\le64$).  This is an indication that the new method is not only faster than well-known
iterative methods but can also be more robust than some of them.

\subsection{Grad-Shafranov equation in spherical coordinates}

The Grad-Shafranov (GS) equation \citep{Grad:1970,Shafranov:1958} describes equilibrium solutions in ideal
magnetohydrodynamics for a two dimensional plasma. It is of interest
in studying the plasma in magnetic confinement fusion (e.g. Tokamaks),
the solar corona and neutron star magnetospheres, among others.

In spherical coordinates $(r,\theta,\varphi)$ the magnetic field of an axisymmetric
($\partial_\varphi=0$) plasma configuration can be expressed as
\begin{equation}
  \boldsymbol{B} (r,\theta) = \boldsymbol{\nabla} \times \boldsymbol{A}
= \frac{1}{r \sin{\theta}} \boldsymbol{\nabla} \Psi (r,\theta) \times \boldsymbol{\hat e_{\varphi}} 
  + \frac{F(r,\theta)}{r \sin{\theta}} \boldsymbol{\hat e_{\varphi}},
\end{equation}
where $\boldsymbol{A}$ is the vector potential and $\boldsymbol{\hat e_{\varphi}}$
is the unit vector in the $\varphi$ direction. The flux function,
$\Psi \equiv r \sin{\theta} A_\varphi$, is constant along magnetic field
lines and is a measure of the poloidal magnetic field strength. The
toroidal function, $F \equiv B_{\varphi} r \sin{\theta}$, is a measure
of the toroidal field strength. Using Ampere's law,
$\boldsymbol{J} = \boldsymbol{\nabla}\times \boldsymbol{B}$, being $\boldsymbol{J}$
the electric current, the flux function can be linked to the toroidal current as
\begin{equation}
\Delta^* \Psi \equiv \partial_{rr} \Psi  + \frac{1}{r^2} \partial_{\theta\theta}\Psi
- \frac{\cot{\theta}}{r^2} \partial_\theta \Psi = - J_\varphi r \sin{\theta},
\label{eq:ampere_phi}
\end{equation}
where $\Delta^*$ is the GS elliptic operator. For simplicity we consider here
the case in which the inertia of the fluid can be neglected (magnetically dominated).
In this case, if we impose force balance, $\boldsymbol{J} \times \boldsymbol{B} = 0$,
the toroidal function depends on the flux function, $F(\Psi)$. As a result Eq.\,(\ref{eq:ampere_phi})
leads to the GS equation
\begin{equation}
  \Delta^* \Psi = - F(\Psi) F'(\Psi). 
  \label{eq:GS}
\end{equation}
Not neglecting the inertia of the fluid leads to additional pressure terms, which are not
considered here. A popular choice for the toroidal function is $F(\Psi) = C \Psi $, being $C$ a constant.
In this case the GS equation results in
\begin{equation}
  \Delta^* \Psi + C^2 \Psi = 0,
  \label{eq:GS_linear}
\end{equation}
which is a suitable elliptic problem to be solved with SRJ methods.
Equation\,(\ref{eq:GS_linear}) resembles the Helmholtz differential equation in that it
contains a Laplacian-like operator and a linear term in $\Psi$. Therefore,
this test will show the ability of SRJ methods to handle more complicated
elliptic operators. In addition we use this test to demonstrate the
ability of iterative methods to handle boundary conditions imposed
at arbitrarily shaped boundaries.

We compute the solution of Eq.\,(\ref{eq:GS_linear}) for two sets
of boundary conditions, in the numerical domain
$r \in [1,10]$ and $\theta \in [0,\pi]$. In all cases we impose homogeneous
Dirichlet conditions at $\theta=0$ and $\theta=\pi$.
In {\it test A} we impose Dirichlet boundary conditions at $r=1$ and $r=10$
with $\Psi=\sin^2{\theta}/r$. In the case $C=0$, the solution for this test is
a dipolar field. As the value of $C$ is increased the solution results in a
twisted dipole.

In {\it test B} we solve the GS equation in part of the domain, the region defined by
\begin{eqnarray}
  &r < \left (4.5 \sin^2{\theta} + 2.5 \sin^2(2\theta)\right)
  \left ( 1 - 0.4 \cos(3\theta) + 0.3 \cos(5\theta) + 0.05 \sin(25 \theta) \right ) & \nonumber \\
  &\&& \nonumber\\
  &  (r \sin{\theta}-4)^2+(r\cos{\theta}-1.6)^2<1, &
  \label{eq:shape_b}
\end{eqnarray}
inside the aforementioned numerical domain.
The boundary of this region intersects the sphere $r=1$ at $\theta_1=0.3037$ and $\theta_2=2.8903$.
At $r=1$ we impose Dirichlet boundary conditions with $\Psi =\sin((\theta-\theta_1)/(\theta_2-\theta_1)\pi)^2$,
and homogeneous Dirichlet conditions at the remaining boundaries.
Imposing boundary conditions in arbitrarily shaped boundaries is straightforward
when using iterative methods such as SRJ; we set $\Psi=0$ everywhere
outside the region~(\ref{eq:shape_b}) and apply the SRJ iteration only inside
this region using a mask.

\begin{table}
  \caption{Number of iterations and computational time used by the SRJ method with $N=300$ and $P=14$
  to solve the GS equation, depending on the value of $C$. }
\begin{center}
\begin{tabular}{c c c | c c c}
\hline
{\it test A} &&& {\it test B} &&\\
$C$ & iterations & computational & $C$ & iterations & computational \\
 & & time [s] &  &  & time [s]  \\
\hline
 0 & 5350  & 8.55  & 0 & 3450  & 1.78  \\
 0.01 & 5350 & 8.55 & 0.01 & 3450 & 1.78  \\
 0.1 & 5350 & 8.58    & 0.1 & 3450 & 1.78  \\
 0.2 & 4660 & 7.48   & 0.5 & 3600 & 1.86  \\ 
 0.3 & 7750 & 12.43  & 1.0 & 3550 & 1.84  \\  
 0.31 & 9830 & 15.73    & 1.3 & 3410 & 1.78  \\
 0.32 & 13540 & 21.80    & 1.4 & 3660 & 2.41  \\
 0.33 & 21390 & 34.38    & 1.45 & 4980 & 2.57  \\
 0.34 & 49080 & 79.22    & 1.47 & 11620 & 6.03 \\
\hline
\end{tabular}
\end{center}
\label{tab:testab}
\end{table}

In both tests we use a second order discretization of the GS equation and a 
numerical resolution of $300\times 300$ equispaced grid points
covering the numerical domain. We solve the equations using the SRJ method
with weights corresponding to $N=300$ and $P=14$. In both tests we initialize 
$\Psi$ to zero in the whole domain.
Table~\ref{tab:testab} shows the number of iterations and computational time
to obtain a numerical solution with residual bellow $10^{-12}$,
depending on the value of $C$ used.

\begin{figure*}[p!]
\vspace{-2cm}
\includegraphics[width= 0.99\textwidth]{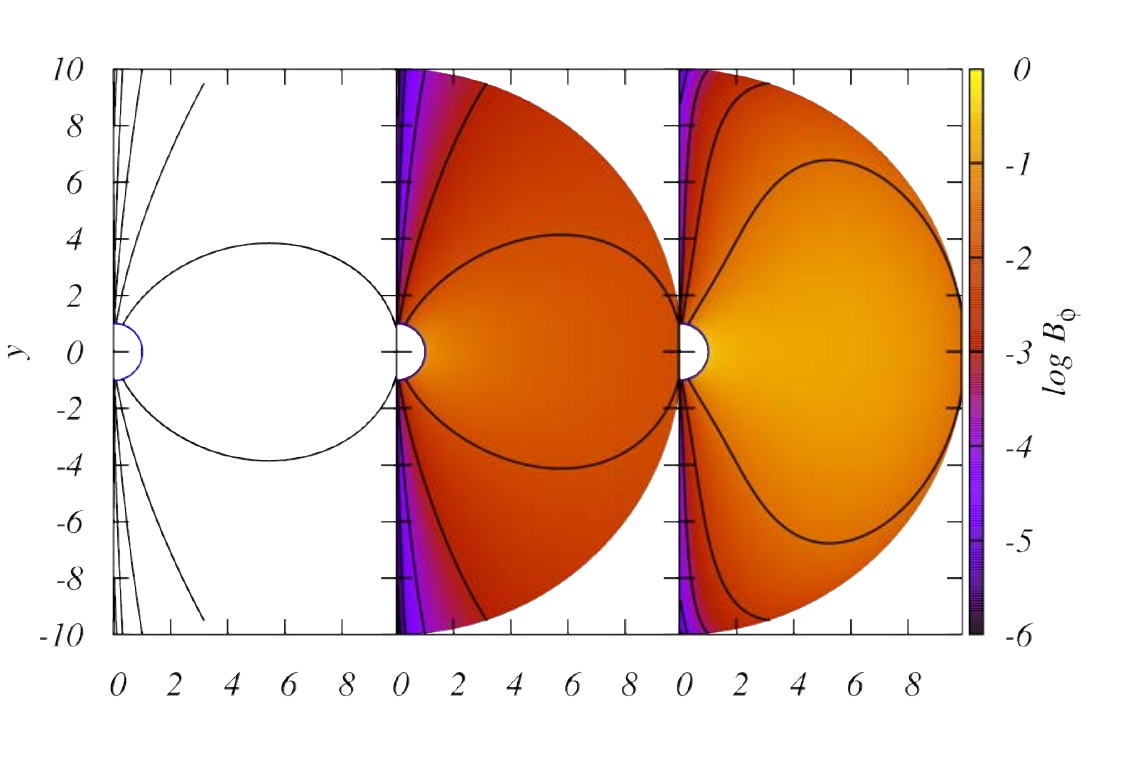}\\
\vspace{-1.5cm}\\
\includegraphics[width= 0.99\textwidth]{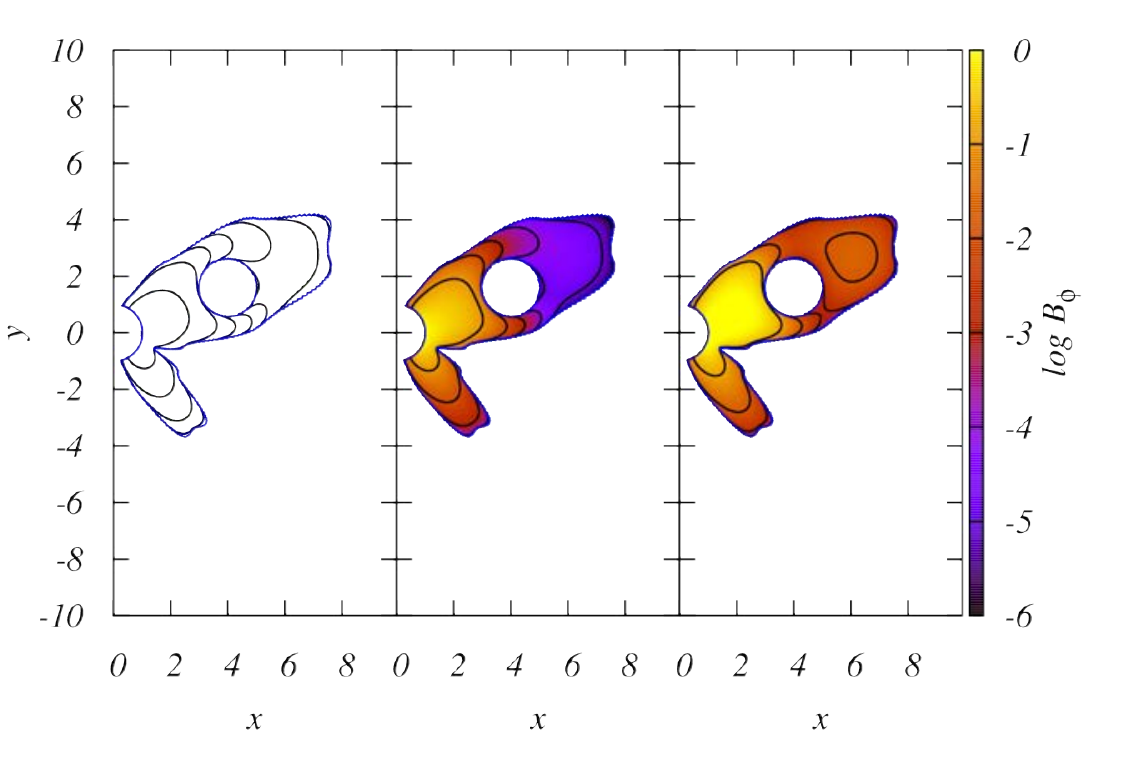}
\caption{Numerical solution of the GS equation for {\it test A} (upper panels)
  and {\it test B} (lower panels) for different values of the constant $C$.
  From left to right $C=0$, $0.1$ and $0.34$ (upper panels) and  
  $C=0$, $1.0$ and $1.47$ (lower panels).
  Isocontours of $\log \Psi$ (solid black lines),
  which coincide with magnetic field lines, are plotted in increments of $1$. Colors
  show $\log B_\varphi$. For convenience we plot the $(x,y)$ plane,
  being $x\equiv r \sin{\theta}$ and  $y\equiv r \cos{\theta}$. Blue line in lower panels
  show the boundary of the region in which the GS equation is solved.
}
\label{fig:grav:rho2}
\end{figure*}

Upper panels of Fig.\,(\ref{fig:grav:rho2}) show the results for {\it test A}, for three
different values of $C$. For the case $C=0$, the analytical solution is $\Psi=\sin^2{\theta}/r$.
In this case the maximum difference between the analytical and the numerical result, in absolute value, 
is $8.5 \times 10^{-5}$, which is consistent with the second order discretization ($9/N^2= 10^{-4}$). 
For $C=0.1$ a toroidal component appears, but the flux function, $\Psi$, remains essentially the same.
For higher values of $C$ there is a tendency of the magnetic field lines to become more inflated, to support
the increased magnetic tension due to the high magnetic field. In this regime the number of iterations 
needed in the SRJ method increases. We were not able to obtain solutions for values larger than $C=0.34$.
This is not a problem of the numerical method itself, since other methods (Jacobi, Gauss-Seidel, SOR) show similar 
behavior. The value $C\approx 0.35$ corresponds to an eigenvalue of the GS operator. For this case the matrix 
associated to the discretization of the GS equation is singular and hence it cannot be inverted. This is causing 
the convergence problems near this point.

Lower panels of Fig.\,(\ref{fig:grav:rho2}) show the results for {\it test B}, for three
different values of $C$. This case behaves qualitatively similar to {\it test A} but with more 
complicated geometry. The case $C=0$ shows no toroidal field, which appears as $C$ is increased.
For $C=1.0$ the flux function is still similar to that of the untwisted case, albeit slightly deformed.
For $C=1.47$, the maximum value that we were able to achieve, the topology of the field has changed, showing
a region of close magnetic field lines in the upper right part of the domain. As in {\it test A}, the difficulty to 
achieve convergence for larger values of $C$ is related to the presence of an eigenvalue of the GS operator. 
Note that the solution is everywhere smooth, and magnetic field lines (black lines) are tangent to the domain boundary
(blue curve) as expected (except for $r=1$ where non-zero Dirichlet boundary conditions are applied).

In general the SRJ method shows reasonable rates of convergence and computational time to solve the problem with 
high accuracy, despite of the complicated boundary conditions. This renders a method which can be used in real applications
of the GS equation with a good trade of excellent performance and ease of implementation.

\section{Conclusions and future work}
\label{sec:conclusions}

Building upon the results of YM14, we have devised a new method for
obtaining the optimal parameters for SRJ schemes applied to the
numerical solution of ePDEs.

We have shown that the new multilevel SRJ schemes keep improving the
convergence performance index of the scheme, which means that
increasing the value of $P$ we obtain ever larger acceleration factors
with respect to the Jacobi method. In the present paper we report
acceleration factors of a few hundreds and, in some cases, more than 1000 with respect to the Jacobi
method if a sufficiently large number of points per dimension
(namely, $N>16000$) and number of levels are 
considered. 

Mainly due to the fact that we have derived analytic solutions for
part of the unknowns, our new method reduces the stiffness of the
non-linear system of equations from which optimal parameters are
computed, allowing us to obtain new SRJ methods for up to $P=15$ and
arbitrarily large number of points per dimension $N$. 

From this number of levels, new problems arise, which hinder the computation of optimal coefficients at relatively low
number of discretization points. These problems are related to the fact that for large values of $P$ the solution to the
problem are very sensitive to tiny changes in the smaller wave numbers, and small numerical errors prevent the succesful
evaluation of the solution of even the simplified system of non-linear equations resulting from the algebraic
simplifications we have shown here. In order to tackle this problem, we are working in two new improvements: an alternative equivalent new
system, and alternative methods for the solution of the optimization problem (including genetic algorithms).

Currently, we have reached acceleration factors that have made that the SRJ methods become competitive (depending on the
dimensionality of the problem and its size) with, e.g., spectral methods for the solution of some ePDEs. In particular,
we have made the comparison in the case of an astrophysical problem that we are interested in for whose solution we were
using spectral methods. We find that for 1D Poisson-like problems, the fastest method of solution is the direct
inversion method implemented in {\tt LAPACK}. This happens because the LU decomposition of the matrix solver, where most
of the computational work is done, needs to be performed once, and the it can be stored for the rest of the
evolution. In 2D, the best performing method depends on whether our initial guess is close to the actual solution or far
off. In realistic applications, where ePDEs are coupled to systems of hyperbolic PDEs, the solution from a previous time
iteration does not change significantly over the course fo a single timestep. In such conditions, the {\tt LAPACK}
libraries are the best performing. However, spectral methods are advantageous if, in 2D, the initial values are far from
the actual solution of the problem. We further note that in realistic coupled systems, and for a relatively large number
of points per dimension ($N>500$), the SRJ methods are competitive with spectral ones. In 3D applications, we find that
the total computational cost of SRJ methods scales in 3D as $N^4$, i.e, as in the case of spectral
methods. Considering that (i) applying direct inversion methods to 3D problems is unfeasible because of memory
restrictions, and that (ii) SRJ methods can be parallelized straightforwardly (much more easily than, e.g., spectral or
multigrid methods), we foresee that they are a competitive alternative for the solution of elliptic problems in
supercomputing applications and in 3D. We are studying improvements in the method from this point of view. Finally, we
outline that the easy implementation of complex boundary conditions in SRJ methods is also an advantage with respect to
other existing alternatives.  

\section*{Acknowledgements}
We thank P.\,Mimica for bringing the paper YM14 to our attention. We acknowledge the support from the European Research
Council (Starting Independent Researcher Grant CAMAP-259276), and the partial support of grants AYA2013-40979-P,
PROMETEO-II-2014-069 and SAF2013-49284-EXP. We thankfully acknowledge the computer resources, technical expertise and
assistance provided by the Servei de Inform\`atica of the University of Valencia.
%
\appendix

\section{Compendium of parameters of optimal SRJ schemes}
\label{sec:P2_P13}

\begin{table}
\caption{Parameters for optimized $P=6$ SRJ schemes for various values of $N$.}
\vspace{-0.cm}
\begin{center}
\scriptsize
\begin{tabular}{c c c c}
\hline
N & Optimal Scheme Parameters & $\rho$  \\
\hline
\multirow{2}{*}{100} & $\boldsymbol{\omega}=\{3177.91,734.924,110.636,15.7187,2.41695,0.614642\}$ & \multirow{2}{*}{23.75} \\
& $\boldsymbol{\beta}=\{0.00367219,0.00906128,0.0276578,0.0863039,0.261008,0.612297\}$ & \\
\hline
\multirow{2}{*}{150} & $\boldsymbol{\omega}=\{6903.64,1362.54,176.155,21.6587,2.85787,0.629548\}$ & \multirow{2}{*}{31.57} \\
& $\boldsymbol{\beta}=\{0.0023493,0.00645108,0.0217415,0.0745986,0.249258,0.645601\}$ & \\
\hline
\multirow{2}{*}{200} & $\boldsymbol{\omega}=\{11953.5,2106.17,244.75,27.1876,3.22218,0.640444\}$ & \multirow{2}{*}{38.38} \\
& $\boldsymbol{\beta}=\{0.00169767,0.00503682,0.018218,0.0669113,0.240064,0.668072\}$ & \\
\hline
\multirow{2}{*}{250} & $\boldsymbol{\omega}=\{18283.1,2948.76,315.677,32.4263,3.53814,0.649052\}$ & \multirow{2}{*}{44.49} \\
\hline
\end{tabular}
\label{table:weiP6_1}
\end{center}
\end{table}

\begin{table}

\caption{Parameters for optimized $P=7$ SRJ schemes for various values of $N$.}

\scriptsize
\begin{center}

\begin{tabular}{c c c c}
\hline
N & Optimal Scheme Parameters & $\rho$  \\
\hline
\multirow{2}{*}{100} & $\boldsymbol{\omega}=\{3392.32,1082.34,224.848,42.594,8.09726,1.71942,0.585658\}$ & \multirow{2}{*}{26.38} \\
& $\boldsymbol{\beta}=\{0.00342091,0.0068221,0.0169684,0.0438593,0.113205,0.27876,0.536965\}$ & \\
\hline
\multirow{2}{*}{150} & $\boldsymbol{\omega}=\{7440.09,2083.58,378.528,63.3295,10.6285,1.97391,0.59717\}$ & \multirow{2}{*}{35.79} \\
& $\boldsymbol{\beta}=\{0.00222099,0.00482464,0.0130515,0.0364807,0.101845,0.272802,0.568775\}$ & \\
\hline
\multirow{2}{*}{200} & $\boldsymbol{\omega}=\{12975.6,3307.79,547.016,83.8798,12.8988,2.18024,0.605662\}$ & \multirow{2}{*}{44.19} \\
& $\boldsymbol{\beta}=\{0.00162297,0.00375204,0.0107739,0.0318451,0.0940257,0.267403,0.590577\}$ & \\
\hline
\multirow{2}{*}{250} & $\boldsymbol{\omega}=\{19962.8,4727.89,727.329,104.286,14.9921,2.35673,0.612421\}$ & \multirow{2}{*}{51.87} \\
\hline
\end{tabular}

\label{table:weiP7_1}

\end{center}
\end{table}

\begin{table}

\caption{Parameters for optimized $P=8$ SRJ schemes for various values of $N$.}
\scriptsize 
\begin{center}

\begin{tabular}{c c c c}
\hline
N & Optimal Scheme Parameters & $\rho$  \\
\hline
\multirow{2}{*}{100} & $\boldsymbol{\omega}=\{3537.92,1419.79,377.85,89.614,20.816,4.96187,1.3481,0.566216\}$ & \multirow{2}{*}{28.46} \\
& $\boldsymbol{\beta}=\{0.00319391,0.00549334,0.0117303,0.0263687,0.0597575,0.134042,0.283454,0.47596\}$ & \\
\hline
\multirow{2}{*}{150} & $\boldsymbol{\omega}=\{7806.72,2810.62,663.524,140.897,29.3973,6.27138,1.51276,0.575281\}$ & \multirow{2}{*}{39.17} \\
& $\boldsymbol{\beta}=\{0.00209747,0.00386747,0.00888556,0.0213663,0.0517106,0.124107,0.282297,0.505669\}$ & \\
\hline
\multirow{2}{*}{200} & $\boldsymbol{\omega}=\{13677.4,4551.55,987.73,194.11,37.56,7.41325,1.64438,0.582004\}$ & \multirow{2}{*}{48.90} \\
& $\boldsymbol{\beta}=\{0.00154611,0.00299977,0.00725956,0.018313,0.0464484,0.116968,0.280231,0.526234\}$ & \\
\hline
\multirow{2}{*}{250} & $\boldsymbol{\omega}=\{21120.6,6606.74,1343.7,248.789,45.4233,8.44441,1.75577,0.587377\}$ & \multirow{2}{*}{57.94} \\
\hline
\end{tabular}

\label{table:weiP8_2}

\end{center}
\end{table}

\begin{table}

\caption{Parameters for optimized $P=9$ SRJ schemes for various values of $N$.}
\scriptsize 
\begin{center}

\begin{tabular}{c c c c}
\hline
N & Optimal Scheme Parameters & $\rho$  \\
\hline
\multirow{2}{*}{100} & $\boldsymbol{\omega}=\{3640.86,1729.24,559.604,158.917,43.5786,11.9645,3.41756,1.12661,0.552629\}$ & \multirow{2}{*}{30.12} \\
& $\boldsymbol{\beta}=\{0.00298948,0.00462251,0.0087721,0.0176949,0.0362139,0.0740779,0.149046,0.280346,0.426238\}$ & \\
\hline
\multirow{2}{*}{150} & $\boldsymbol{\omega}=\{8067.04,3497.09,1016.04,260.97,65.0344,16.2237,4.19023,1.24183,0.559916\}$ & \multirow{2}{*}{41.92} \\
& $\boldsymbol{\beta}=\{0.00198105,0.00324426,0.00656828,0.0140648,0.0304787,0.0660202,0.141034,0.282856,0.453753\}$ & \\
\hline
\multirow{2}{*}{200} & $\boldsymbol{\omega}=\{14177.3,5750.32,1548.34,370.696,86.3811,20.1478,4.84929,1.33292,0.565337\}$ & \multirow{2}{*}{52.78} \\
& $\boldsymbol{\beta}=\{0.0014704,0.00251182,0.00532492,0.0118951,0.0268474,0.0605723,0.135008,0.283441,0.472929\}$ & \\
\hline
\multirow{2}{*}{250} & $\boldsymbol{\omega}=\{21947.5,8446.56,2144.68,486.462,107.643,23.8396,5.43484,1.4094,0.56968\}$ & \multirow{2}{*}{62.97} \\
\hline
\end{tabular}

\label{table:weiP9_2}

\end{center}
\end{table}

\begin{table}
\caption{Parameters for optimized $P=10$ SRJ schemes for various values of $N$.}
\scriptsize 
\begin{center}
\begin{tabular}{c c c c}
\hline
N & Optimal Scheme Parameters & $\rho$  \\
\hline
\multirow{2}{*}{100} & $\boldsymbol{\omega}=\{3716.11,2003.92,758.907,249.869,78.5973,24.4799,7.71415,2.55553,0.983449,0.542795\}$ & \multirow{2}{*}{31.49} \\
& $\boldsymbol{\beta}=\{0.0028058,0.00400978,0.00692995,0.0128148,0.0241718,0.0457763,0.0862794,0.159105,0.272849,0.385258\}$ & \\
\hline
\multirow{2}{*}{150} & $\boldsymbol{\omega}=\{8257.94,4120.6,1415.79,425.028,122.54,35.0395,10.1187,3.05606,1.06871,0.548763\}$ & \multirow{2}{*}{44.20} \\
& $\boldsymbol{\beta}=\{0.00187294,0.00280821,0.00514163,0.0100348,0.019919,0.0396644,0.0786806,0.153149,0.278034,0.410696\}$ & \\
\hline
\multirow{2}{*}{200} & $\boldsymbol{\omega}=\{14544.8,6857.12,2199.03,618.829,167.848,45.1996,12.278,3.47555,1.13556,0.553212\}$ & \multirow{2}{*}{56.02} \\
& $\boldsymbol{\beta}=\{0.00139791,0.00217152,0.00414295,0.00839998,0.0172885,0.0356771,0.0733849,0.148408,0.280624,0.428506\}$ & \\
\hline
\multirow{2}{*}{250} & $\boldsymbol{\omega}=\{22556.4,10166.9,3091.02,827.681,214.189,55.0715,14.2712,3.84341,1.19131,0.556783\}$ & \multirow{2}{*}{67.18} \\
\hline
\end{tabular}

\label{table:weiP10_2}

\end{center}
\end{table}

\begin{table}

\caption{Parameters for optimized $P=11$ SRJ schemes for various values of $N$.}
\scriptsize 
\begin{center}

\begin{tabular}{c c c c}
\hline
N & Optimal Scheme Parameters & $\rho$  \\
\hline
\multirow{2}{*}{100} & $\boldsymbol{\omega}=\{3772.69,2243.38,965.764,359.834,127.002,44.0668,15.2981,5.41117,2.02855,0.885269,0.535461\}$ & \multirow{2}{*}{32.64} \\
& $\boldsymbol{\beta}=\{0.00264078,0.0035551,0.00569791,0.00980872,0.0173037,0.0307448,0.054615,0.0962556,0.165235,0.263076,0.351068\}$ & \\
\hline
\multirow{2}{*}{150} & $\boldsymbol{\omega}=\{8401.79,4674.24,1842.62,630.247,205.246,65.8794,21.1553,6.90331,2.37588,0.951043,0.540431\}$ & \multirow{2}{*}{46.12} \\
& $\boldsymbol{\beta}=\{0.00177338,0.00248623,0.0041962,0.00758878,0.0140228,0.0260629,0.048432,0.0894173,0.161257,0.27015,0.374614\}$ & \\
\hline
\multirow{2}{*}{200} & $\boldsymbol{\omega}=\{14822.2,7852.9,2907.6,936.628,288.313,87.62,26.639,8.21547,2.66282,1.00226,0.544141\}$ & \multirow{2}{*}{58.75} \\
& $\boldsymbol{\beta}=\{0.00132965,0.00192098,0.00336441,0.0063002,0.0120299,0.023085,0.0442925,0.0845072,0.157812,0.274205,0.391154\}$ & \\
\hline
\multirow{2}{*}{250} & $\boldsymbol{\omega}=\{23016.7,11730.6,4137.15,1272.65,375.138,109.305,31.8602,9.40806,2.9118,1.04478,0.547122\}$ & \multirow{2}{*}{70.76} \\
\hline
\end{tabular}

\label{table:weiP11_2}

\end{center}
\end{table}

\begin{table}

\caption{Parameters for optimized $P=12$ SRJ schemes for various values of $N$.}
\tiny 
\begin{center}

\begin{tabular}{c c c c}
\hline
N & Optimal Scheme Parameters & $\rho$  \\
\hline
\multirow{2}{*}{100} & $\boldsymbol{\omega}=\{3816.09,2449.48,1171.54,484.552,188.591,71.7552,27.1524,10.3398,4.03462,1.67921,0.814946,0.530037\}$ & \multirow{2}{*}{33.60} \\
& $\boldsymbol{\beta}=\{0.00249324,0.0032053,0.00483143,0.0078347,0.0130695,0.0220179,0.0371777,0.0626105,0.104333,0.168776,0.25022,0.32343\}$ & \\
\hline
\multirow{2}{*}{150} & $\boldsymbol{\omega}=\{8512.57,5158.9,2278.59,870.502,314.372,111.324,39.2426,13.9067,5.03399,1.93522,0.867388,0.534161\}$ & \multirow{2}{*}{47.75} \\
& $\boldsymbol{\beta}=\{0.00168233,0.00223889,0.00353484,0.00599785,0.0104405,0.0183241,0.0322191,0.0565309,0.0983065,0.166362,0.259643,0.344719\}$ & \\
\hline
\multirow{2}{*}{200} & $\boldsymbol{\omega}=\{15036.2,8734.56,3644.21,1316.62,451.186,151.917,50.9495,17.1663,5.89563,2.1434,0.907955,0.537264\}$ & \multirow{2}{*}{61.09} \\
& $\boldsymbol{\beta}=\{0.0012661,0.00172882,0.0028221,0.00494455,0.00886996,0.0160265,0.0290014,0.0523866,0.0938883,0.164052,0.265144,0.35987\}$ & \\
\hline
\multirow{2}{*}{250} & $\boldsymbol{\omega}=\{23372.3,13127.2,5239.13,1813.2,596.796,193.289,62.3814,20.2163,6.66795,2.32215,0.941482,0.539767\}$ & \multirow{2}{*}{73.83} \\
\hline
\end{tabular}

\label{table:weiP12_2}

\end{center}
\end{table}

\begin{table}

\caption{Parameters for optimized $P=13$ SRJ schemes for various values of $N$.}
\tiny 
\begin{center}

\begin{tabular}{c c c c}
\hline
N & Optimal Scheme Parameters & $\rho$  \\
\hline
\multirow{2}{*}{100} & $\boldsymbol{\omega}=\{3850.42,2628.04,1372.85,621.295,263.269,108.54,44.3191,18.1012,7.47022,3.1732,1.46178,0.770791,0.52618\}$ & \multirow{2}{*}{34.43} \\
& $\boldsymbol{\beta}=\{0.00235853,0.00292322,0.00418701,0.00645168,0.0102578,0.0165109,0.0266858,0.0431288,0.0694093,0.110458,0.161845,0.243877,0.301907\}$ & \\
\hline
\multirow{2}{*}{150} & $\boldsymbol{\omega}=\{8600.11,5583.,2713.17,1140.26,450.396,173.646,66.4069,25.4013,9.80185,3.88069,1.65106,0.810599,0.529505\}$ & \multirow{2}{*}{49.15} \\
& $\boldsymbol{\beta}=\{0.0015982,0.00204078,0.00304774,0.00489728,0.00810045,0.0135406,0.0227101,0.0380865,0.06365,0.105277,0.164788,0.25184,0.320423\}$ & \\
\hline
\multirow{2}{*}{200} & $\boldsymbol{\omega}=\{15205.3,9511.83,4388.27,1750.36,658.133,242.067,88.403,32.2912,11.8866,4.4795,1.80488,0.84212,0.532062\}$ & \multirow{2}{*}{63.10} \\
& $\boldsymbol{\beta}=\{0.00120674,0.00157554,0.00242509,0.00401446,0.00682843,0.0117242,0.0201874,0.0347575,0.059662,0.101442,0.165082,0.257141,0.333953\}$ & \\
\hline
\multirow{2}{*}{250} & $\boldsymbol{\omega}=\{23653.2,14365.9,6363.85,2438.1,882.631,313.079,110.345,38.8978,13.808,5.00982,1.93667,0.868438,0.534148\}$ & \multirow{2}{*}{76.48} \\
\hline
\end{tabular}

\label{table:weiP13_2}

\end{center}
\end{table}

\newpage
\bibliographystyle{elsarticle-num} 
\bibliography{biblio}

\end{document}